\pdfoutput=1
\RequirePackage{ifpdf}
\ifpdf 
\documentclass[pdftex]{sigma}
\else
\documentclass{sigma}
\fi

\usepackage{mathrsfs}

\numberwithin{equation}{section}

\newtheorem{Theorem}{Theorem}[section]
\newtheorem*{Theorem*}{Theorem}
\newtheorem{Corollary}[Theorem]{Corollary}
\newtheorem{Lemma}[Theorem]{Lemma}
\newtheorem{Proposition}[Theorem]{Proposition}

\theoremstyle{definition}
\newtheorem{Definition}[Theorem]{Definition}
\newtheorem{Remark}[Theorem]{Remark}
\newtheorem{Question}[Theorem]{Question}

\usepackage{tikz-cd}
\usepackage{enumitem}

\newcommand{\Z}{\mathbb{Z}} 
\newcommand{\R}{\mathbb{R}} 
\newcommand{\C}{\mathbb{C}} 
\newcommand{\DD}{\mathbb{D}} 

\newcommand{\M}{\mathbb{M}}
\newcommand{\Ms}{\M_{\mathrm{s}}}

\newcommand{\ev}{\mathrm{ev}}

\DeclareMathOperator{\Op}{Op}

\DeclareMathOperator{\supp}{supp}

\newcommand{\Id}{\operatorname{Id}}

\newcommand{\Om}{{\underline{\Omega}}}
\newcommand{\ntc}{\mathrm{ntc}}

\newcommand{\D}{\mathscr{D}}

\newcommand{\loc}{\mathrm{loc}}
\newcommand{\cl}{\mathrm{cl}}

\newcommand{\act}{\triangleleft}

\DeclareMathOperator{\Lip}{Lip}

\DeclareMathOperator{\spn}{span}

\DeclareMathOperator{\Diff}{Diff}
\DeclareMathOperator{\Riem}{Riem}
\DeclareMathOperator{\Bel}{Bel}

\renewcommand{\Re}{\operatorname{Re}}
\renewcommand{\Im}{\operatorname{Im}}

\newcommand \Cm { \mathbb{C}}
\newcommand \Dm { \mathbb{D}}
\newcommand \Rm { \mathbb{R}}
\newcommand \Sm { \mathbb{S}}
\newcommand \Zm { \mathbb{Z}}
\newcommand \zbar { \bar{z}}

\newcommand \mubar {{\bar{\mu}}}

\renewcommand \ss {\mathfrak s}
\newcommand \tu {\tilde{u}}

\begin{document}

\allowdisplaybreaks

\newcommand{\arXivNumber}{2403.05985}

\renewcommand{\PaperNumber}{045}

\FirstPageHeading

\ShortArticleName{Local and Global Blow Downs of Transport Twistor Space}

\ArticleName{Local and Global Blow Downs\\ of Transport Twistor Space}

\Author{Jan BOHR~$^{\rm a}$, Fran\c{c}ois MONARD~$^{\rm b}$ and Gabriel P. PATERNAIN~$^{\rm c}$}

\AuthorNameForHeading{J.~Bohr, F.~Monard and G.P.~Paternain}

\Address{$^{\rm a)}$~Mathematical Institute, University of Bonn, Endenicher Allee 60, 53115 Bonn, Germany}
\EmailD{\mail{bohr@math.uni-bonn.de}}

\Address{$^{\rm b)}$~Department of Mathematics, University of California, Santa Cruz, CA 95064, USA}
\EmailD{\mail{fmonard@ucsc.edu}}

\Address{$^{\rm c)}$~Department of Mathematics, University of Washington, Seattle, WA 98195, USA}
\EmailD{\mail{gpp24@uw.edu}}

\ArticleDates{Received April 02, 2025, in final form April 20, 2026; Published online May 05, 2026}

\Abstract{Transport twistor spaces are degenerate complex $2$-dimensional manifolds $Z$ that complexify transport problems on Riemannian surfaces, appearing, e.g., in geometric inverse problems. This article considers maps $\beta\colon Z\to \C^2$ with a {\it holomorphic blow-down structure} that resolve the degeneracy of the complex structure and allow to gain insight into the complex geometry of $Z$. The main theorems provide global $\beta$-maps for constant curvature metrics and their perturbations and local $\beta$-maps for arbitrary metrics, thereby proving a~version of the classical {\it Newlander--Nirenberg theorem} for degenerate complex structures.}


\Classification{32G05; 53C28; 35R30}

\section{Introduction}

Transport twistor spaces are degenerate complex $2$-dimensional manifolds that can be associated to any oriented Riemannian surface and whose complex geometry is closely linked to the geodesic flow of the surface.
In the last few years, these twistor spaces have become a useful device to organise and reinterpret various questions arising in geometric inverse problems and dynamical systems \cite{BLP24,BoPa23}.
Vice versa, they give rise to an array of intriguing complex geometric questions that are amenable to tools of the aforementioned areas.

This article is an instance of the latter---using results from X-ray tomography and microlocal analysis, we address some fundamental issues pertaining to the complex structure of transport twistor space. In particular, we
show that locally the degeneracy of the complex structure can always be resolved. Using this, we derive an identity principle for biholomorphisms between twistor spaces, which features as an important ingredient in our follow-up work on biholomorphism rigidity \cite{BMP24b}.

Our strategy consists in finding appropriate blow-down maps $\beta$ from twistor space into~$\C^2$, first explicitly in constant curvature, then more systematically for simple metrics. We prove that the blow down structure persists for small perturbations of the constant curvature models and derive our local results by comparing small geodesic balls with nearly Euclidean disks.

\subsection{Transport twistor spaces}

 Let $(M,g)$ be an oriented Riemannian surface, possibly with non-empty boundary $\partial M$. We denote the unit tangent bundle with $SM=\{(x,v)\in TM\mid g(v,v)=1\}$ and the geodesic vector field by $X$. Recall that the latter is defined by
 \[
 	X(x,v)=\frac{\rm d}{{\rm d}t}\Big|_{t=0} (\gamma_{x,v}(t),\dot \gamma_{(x,v)}(t)) \in T_{(x,v)}SM,\qquad (x,v)\in SM,
 \]
 where $\gamma_{x,v}(t)$ is the unique geodesic 
 with initial conditions $\gamma(0)=x$ and $\dot \gamma(0)=v$.
 The {\it transport twistor space} of $(M,g)$ is the unit disk bundle
$Z=\{(x,v)\in TM\mid g(v,v)\le 1\}$,
equipped with a certain rank $2$ involutive structure $\D$ (that is, a subbundle of $T_\C Z=TZ\otimes \C$ that is closed under taking commutators, see also~\cite{Tre92}) that satisfies
\begin{gather}
	\D\cap \bar \D = 0\quad \text{on} \  Z\backslash SM \qquad \text{and} \qquad \D\cap \bar \D = \C X\quad \text{on}\  SM, \label{degeneracy}\\
	\D\cap T Z_x = T^{0,1}Z_x\qquad \text{for all} \  x\in M. \label{fibres}
\end{gather}
Here $Z_x\subset Z$ is the fibre over $x$ and $T^{0,1}Z_x$ is the anti-holomorphic tangent bundle with respect to the complex structure that $g$ and the orientation induce on $T_xM$. The first condition in \eqref{degeneracy} implies that $\D|_{Z^\circ} = \ker(J+i)$ for an almost complex structure $J \in C^\infty(Z^\circ, \operatorname{End}(TZ^\circ))$, and involutivity of $\D$ is equivalent to $J$ being integrable.\footnote{If one asks $J$ to be compatible with the given orientation of $Z\subset TM$, then $\D$ is uniquely characterised by properties \eqref{degeneracy} and \eqref{fibres}.}
By the second condition in~\eqref{degeneracy}, $\D$ ceases to be of complex type at points of $SM$ and in this sense $Z$ may be thought of as a degenerate complex surface.

There are several ways to define this involutive structure $\D$---we refer to Section~\ref{s_conformallyeuclidean} below for a definition of $\D$ in isothermal coordinates, which is all we need for the present purpose, and to
to~\cite{BoPa23} for a coordinate free approach. For now, we confine ourselves to a description for Euclidean domains $M\subset \bigl(\C,|{\rm d}z|^2\bigr)$, where $Z=\{(z,\mu)\in M\times \C\mid |\mu|\le 1\}$ and $\D$ is spanned by the complex vector fields
$\Xi = \partial_{\bar z} + \mu^2 \partial_z$ and $\partial_{\bar \mu}$.
Here properties \eqref{degeneracy} and \eqref{fibres} are easily verified, noting that for $z=x_1+{\rm i}x_2$ and $\mu={\rm e}^{{\rm i}\theta}$ ($\theta \in \R$) we have
\[
	\bar \mu \Xi = (\cos \theta) \partial_{x_1} + (\sin \theta) \partial_{x_2}= X.
\]

\subsubsection{Twistor correspondences}\label{s_correspondences}
In the articles \cite{BLP24,BoPa23}, several correspondence principles have been set up that link the geodesic flow (in the form of the transport equation $Xu=f$) with the complex geometry of the transport twistor space:
	
	\begin{center}
\begin{tabular}{|l||l|}
	\hline
	&\\[-1em]
	\multicolumn{1}{|c||}{\textbf{Geodesic flow}} &\multicolumn{1}{c|}{\textbf{Twistor space}}\\
	\hline
	& \\[-1em]
		Invariant functions (= $\ker X|_{C^\infty(SM)}$)& Holomorphic functions on $Z$\\
		that are `fibrewise holomorphic' &\\
		&\\[-.8em]
		Invariant distributions (= $\ker X|_{\mathcal D'(SM)}$) & Holomorphic functions on $Z^\circ$\\that are `fibrewise holomorphic'& with polynomial growth at $SM$\\
	&\\[-.8em]
		Connections and matrix potentials on $M$ & Holomorphic vector bundles on $Z$\\
		(= $0$-th order perturbations of $X$) &\\
		\hline
	\end{tabular}
	\end{center}

\noindent	Here holomorphic functions on $Z$ are functions $f\colon Z\rightarrow \C$ that are smooth up to the boundary and holomorphic in the interior (equivalently, ${\rm d}f|_\mathscr D = 0$ on $Z$). Similarly, holomorphic vector bundles on $Z$ are smooth up to the boundary in a suitable sense.

 These correspondence principles have, at least in the view of the authors, become tremendously helpful in organising and reinterpreting various developments in geometric inverse problems and dynamical systems in two dimensions. For example, Pestov--Uhlmann's proof of boundary rigidity for simple surfaces \cite{PeUh05} admits a natural formulation in terms of the algebra $\mathcal A(Z)$ of holomorphic functions on $Z$ and a Cartan extension operator $\mathcal A(M)\to \mathcal A(Z)$; moreover, the influential notion of {\it holomorphic integrating factors} that was introduced by Salo--Uhlmann~\cite{SaUh11}
is underpinned by an Oka--Grauert principle for holomorphic line bundles on $Z$---for more examples we refer to \cite{BLP24,BoPa23}.

 It seems intriguing to further explore this link and possibly
use the power of complex geometry to gain insights into the geodesic flow. Doing this requires both a better understanding of the degeneracy of the complex structure and of the complex geometry of twistor space---the present article aims at making a start on both objectives.

\subsubsection{Projective twistor spaces} Quotienting $Z$ by the antipodal action $(x,v)\mapsto (x,-v)$, one obtains another degenerate complex surface that depends only on the projective class of the geodesic flow. This surface---that we shall refer to as {\it projective twistor space} and denote $Z_\mathbb P$---is in fact the more classical object and has been studied by several other authors \cite{Dub83,Hit80,Leb80,LeMa02,LM10,Met21,MePa20,OR85}. The quotient map~${Z\rightarrow Z_\mathbb P}$
is a branched $2:1$ covering and thus locally near a boundary point of $Z_\mathbb P$ the complex structure degenerates in the same way as $\D$. This degeneracy has been explicitly addressed in the work of LeBrun and Mason \cite{LeMa02,LM10}, who---in the case of Zoll surfaces---construct a holomorphic map $\beta_{\mathrm{LM}}\colon Z_\mathbb P\rightarrow \C P^2$ that is an embedding in the interior and maps~$\partial Z_\mathbb P$ onto a totally real surface $P\subset\C P^2$, thereby collapsing all geodesics. As observed in \cite{Roc11}, there is a~diffeomorphism~${Z_\mathbb P \cong \big[\C P^2;P\big]}$ with the Melrosian blow-up of $P$ under which $\beta_{\mathrm{LM}}$ becomes the usual blow-down map.

This picture has been an inspiration for the maps $\beta\colon Z\to \C^2$ in this article, which recover many of the features of $\beta_{\mathrm{LM}}$, while being constructed by completely different means. By analogy, we refer to them also as blow-down maps, though the geometry of $\beta(Z)\subset \C^2$ is more complicated and it is unclear to us whether it fits into the framework of Melrosian blow-ups.

\subsubsection[The Euclidean beta-map]{The Euclidean $\boldsymbol{\beta}$-map}\label{sec_Eucl}

If $M=\R^2$, equipped with the Euclidean metric, the transport twistor space is given by
\[
Z=\bigl\{(z,\mu)\in \C^2\mid |\mu|\le 1\bigr\},\]
 with involutive structure $\D =\spn_\C\bigl(\mu^2\partial_z+\partial_{\bar z},\partial_{\bar\mu}\bigr)$.
The complex geometry of $Z$ can be completely understood by means of the map
\begin{equation}\label{euclidbeta}
	\beta\colon \ Z\rightarrow \C^2,\qquad \beta(z,\mu)=\bigl(z-\mu^2\bar z,\mu\bigr),
\end{equation}
which is holomorphic and maps the interior $Z^\circ$ diffeomorphically onto the poly-disk $\big\{(w,\xi)\in\C^2\mid |\xi|<1\big\}$, with inverse explicitly given by \[
\beta^{-1}(w,\xi)=\bigl(\bigl(w+\xi^2\bar w \bigr)/\bigl(1-|\xi|^4\bigr),\xi\bigr),\qquad |\xi|<1.
\]
The existence of such a map immediately implies that there is a wealth of holomorphic functions (in fact, if $M\subset \R^2$ is any open domain, then $Z^\circ$ is a Stein surface). Moreover, $\beta$ can be viewed as desingularisation of the complex structure. Indeed, introducing $w$ and $\xi$ as new coordinates on $Z^\circ$, we simply have $\D=\spn_\C(\partial_{\bar w},\partial_{\bar \xi})$.

The Euclidean $\beta$-map serves as a prototype for the type of desingularisation that we seek out in this article. To capture the characteristic properties of $\beta$, we propose the notion of {\it holomorphic blow down structure} that we shall discuss after a brief intermezzo regarding Hermitian structures on twistor space.

\subsubsection{Hermitian structure} \label{s_hermitian}
One of the characteristic properties of the Euclidean $\beta$-map
is that it restricts to an embedding on the interior $Z^\circ$. Unfortunately, the space of embeddings
\[
	\operatorname{Emb}\bigl(Z^\circ,\C^2\bigr)\subset C^\infty\bigl(Z^\circ,\C^2\bigr)
\]
is {\it not open} in the $C^\infty$-topology, which is problematic for the purpose of perturbation theory.\footnote{If $X$ is a smooth manifold, the space $C^\infty(X,\R^n)$ can be equipped with either the {\it weak} or the {\it strong} topology and if $X$ is non-compact these topologies are genuinely different. In the present article, we exclusively use the weak topology, which gives $C^\infty(X,\R^n)$ its standard Fr{\'e}chet space structure, and simply refer to it as the $C^\infty$-topology. The strong topology, while being a differential topologist's favourite, is less suitable for us, as the restriction map~${C^\infty\bigl(Z,\C^2\bigr)\rightarrow C^\infty\bigl(Z^\circ,\C^2\bigr)}$ fails to be continuous in the strong topology. See \cite[Chapter~2]{Hir76} for more details.

 To illustrate the failure of openness in $C^\infty$-topology, one can consider the following toy example: On $X=(0,2\pi)$ define $f_\epsilon(x)=\exp({\rm i}(1+\epsilon)x)$, then $f_0\in \operatorname{Emb}(X,\C)$ is not an interior point, as it can be approximated by the non-injective immersions $(f_\epsilon : \epsilon>0)$.
} This problem can be solved by additionally keeping track of metric data and for this reason, we equip $Z^\circ$ with a Hermitian metric.

Recall that a Hermitian metric on the complex surface $Z^\circ$ is the data of a $(1,1)$-form $\Omega\in C^\infty\bigl(Z^\circ,\Lambda^{1,1}Z^\circ\bigr)$ that, expressed in a frame where $\Lambda^{1,1}\cong \C^{2\times 2}$, corresponds to ${\rm i}H$ for a Hermitian matrix~$H$. Since $\Lambda^{1,1}Z^\circ$ is globally trivial (irrespective of the topology of $M$, see Section~\ref{s_metriccx}), there is a one-to-one correspondence
\[
	\text{Hermitian metric $\Omega$} \quad \leftrightarrow \quad H\in C^\infty\bigl(Z^\circ,\mathrm{Her}_+^2\bigr),
\]
where $\mathrm{Her}_+^2\subset \C^{2\times 2}$ is the space of positive definite Hermitian $2\times 2$ matrices.
 Under this correspondence, we pick out a distinguished metric $\Om$
\begin{equation}\label{def_om}
	\Om \quad \leftrightarrow \quad \underline{H}(x,v)=\begin{bmatrix}
 \bigl(1-|v|_g^4\bigr)^2 &\\
 & 1
\end{bmatrix} \in C^\infty\bigl(Z^\circ,\mathrm{Her}_+^2\bigr).
\end{equation}
In the Euclidean case, where $\Lambda^{0,1}Z^\circ$ is canonically framed by $ \Xi^\vee = \bigl({\rm d}\bar z - \bar \mu^2 {\rm d}z \bigr)/\bigl(1-|\mu|^4\bigr)$ and~${\rm d}\bar \mu$ this metric is given by
\smash{$ \Om ={\rm i} \bigl(1-|\mu|^4\bigr)^2 \bar \Xi^\vee\wedge \Xi^\vee + {\rm i} {\rm d}\mu \wedge {\rm d}\bar \mu$}.
The significance of $\Om$ stems from the observation that, at least in the Euclidean case, it degenerates in the same way as the Jacobian of the $\beta$-map. Precisely, there exists a constant $C>0$ such that
\begin{equation}\label{betainequality}
	C^{-1} \Om \le \beta^*\Omega_{\C^2} \le C~\Om,
\end{equation}
where $\Omega_{\C^2}={\rm i}{\rm d}w\wedge {\rm d}\bar w + {\rm i} {\rm d}\xi \wedge {\rm d}\bar \xi$ is the standard Hermitian structure on $\C^2$ and `$\le$' is understood as inequality between the associated Hermitian matrices. The verification of \eqref{betainequality} is elementary and the curious reader is invited to check the inequalities at once---we will come back to them in Section~\ref{s_fra} and give the required computations alongside the constant curvature case.

\subsubsection{Holomorphic blow-down structure}
Let now $(M,g)$ be a compact oriented Riemannian surface with non-empty boundary $\partial M$. Let~$\nu$ be the inward pointing unit normal to $\partial M$ and define
influx $(+)$ and outflux $(-)$ boundaries
\[
	\partial_\pm SM=\{(x,v)\in \partial SM\mid \pm g(v,\nu(x))\ge 0\}.
\]
If $\partial M$ is strictly convex and $(M,g)$ is non-trapping (i.e., geodesics reach $\partial M$ in finite time), then~$\partial_+SM$ can be thought of as
 parameter space for the geodesics of $M$. (Here and below one could equally well decide to work with $\partial_-SM$ instead of $\partial_+SM$, and this would lead to the same notion of holomorphic blow-down structure.)

\begin{Definition}\label{def_bds}
We say that a map $\beta \in C^{\infty}\bigl(Z,\C^2\bigr)$ has a {\it holomorphic blow-down structure} if it has the following properties:
\begin{enumerate}[label=\rm(\roman*)]\itemsep=0pt
	\item\label{bds1} $\beta|_{\partial_{+}SM}\colon\partial_{+}SM\to \C^2$ is a totally real $C_\alpha^\infty$-embedding;
	
	\item\label{bds2} the restriction of $\beta$ to $Z^{\circ}$ is a biholomorphism onto its image;
	
	\item\label{bds3} $\beta^*\Omega_{\C^2}\ge c \Om$ for some constant $c>0$.
\end{enumerate}
\end{Definition}

Let us comment on the three properties: The first property implies in particular that $\beta$ separates geodesics and
is the strongest embedding property one can expect. Note that $\beta$ being holomorphic implies that $X\beta\vert_{SM}=0$ and since the geodesic vector field $X$ is tangent to the boundary of $\partial_+SM$, this enforces the Jacobian of $\beta|_{\partial_+SM}$ to drop rank. To address this issue, we change the smooth structure near the boundary of $\partial_+SM$ and refer to embeddings with respect to this new smooth structure as {\it $C_\alpha^\infty$-embeddings} (see also Definition~\ref{def_calphaemb}). Further, if~$S$ is a smooth surface (such as $\partial_+SM$ equipped with the adjusted smooth structure), then an embedding $f\colon S\rightarrow \C^2$ is called {\it totally real}, if
\[
	{\rm d}f_x(T_x S) \cap J_{f(x)} {\rm d}f_x(T_x S) = 0\qquad \text{for all} \  x\in S.
\]
Here $J$ is the complex structure of $\C^2$. This aspect parallels LeBrun--Mason's paper on Zoll surfaces \cite{LeMa02}, where the space of geodesics (the analogue of $\partial_+SM$) is realised as a totally real submanifold of $\C P^2$.

The second property is equivalent to $\beta\colon Z^\circ \rightarrow \C^2$ being both holomorphic and injective. In this case the image $\beta(Z^\circ)$ is automatically open and the Jacobian has full rank. To ensure that this property is preserved under small perturbations, it is supplemented by the third property, which corresponds to a uniform lower bound on the singular values of the Jacobian in an $\Om$-orthonormal frame. If we were merely interested in $\beta$ being an immersion on $Z^\circ$, it would suffice to keep track of the determinant of this Jacobian, which is encoded in the $2$-form ${\rm d}\beta_1\wedge {\rm d}\beta_2$. In fact,
\begin{equation}\label{locint}
\beta^*\Omega_{\C^2} \ge c \Om \quad\text{for some} \  c>0\quad \Rightarrow \quad {\rm d}\beta_1\wedge {\rm d}\beta_2 \neq 0 \quad\text{pointwise on} \  Z
\end{equation}
(see Remark \ref{rlocint}) and requiring the stronger property on the left is
 precisely what ensures that also injectivity of $\beta$ persists under small perturbations. As will be shown later, the reverse inequality in \ref{bds3} holds automatically due to holomorphicity and thus having a holomorphic blow-down structure ensures that $\beta\colon Z^\circ \rightarrow \C^2$ is a bi-Lipschitz map.

The definition of holomorphic blow down structure has been chosen such that it captures the characteristic behaviour of $\beta$-maps in special geometries, while at the same time being an open condition with respect to simultaneous variations of $\beta$ and $g$, see Theorem~\ref{thm_openbds} below.

\subsection{Main theorems}
The work on this article started with the question of how to appropriately generalise the Euclidean $\beta$-map from Section~\ref{sec_Eucl} to other geometries and how these maps behave under small perturbations.
 We discuss our approach to this systematically in Section~\ref{s_betamaps} below---first, however, we shall present some facets and consequences that can be discussed without further preparation.

\subsubsection{Global blow-downs} Let $\DD=\{z\in \C \mid |z|\le 1\}$ be the closed unit disk. Our first result concerns constant curvature metrics on $\DD$ (see also Section~ \ref{s_fra} for a more precise formulation, allowing also curvatures with~${|\kappa|\ge 1}$).

\begin{Theorem}\label{thm_constantcurvature}
Let $M=\DD$ and \smash{$g=\bigl(1+\kappa |z|^2\bigr)^{-2}|{\rm d}z|^2$} $(-1<\kappa<1)$. On transport twistor space $Z=\big\{\bigl(z,\bigl(1+\kappa|z|^2\bigr)\mu\bigr)\big\}\subset T\DD$ with coordinates $z,\mu\in \DD$, the map
\[
			\beta_{\kappa}(z,\mu) = \left(\frac{z-\mu^2\bar z}{1+\kappa \bar z^2 \mu^2},\mu \frac{1+\kappa |z|^2}{1+\kappa \bar z^2 \mu^2}\right),\qquad (z,\mu)\in \DD\times \DD,
	\]
has a {\it holomorphic blow-down structure}.
\end{Theorem}

By means of a perturbation argument, we also obtain global blow-downs for nearby metrics. For simplicity we state the following theorem as an existence result.

\begin{Theorem}\label{thm_globalbeta}
Let $M=\DD$ and suppose that $g$ is a Riemannian metric which is sufficiently close $($in $C^\infty$-topology$)$ to the Euclidean metric $($or one of the constant curvature models from the preceding theorem$)$. Then there exists a map $\beta\in C^\infty\bigl(Z,\C^2\bigr)$ with holomorphic blow-down structure.
\end{Theorem}

It turns out that the $\beta$-map in the theorem arises from a {\it canonical construction} that works for any simple surface (see Definition~\ref{def_simple}). It is canonical in the sense that it is obtained by an energy minimisation and a subsequent Szeg\H{o} projection. This energy minimiser can be expressed via the geodesic X-ray transform and the inverse of its normal operator, see Definition~\ref{def_betaextension}. The underlying invertibility result for the normal operator was proved in \cite{MNP19a}.

\subsubsection{Local blow-downs} Applying the preceding theorem to small geodesic disks, we deduce the following.

\begin{Theorem}[transport Newlander--Nirenberg theorem]\label{TNNT} Let $(M,g)$ be an oriented Riemannian surface with twistor space $Z$. Then any point $p \in Z$ admits an open neighbourhood $U\subset Z$ and a holomorphic map $\beta\colon U\rightarrow \C^2$ such that
\begin{enumerate}[label=\rm(\roman*)]\itemsep=0pt
	\item \label{thmtnnt1}${\rm d}\beta_1\wedge {\rm d}\beta_2 \neq 0$ pointwise on $U$;
	\item\label{thmtnnt2} if $q_1,q_2\in U$ are two distinct points with $\beta(q_1)=\beta(q_2)$, then both points lie on a common geodesic on $U\cap SM$---in particular, $\beta$ is injective on $U\backslash SM$.
\end{enumerate}
Consequentially $\mathcal O(U)=\{f\in C^\infty(U)\mid f \text{ holomorphic} \}$ separates points in $U\backslash SM$.
\end{Theorem}

In the context of involutive structures (cf.~\cite{Tre92}), the existence of local solutions to ${\rm d}\beta|_{\D} = 0$ with property \ref{thmtnnt1} is referred to as {\it local integrability}. For $p\in Z\backslash SM$, local integrability is guaranteed by the classical Newlander--Nirenberg theorem (e.g., \cite[Theorem~2.6.19]{Huy05}) and in this case property~\ref{thmtnnt2} is a consequence of \ref{thmtnnt1} and the inverse function theorem. The crux of the preceding theorem is that we obtain local integrability together with the separation property~\ref{thmtnnt2} also near points $p\in SM$, where $\D$ degenerates and the inverse function theorem cannot be applied.

 Since this is a local result, it is also valid on the projective twistor space $Z_\mathbb P$ and it appears like this is the first such integrability result beyond LeBrun--Mason's work in the Zoll setting.

\subsubsection{Application: Identity principles} Let $Z$ be the twistor space of an oriented Riemannian surface $(M,g)$ and let $\Sigma$ be a compact connected Riemann surface with non-empty boundary. We consider holomorphic maps
${f\colon \Sigma\rightarrow Z}$
and refer to these as {\it holomorphic curves} parametrised by $\Sigma$. Here holomorphic means that $f$ is smooth up to the boundary and satisfies $f_*\bigl(T^{0,1}\Sigma\bigr)\subset \mathscr D$.

An example of a holomorphic curve is the embedding of the closed unit disk $\DD=\{\mu\in \C\mid |\mu|\le 1\}$ as fibre of $Z$. Precisely, if $(x,v)\in SM$, then the fibre over $x$ is parametrised by $f(\mu)=(x,\mu\cdot v)$. The local point separation of Theorem~\ref{TNNT} has the following consequence.

\begin{Corollary}[identity principle for holomorphic curves]\label{cor_curves}
Let $f_1,f_2\colon \Sigma\rightarrow Z$ be holomorphic curves with $f_1=f_2$ on $\partial \Sigma$. Then $f_1=f_2$ on all of $\Sigma$.
\end{Corollary}

Let now $(M_i,g_i)$ ($i=1,2$) be two oriented Riemannian surfaces and $Z_1$ and $Z_2$ their associated twistor spaces. We say that a map $\Phi\colon Z_1\rightarrow Z_2$ is holomorphic, if it is smooth up to the boundary and satisfies $\Phi_*(\mathscr D_1)\subset \mathscr{D}_2$. Such a map $\Phi$ is determined by its action on the fibres of $Z_1$ and hence Corollary \ref{cor_curves} implies the following.

\begin{Corollary}[identity principle for holomorphic maps]\label{cor_maps}
Suppose $\Phi,\Psi\colon Z_1\rightarrow Z_2$ are two holomorphic maps with $\Phi|_{SM_1}=\Psi|_{SM_1}$. Then $\Phi=\Psi$ on $Z_1$.
\end{Corollary}

This identity principle is one of the key ingredients of our follow-up work \cite{BMP24b}, where we prove that biholomorphisms between transport twistor spaces exhibit rigidity: if the underlying metrics are simple or Anosov, then every biholomorphism is, up to constant rescaling and the antipodal map, the lift of an orientation preserving isometry. As the restriction of a biholomorphism to unit tangent bundles always induces an orbit equivalence,
this is closely related to questions such as conjugacy rigidity and lens rigidity (see, e.g., \cite{CGL23,GoRo23,GLP25,GM18} for some recent related works). We refer to \cite{BMP24b} for further details.

\subsubsection[Strategy for the proof of Theorem 1.3]{Strategy for the proof of Theorem~\ref{thm_globalbeta}}
As already pointed out, the map $\beta_g$ in Theorem~\ref{thm_globalbeta} is canonically constructed using the geodesic X-ray transform together with the inverse of its normal operator. The first step in the proof consists of showing that this canonical construction agrees with the explicit maps from Theorem~\ref{thm_constantcurvature}; here we exploit the rotational symmetry to avoid the difficult task of inverting normal operators. Once this is achieved, the core idea of the proof is simple---we show that
$g\mapsto \beta_g$ is continuous in a suitable sense and ensure that the notion of holomorphic blow-down structure (Definition~\ref{def_bds}) survives small perturbations. Implementing this idea is technically challenging. The continuous dependence involves many moving pieces, such as changes in the $C_\alpha^\infty$-structure or changes of the inverse of the normal operator in the Grubb--H{\"o}rmander framework used in \cite{MNP19a}. At the same time, the perturbation theory of the holomorphic blow-down structure takes place on the non-compact manifold $Z^\circ$ and this requires dealing with a metric-dependent degeneracy of the complex structure.

\subsection{Structure of the article} Section~\ref{s_prelim} discusses preliminaries, including fibrewise Fourier analysis, X-ray transforms and their normal operators, the $C^{\infty}_{\alpha}$-structure, the twistor space and the Hermitian metric we will use.
Section~\ref{s_betamaps} introduces the canonical $\beta$-maps for simple surfaces and sets up the framework for the perturbation theory.
Section~\ref{s_fra} is devoted to the constant curvature models and proves Theorem~\ref{thm_constantcurvature}.
In Sections \ref{s_upsilon} and \ref{s_verstappen},
 various continuous metric dependencies are carefully tracked. Section~\ref{s_bds} contains the central openness statement needed for the perturbation theory.
The proofs of the main theorems and corollaries are completed in Section~\ref{s_tnnt}. Finally, the paper is supplemented with two appendices providing facts needed for the continuity arguments; some of these facts were not readily available in the literature in the form we required.\looseness=1

\subsubsection{Further directions} It would be very interesting to exhibit maps with a holomorphic blow-down structure for the transport twistor space of further Riemannian surfaces---a natural setting where one might expect such blow blowns is that of so-called simple surfaces, see also the discussion in Section~\ref{s_canonicalbeta} and Question~\ref{questionsimplebds}.
However, simplicity is not a necessary condition, as recently demonstrated by the second author and Qi by constructing explicit maps with holomorphic blow-down structure for certain non-simple Herglotz metrics on the disk~\cite{MoQi25}.

The existence of a map with holomorphic blow-down structure has several implications on the complex geometry of transport twistor space, e.g., one immediately sees that the interior points of $Z$ are separated by the holomorphic functions $f\colon Z\to \C$. Proving such a separation result for geometries beyond the present perturbative setting would be of independent interest.

\section{Preliminaries}\label{s_prelim}

\subsection{Fibrewise Fourier analysis}\label{s_ffa} Let $(M,g)$ be an oriented Riemannian surface. The vertical vector field $V$ on $SM$ is defined as the infinitesimal generator of the circle action $(x,v)\mapsto \bigl(x,{\rm e}^{{\rm i}t}v\bigr)$, $t\in \R$. Defining
\[
	\Omega_k=\{f\in C^\infty(SM)\mid Vf={\rm i}k f\},\qquad k\in \Z,
\]
any $u\in C^\infty(SM)$ has a unique decomposition $u=\sum_{k\in \Z}u_k$ into its vertical Fourier modes $u_k\in \Omega_k$. Functions that only have non-trivial Fourier modes in degrees $k\ge 0$ are called {\it fibrewise holomorphic}.
There are natural isomorphisms
\begin{equation}\label{modeiso}
	\Omega_k \cong C^\infty\bigl(M,\otimes^k \bigl(T^{1,0}M\bigr)^*\bigr)\qquad \text{and} \qquad \Omega_{-k} \cong C^\infty\bigl(M,\otimes^k \bigl(T^{0,1}M\bigr)^*\bigr),\qquad k\ge 0,
\end{equation}
where a $k$-tensor is pulled back to a function in $\Omega_k\subset C^\infty(SM)$ by inserting the velocity variable $k$-times.

Let now $Z$ be the twistor space of $(M,g)$ and consider the algebra of holomorphic functions on $Z$, assumed to be smooth up to the boundary
\[ \mathcal A(Z)=\{f\in C^\infty(Z)\mid f\text{ holomorphic}\}.\]
As mentioned in Section~\ref{s_correspondences}, there is an isomorphism
\begin{equation}\label{def_aoriginal}
\mathcal A(Z)\xrightarrow{\sim} \biggl\{u\in C^\infty(SM) \mid Xu=0,\, u={\sum_{k\ge 0}}u_k\biggr\},
\end{equation}
sending a holomorphic function on $Z$ to its restriction on $SM$. For $m\ge 0$, define
\begin{gather*}
\mathcal A_m(Z):=\biggl\{f\in \mathcal A(Z)\mid f|_{SM}=\sum_{k\ge m}u_k\biggr\},\\
 \mathcal H_m := \big\{a\in C^\infty\bigl(M,\otimes^m\bigl(T^{1,0}M\bigr)^*\bigr)\mid \bar \partial a=0\big\}.
\end{gather*}
Given an element $f\in \mathcal A_m(Z)$, we may select the $m$-th Fourier mode of $f|_{SM}$ and---by means of isomorphism \eqref{modeiso}---view this as a section $a\in C^\infty\bigl(M,\bigl(T^{1,0}M\bigr)^m\bigr)$. We then have $a\in \mathcal H_m$, the space of holomorphic $m$-differentials, that is, tensors that are locally of the form $A(z) {\rm d}z^m$ for a holomorphic function $A(z)$.
Writing $a=\pi_{m*}(f)$, the situation is neatly summarised in the following (not necessarily exact) sequence:
\[
	0 \rightarrow \mathcal A_{m+1}(Z) \hookrightarrow \mathcal A_m(Z) \xrightarrow{\pi_{m*}} \mathcal{H}_m\rightarrow 0.
\]

\subsection{X-ray transforms and normal operator}\label{s_geoxray} Let $(M,g)$ be a manifold of dimension $d\ge 2$ that is non-trapping
and has a strictly convex boundary.
The
 {\it geodesic X-ray transform} on $(M,g)$ is defined as map
\[
	I\colon\ C_c^\infty(SM^\circ)\rightarrow C^\infty_c(\partial_+SM^\circ),\qquad If(x,v)=\int_0^{\tau(x,v)} f(\varphi_t(x,v)) {\rm d}t,
\]
where $(\varphi_t)$ is the geodesic flow on $SM$ and $\tau(x,v)\ge 0$ is the hitting time of the orbit $\varphi_t(x,v)$ with $\partial SM$. Let $\rho\colon M\to [0,\infty)$ be a boundary defining function of $M$ (that is, $\partial M= \{\rho=0\}$ and ${\rm d}\rho \neq 0$ on $\partial M$) and consider this as function on $SM$. Then the X-ray transform extends to a continuous map
\begin{equation}\label{def_xray}
	I\colon\ \rho^{-1/2}C^\infty(SM)\rightarrow C_\alpha^\infty(\partial_+SM),
\end{equation}
where $C_\alpha^\infty(\partial_+SM)$ denotes the space of smooth maps $h\colon \partial_+SM\rightarrow \C$ for which the flow-invariant extension $h^\sharp$ is smooth on all of $SM$ (i.e., $h^\sharp \in C^\infty(SM), Xh^\sharp =0$ and $h^\sharp = h$ on $\partial_+SM$). If~${\pi\colon SM\rightarrow M}$ is the base point projection, then we define $I_0=I\circ \pi^*$ and the {\it normal operator}~by
\begin{gather*}
	N_0\colon \ \rho^{-1/2}C^\infty(M)\rightarrow C^\infty(M),\\
 N_0 f(x):= \pi_*(I_0 f)^\sharp(x) = 2\int_{S_x M} {\rm d}v \int_0^{\tau(x,v)} f\circ \varphi_t(x,v) {\rm d}t.
\end{gather*}
We may view $N_0$ as $I_0^* I_0$, where $I_0^*$ is the formal adjoint with respect to the {\it symplectic measure} on $\partial_+SM$. This is given by $g(\nu(x),v) {\rm d}\Sigma^{2d-2}$, where ${\rm d} \Sigma^{2d-2}$ is the volume form for the {Sasaki metric}. For the corresponding norm, we use the notation
\[
|| h ||_{\mathrm{sym}}^2 = \int_{\partial_+SM} |h(x,v)|^2\cdot g(v,\nu(x)) {\rm d}\Sigma^{2d-2}(x,v),\qquad h\in C_\alpha^\infty(\partial_+SM).
\]
(This is frequently denoted $||\cdot||_\mu$ or $||\cdot||_{L^2_\mu(\partial_+SM)}$, but we avoid the $\mu$-notation, as to not cause confusion with the canonical variables on twistor space.)

\subsubsection{Two-dimension case}
If $d=2$, we may restrict the X-ray transform to the $k$-th Fourier mode to obtain maps
\[
I_k\colon\ \rho^{-1/2} \Omega_k\rightarrow C_\alpha^\infty(\partial_+SM),\qquad k\in \Z.
\]
For $k=0$, this corresponds to $I_0$, as defined above, if we identify $\Omega_0$ with $C^\infty(M)$. Writing $(\cdot)_k$ for the projection onto the $k$-th mode, we define normal operators
\[
	N_k\colon\ \rho^{-1/2}\Omega_k\rightarrow \Omega_k,\qquad N_k f = \bigl((I_k f)^\sharp\bigr)_k.
\]

\subsection[The C\_alpha\^{}infty-structure]{The $\boldsymbol{C_\alpha^\infty}$-structure}\label{s_calpha}

Let $(M,g)$ be a non-trapping manifold with strictly convex boundary. The space $C_\alpha^\infty(\partial_+SM)$, defined below \eqref{def_xray}, can be viewed as space of {\it all} smooth functions with respect to a modified smooth structure on $\partial_+SM$.
To see this, we embed $M$ into the interior of a slightly larger manifold $\bigl(\hat M,g\bigr)$ which is also non-trapping and has a strictly convex boundary. Let \smash{$\hat \tau \in C^\infty(SM,[0,\infty))$} denote the exit time of \smash{$\hat M$} and consider the smooth map
\begin{equation}\label{def_foldmap}
	F_g\colon \ \partial SM\rightarrow \partial S\hat M,\qquad F_g(x,v) = \hat \varphi_{-\hat \tau(x,-v)}(x,v),
\end{equation}
where $(\hat \varphi_t)$ is the geodesic flow of $\bigl(\hat M,g\bigr)$. This encodes the scattering relation $\alpha$ in the sense that~${F_g(x,v)=F_g(x',v')}$ if and only if $\alpha(x,v)=(x',v')$.
 As discussed in \cite[Section~5.2]{PSU23}, the map~$F_g$ restricts to an embedding on $\partial_+SM^\circ$ and has the structure of a {\it Whitney fold} near every point~$(x_0,v_0)\in \partial_0SM$. The latter property means that there are local coordinates~${(u_1,\dots, u_{2d-2})}$ and~$(v_1,\dots, v_{2d-2})$, centred at $(x_0,v_0)$ and $F_g(x_0,v_0)$, respectively, in which~$F_g$ has the following normal form:
\begin{equation}\label{whitneynormalform}
	F_g(u_1,u_2,\dots, u_{2d-2}) = \bigl(u_1^2,u_2,\dots, u_{2d-2}\bigr).
\end{equation}
As an immediate consequence, $S_g=F_g(\partial_+SM)\subset \partial S\hat M$ is a smooth submanifold with boundary and $F_g\colon \partial_+SM\rightarrow S_g$ is a homeomorphism.

\begin{Definition}
	On $\partial_+SM$, a new smooth structure is defined by transferring the smooth structure of $S_g$ via the homeomorphism $F_g\colon \partial_+SM\rightarrow S_g$. The structure will be referred to as {\it $C_\alpha^\infty$-structure} and its maximal atlas will be denoted by $\mathcal A_\alpha$.
\end{Definition}

The space of $\mathcal A_\alpha$-smooth functions $C^\infty(\partial_+SM,\mathcal A_\alpha)$ coincides with $C^\infty_\alpha(\partial_+SM)$, as we can verify by means of the following chain of equalities
\[
	C_\alpha^\infty(\partial_+SM) =F_g^*C^\infty\bigl(\partial S\hat M\bigr) = F_g^*C^\infty(S_g) = C^\infty(\partial_+SM,\mathcal A_\alpha).
\]
The first equality is a consequence of \cite[Theorem~C.4.4]{Hor07}, the second one follows from $S_g\subset \partial S \hat M$ being a smooth submanifold and the last one holds by definition.

\begin{Definition}\label{def_calphaemb}
	A map $f\colon \partial_+SM\rightarrow \R^m$ is called {\it $C_\alpha^\infty$-embedding}, if it is an injective $\mathcal A_\alpha$-smooth immersion. (Equivalently, if it is a smooth embedding of $(\partial_+SM,\mathcal A_\alpha)$.)
\end{Definition}

\begin{Remark} In the local coordinate system from \eqref{whitneynormalform}, any $f\in C_\alpha^\infty(\partial_+SM,\R^m)$ satisfies ${\partial_{u_1}|_{u_1=0}f = 0}$ and thus---using the standard smooth structure---it cannot be an immersion up to the boundary. For $\mathcal A_\alpha$-smooth immersions, one instead considers rescaled Jacobians, where~$\partial_{u_1}f$ is replaced by~$u_1^{-1}\partial_{u_1}f$.
\end{Remark}

\subsection{Hermitian manifolds}\label{s_metriccx} A Hermitian metric on a complex $n$-manifold $(X,J)$ is a Riemannian metric $G$ for which $J$ is an isometry. The real $(1,1)$-form $\Omega(\cdot,\cdot)=G(J\cdot,\cdot)$ is called the {\it fundamental $2$-form} of $G$, and with respect to a local frame $\eta_1,\dots,\eta_n$ of $\Lambda^{1,0}X$ it is given by
\begin{equation}\label{generalomegah}
\Omega = {\rm i} \sum_{j,k=1}^n H_{jk} \cdot \eta_j\wedge\bar\eta_k
\end{equation}
for a positive definite Hermitian $n\times n$-matrix $H=(H_{jk})$. We also refer to $\Omega$ itself as the Hermitian metric, with the understanding that $G=\Omega(\cdot,J\cdot)$ is the actual metric tensor. The associated Riemannian quantities will be labelled with $\Omega$,~e.g., $d_\Omega$ is the Riemannian distance function and $\vert\cdot\vert_\Omega^2$ is the induced inner product on forms.
Given two Hermitian metrics $\Omega$ and $\Omega'$ we write $\Omega \le \Omega'$ if and only if $G'-G$ is positive semi-definite (or equivalently if $H'-H\in \C^{n\times n}$ is positive semi-definite in any coordinate chart).

\subsubsection{Hermitian metrics on twistor space}
Let $(M,g)$ be an oriented Riemannian surface with twistor space $Z$. To study twistor space invariantly, it is convenient to consider $Z$ as the base of the principal circle bundle
\[
	\mathsf{p}\colon\ SM\times \DD\rightarrow Z,\qquad \mathsf{p}(x,v,\omega)=(x,\omega\cdot v),
\]
where the product $\omega\cdot v$ is defined by means of the complex structure of $T_x M$.
It turns out that~$\mathsf{p}^* \bigl(\Lambda^{0,1}Z^\circ\bigr)$ has a global frame consisting of
$1$-forms
$
\tau,\gamma \in \Omega^1(SM\times \DD^\circ,\C)$.
Let $\mathbf V$ be the vector field on $SM\times \DD$ that generates the circle action $(x,v,\omega)\act {\rm e}^{{\rm i}t}=\bigl(x,{\rm e}^{{\rm i}t}v,{\rm e}^{-{\rm i}t}\omega\bigr)$ ($t\in \R/2\pi \Z$) and denote with $\mathcal L_{\mathbf V}$ the associated Lie derivative. Then $\tau,\gamma \in \ker (\mathcal L_\mathbf{V} - {\rm i})$ and thus $\mathcal L_{\mathbf V}(\bar \tau\wedge \tau)=\mathcal L_{\mathbf V}(\bar \tau\wedge \gamma)=\dots = 0$. As a consequence, the $\mathsf{p}$-push-forwards of these $2$-forms exist and yields a global frame
\[
	\{ \mathsf{p}_*(\bar \tau\wedge \tau),\mathsf{p}_*(\bar \tau\wedge \gamma),\mathsf{p}_*(\bar \gamma\wedge \tau),\mathsf{p}_*(\bar \gamma\wedge \gamma)\}\subset \Lambda^{1,1}Z^\circ.
\]
Using this frame, we associate to any $H=(H_{ij})\in C^\infty\bigl(Z^\circ,\mathrm{Her}_2^+\bigr)$ a Hermitian metric
\[
	 \Omega={\rm i}\{H_{11} \mathsf{p}_*(\bar \tau\wedge \tau)+ H_{12} \mathsf{p}_*(\bar \tau\wedge \gamma)+H_{21} \mathsf{p}_*(\bar \gamma\wedge \tau)+H_{22} \mathsf{p}_*(\bar \gamma\wedge \gamma) \}
\]
on $Z^\circ$ and this gives the one-to-one correspondence from Section~\ref{s_hermitian}. In this article, we only work with the distinguished metric $\Om$ defined in \eqref{def_om}.
For further details on this invariant approach to twistor spaces and the definition of $\tau$, $\gamma$, we refer to \cite{BLP24,BoPa23}.

\subsection{Twistor space of conformally Euclidean disks}\label{s_conformallyeuclidean}
If $M=\DD=\{z\in \C\mid |z|\le 1\}$ is equipped with the conformally Euclidean metric ${\rm e}^{2\sigma}|{\rm d}z|^2$, where~${\sigma\in C^\infty(\DD,\R)}$, then the twistor space is biholomorphic to
\[
	Z=\big\{(z,\mu)\in \C^2\mid |z|,|\mu|\le 1\big\},\qquad \D_\sigma = \spn_\C(\Xi_\sigma,\partial_{\bar \mu}),
\]
where $\Xi_\sigma$ is a suitable complexification of the geodesic vector field, explicitly,
\begin{equation}
	\Xi_\sigma={\rm e}^{-\sigma} \bigl(\mu^2\partial_z+\partial_{\bar z} + \bigl(\mu^2 \partial_z\sigma - \partial_{\bar z}\sigma\bigr)(\bar \mu \partial_{\bar \mu} - \mu \partial_\mu) \bigr) \in C^\infty(Z,T_\C Z).
	\label{Xisig}
\end{equation}
The biholomorphism of this model (which, by abuse of notation, is also denoted $Z$), with the transport twistor space of $\bigl(\DD,{\rm e}^{2\sigma} |{\rm d}z|^2\bigr)$ is provided by the map
\begin{equation}\label{twistorisoconformal}
	Z\rightarrow Z\bigl(\DD,{\rm e}^{2\sigma}|{\rm d}z|^2\bigr),\qquad (z,\mu)\mapsto \bigl(z,{\rm e}^{-\sigma(z)}(\mu\partial_z+\bar \mu \partial_{\bar z})\bigr)\in T\DD.
\end{equation}
We refer to \cite[Section~4.2]{BoPa23} for more details and will henceforth understand this identification implicitly. A~direct computation shows that the dual frame to $\{\Xi_\sigma,\partial_{\bar \mu}\}$ is given by the following $1$-forms:
\[
	\Xi_\sigma^\vee = e^\sigma \frac{{\rm d}\bar z- \bar \mu^2 {\rm d}z}{1-|\mu|^4},\qquad \partial_{\bar\mu}^\vee = {\rm d}\bar \mu + \bar \mu\{\partial_{\bar z} \sigma {\rm d}\bar z - \partial_z \sigma {\rm d}z\}.
\]
Further, under the isomorphism in \eqref{twistorisoconformal} it turns out that $\bar \Xi_\sigma^\vee \wedge \Xi_\sigma^\vee = \mathsf{p}_*(\bar \tau\wedge \tau)$, $\bar \Xi_\sigma^\vee \wedge \partial_{\bar \mu}^\vee = \mathsf{p}_*(\bar \tau\wedge \gamma),\dots$ and thus the Hermitian metric encoded by \smash{$H\in C^\infty\bigl(Z^\circ,\mathrm{Her}_2^+\bigr)$} takes the form~\eqref{generalomegah}, when we choose $\{\eta_1,\eta_2\}=\{\bar {\Xi}^\vee_\sigma,\partial_\mu^\vee\}$ as frame of $\Lambda^{1,0}Z^\circ$. In particular, the distinguished metric~$\Om$ from Section~\ref{s_hermitian} is given by
\begin{equation}\label{def_omegasigma}
\Om_\sigma ={\rm i} \bigl(1-|\mu|^4\bigr)^2 \bar \Xi_\sigma^\vee \wedge \Xi_\sigma^\vee + {\rm i} \partial_\mu^\vee\wedge \partial_{\bar \mu}^\vee.
\end{equation}

\section[Canonical beta-maps and their perturbation theory]{Canonical $\boldsymbol{\beta}$-maps and their perturbation theory}\label{s_betamaps}

The Euclidean $\beta$-map arises in a rather systematic way. We will demonstrate this in the context of simple surfaces, where we define the notion of {\it $\beta$-extensions} and study their behaviour under perturbations of the metric.

\begin{Definition} \label{def_simple}A compact Riemannian surface $(M,g)$ is called {\it simple} if it is non-trapping, free of conjugate points and has a strictly convex boundary $\partial M$.
\end{Definition}

\subsection[Canonical beta-maps for simple surfaces]{Canonical $\boldsymbol{\beta}$-maps for simple surfaces}\label{s_canonicalbeta}

\subsubsection{Holomorphic extensions}\label{s_holext}
The Euclidean $\beta$-map in \eqref{euclidbeta} has the components $\beta^{(0)}(z,\mu)=z-\mu^2 \bar z$ and $\beta^{(1)}(z,\mu)=\mu$. In terms of the spaces $\mathcal A_m(Z)$ and $\mathcal H_m$ and the maps $\pi_{m*}$ defined in Section~\ref{s_ffa}, we may express this as
\begin{gather*}
	\beta^{(0)} \in \mathcal A_0(Z)\qquad \text{and} \qquad \pi_{0,*}\beta^{(0)} = z \in \mathcal H_0,\\
	\beta^{(1)} \in \mathcal A_1(Z)\qquad \text{and} \qquad \pi_{1,*}\beta^{(1)} = {\rm d}z \in \mathcal H_1.
\end{gather*}
That is, $\beta$ is a holomorphic extension of the tuple $(z,{\rm d}z)\in \mathcal H_0\times \mathcal H_1$. Due to a classical result of Pestov--Uhlmann \cite{PeUh05}, this type of holomorphic extension exists for general simple surfaces (see~\cite{BLP24,BoPa23} for a discussion in the context of twistor spaces).

\begin{theorem*}[holomorphic extensions]
Let $(M,g)$ be simple and $Z$ its transport twistor space. Then the following sequence is exact:
\[
	0\rightarrow \mathcal A_{m+1}(Z)\hookrightarrow \mathcal A_m(Z) \xrightarrow{\pi_{m*}} \mathcal H_m\rightarrow 0.
\]	
\end{theorem*}
Of course, the holomorphic extension of $(z,{\rm d}z)$ is far from unique, but we shall see below that there is a canonical choice that is consistent with the Euclidean $\beta$-map.

\subsubsection{Canonical first integrals}
 Let us first discuss a canonical choice of first integral on $SM$, that is based on the following result from \cite{MNP19a}.

\begin{Proposition}[isomorphism property]\label{prop_isonk}
	Let $(M,g)$ be simple and $k\in \Z$. Then the normal operator $N_k\colon \rho^{-1/2}\Omega_k\to \Omega_k$ {\rm(}defined in Section~{\rm\ref{s_geoxray})} is an isomorphism.
\end{Proposition}

\begin{proof}
For $k=0$, this is part of \cite[Theorem 2.2]{MNP19a}. For $k\neq 0$, we introduce the attenuation $a= {\rm e}^{-{\rm i}k\theta} X\bigl({\rm e}^{{\rm i}k\theta}\bigr)\in \Omega_{-1}\oplus\Omega_1$, where the angle $\theta$ is chosen with respect to a trivialisation of $SM$. One can check that $N_k$ fits into the commutative diagram
\[
	\begin{tikzcd}
		\rho^{-1/2}\Omega_0 \arrow["N_{a,0}"]{r} \arrow[swap,"\times {\rm e}^{{\rm i}k\theta}"]{d}& \Omega_0 \arrow["\times {\rm e}^{{\rm i}k\theta}"]{d}\\
		\rho^{-1/2}\Omega_k \arrow["N_k"]{r} & \Omega_k,
	\end{tikzcd}
\]
where $N_{a,0}=I_{a,0}^*I_{a,0}$ is the normal operator of the attenuated X-ray transform $I_{a,0}$ considered in~\cite[Section~12.2]{PSU23}. It is then enough to prove the corresponding isomorphism property for~$N_{a,0}$ and this can be done by the same methods as in~\cite{MNP19a}.
\end{proof}

\begin{Proposition}[canonical first integrals]\label{minimalsolution} Let $(M,g)$ be simple and $f\in \Omega_k$ for some $k\in \Z$. Then there exists a solution $u\in C^\infty(SM)$ to the problem
\begin{equation}\label{minimalsolution0}
		X u = 0 \qquad \text{and} \qquad u_k = f
\end{equation}
and this solution is unique, if one requires any of the following equivalent conditions:
\begin{enumerate}[label=\rm (\roman*)]\itemsep=0pt
	\item\label{minimalsolution1} $||u|_{\partial_+SM}||_{\mathrm{sym}}$ is minimal amongst all solutions to \eqref{minimalsolution0};
	
	\item\label{minimalsolution2} $u\vert_{\partial_+SM}$ lies in the range of $I_k\colon \rho^{-1/2}\Omega_k \rightarrow C_\alpha^\infty(\partial_+SM)$;
	
	\item\label{minimalsolution3} $u = \big[I_k N_k^{-1} f\big]^\sharp$.
\end{enumerate}
\end{Proposition}

\begin{proof}
Define $u$ be the formula in \ref{minimalsolution3} and let $h=u|_{\partial_+SM}\in C^\infty_\alpha(\partial_+SM)$, then $I_k^* h = f$. Suppose that $h=I_k a$ for some $a\in \rho^{-1/2}\Omega_k$, then we must have $N_k a = I_k^*I_k a = f$ and hence $a=N_k^{-1} f$. This shows that \ref{minimalsolution2} and \ref{minimalsolution3} are equivalent. Let $h'=u'|_{\partial_+SM}$ for any other solution~$u'$ to \eqref{minimalsolution0}. Then
$
	\langle h, h-h'\rangle_{\mathrm{sym}}=\langle I_k a, h-h'\rangle_{\mathrm{sym}}= \langle a, I_k^*(h-h')\rangle_{L^2(SM)}=0
$ and hence
\[
|| h ||_\mathrm{sym}^2 =|\langle h, h'\rangle_\mathrm{sym}| \le || h ||_\mathrm{sym}\cdot || h' ||_\mathrm{sym},
\]
by the Cauchy--Schwarz inequality. Hence
 \ref{minimalsolution2} implies \ref{minimalsolution1}. Vice versa, if $h'$ is also an energy minimiser, then in the previous display we have equality, and thus $h$ and $h'$ have to be linearly dependent, which is only possible if $u=u'$.
\end{proof}

As discussed in Section~\ref{s_holext}, the Euclidean $\beta$-map from \eqref{euclidbeta} is a holomorphic extension of~$(z,{\rm d}z)$. In view of the preceding proposition, we now see that it is singled out amongst all such extensions by having minimal $||\cdot||_\mathrm{sym}$-energy.

\begin{Proposition}[extremality of Euclidean $\beta$-map] \label{extramlbeta} Let $M=\DD$, equipped with the Euclidean metric. Then the map $\beta(z,\mu)=\bigl(z-\mu^2\bar z,\mu\bigr)$ satisfies
\[\beta(z,\mu)=\bigl(\big[I_{0}N_{0}^{-1}(z)\big]^{\sharp}, \big[I_{1}N_{1}^{-1}({\rm d}z)\big]^{\sharp}\bigr),\qquad (z,\mu)\in S\DD.\]
$($Here ${\rm d}z$ is viewed as element of $\Omega_1$ via the isomorphism in \eqref{modeiso}.$)$
\end{Proposition}

This proposition is proved together with the constant curvature case in Section~\ref{s_fra}---see also Theorem~\ref{thm:preimages}, where it is restated in context. When the metric is rotation-invariant, the proof relies on the unique identifiability of certain geodesic first integrals which are also rotation-equivariant, see Lemma~\ref{lem:vanishing} below.

\subsubsection[Definition of beta-maps on simple surfaces]{Definition of $\boldsymbol{\beta}$-maps on simple surfaces}

Let $(M,g)$ be a simple surface. Guided by the preceding considerations, we now describe how to canonically extend a pair $(f,{\rm d}f)\in \mathcal H_0\times \mathcal H_1$ to a holomorphic map $\beta\colon Z\rightarrow \C^2$ on twistor space. As the first integrals in Proposition~\ref{minimalsolution} might not always be fibrewise holomorphic (even if $k\ge 0$), we introduce the Szeg\H{o} projectors $\mathbb S_{\ev},\mathbb S_{\mathrm{odd}}\colon C^\infty(SM)\rightarrow C^\infty(SM)$
\[
\mathbb S_{\ev}u=\sum_{k\ge 0} u_{2k}\qquad \text{and} \qquad \mathbb S_{\mathrm{odd}} u = \sum_{k\ge 0} u_{2k+1}.
\]
The Szeg\H{o} projections of an invariant function remain invariant if the lowest mode (i.e., $u_0$ or~$u_1$) is in the kernel if $\eta_-$.

\begin{Definition}\label{def_betaextension}
 Let $(M,g)$ be a simple surface and $f\colon M\rightarrow \C$ a holomorphic map. The {\it $\beta$-extension} of $f$ is defined as the holomorphic map $\beta\colon Z \rightarrow \C^2$ determined by
 \[
 	\beta|_{SM} = \bigl(\mathbb S_{\ev} \bigl(I_0N_0^{-1} f \bigr)^\sharp, \mathbb S_{\mathrm{odd}} \bigl(I_1N_1^{-1}{\rm d}f\bigr)^{\sharp}\bigr),
 \]
via the isomorphism \eqref{def_aoriginal}.
\end{Definition}

\begin{Remark}\label{rmkonglobbeta}
The $\beta$-map in Theorem~\ref{thm_globalbeta} is precisely the $\beta$-extension of an appropriate holomorphic embedding $f\colon M\rightarrow \C$. If $g$ is already known to be conformally Euclidean, then one can simply take $f(z)=z$.
\end{Remark}

This leads to the following question, which currently seems to be out of reach.

\begin{Question}\label{questionsimplebds}
	Does every simple surface $(M,g)$ admit an embedding $f\colon M\rightarrow \C$ such that the $\beta$-extension $\beta\colon Z\rightarrow \C^2$ has a holomorphic blow-down structure?
\end{Question}

\subsection[Perturbation theory for beta-maps]{Perturbation theory for $\boldsymbol{\beta}$-maps}\label{s_perturbationtheory}

Instead of addressing Question \ref{questionsimplebds} directly, we prove that the canonical $\beta$-maps depend continuously on the metric and that the property of having a holomorphic blow-down structure is preserved under perturbations. In view of the Riemann mapping theorem (or rather its version with continuous dependence on the metric, discussed in Appendix~\ref{s_RMT}), it suffices to consider conformally Euclidean metrics and we shall do this from now on.

\subsubsection{Continuous dependence on the conformal factor} We now consider the $\beta$-extensions from Definition~\ref{def_betaextension} in the conformally Euclidean case. For this, write
\[
	\Sigma_\mathrm{s} =\big\{\sigma \in C^\infty(\DD,\R)\mid {\rm e}^{2\sigma}|{\rm d}z|^2 \text{ is simple} \big\}
\]
and note that this is an open set in the $C^\infty$-topology.
As perturbing $\sigma$ does not change the complex structure of $M$, the spaces $\mathcal H_m$ of holomorphic $m$-differentials remain unchanged and~${(z,{\rm d}z)\in \mathcal H_0\times \mathcal H_1}$. Further, via the isomorphism in \eqref{twistorisoconformal}, the $\beta$-extensions of $f(z)= z$ for varying $\sigma$'s can all be considered as maps on the same space $Z=\DD\times \DD$
\[
	\beta_\sigma \colon\ Z \rightarrow \C^2,\qquad \sigma\in \Sigma_{\mathrm{s}}.
\]
We collect the relevant properties of these maps in the following theorem, proved in Section~\ref{s_proof_betafamily}.

\begin{Theorem} \label{thm_betafamily}
Let $Z=\DD\times \DD$ and $\sigma \in \Sigma_\mathrm{s}$. Then,
\begin{enumerate}[label=\rm(\roman*)]\itemsep=0pt
	\item \label{betafamily1} $\beta_\sigma\colon Z\rightarrow \C^2$ is $\D_\sigma$-holomorphic, i.e., $\Xi_\sigma \beta_\sigma= \partial_{\bar \mu} \beta_{\sigma}=0$;
	
	\item \label{betafamily2} if $\sigma=-\log\bigl(1+\kappa|z|^2\bigr)$, then $\beta_\sigma$ equals the map $\beta_\kappa$ from Theorem~{\rm\ref{thm_constantcurvature}};
	
	\item \label{betafamily3}the following map is continuous:
	\[
		\sigma\mapsto \beta_\sigma,\qquad \Sigma_\mathrm{s}\rightarrow C^\infty\bigl(Z,\C^2\bigr).
	\]
\end{enumerate}
\end{Theorem}

Of course, \ref{betafamily1} follows directly by construction and \ref{betafamily2} is a consequence of Proposition~\ref{extramlbeta} (or rather, its extension to constant curvature). The novel part of the theorem is thus the continuity claim and this is proved by keeping track of the metric dependencies at each step in the construction in Section~\ref{s_canonicalbeta}.

\begin{Remark}
There is an alternative way to obtain a continuous family as in Theorem~\ref{thm_betafamily}\,\ref{betafamily3}, at least in the vicinity of a fixed background map, say $\beta_0$. Consider the vector field $\mathbb X_\sigma = \bar \mu \Xi_\sigma|_{S\DD}={\rm e}^{-\sigma}(X+\lambda_\sigma V)$ on $S\DD$, where $\lambda_\sigma(z,\mu)=\mu \partial_z\sigma - \bar\mu\partial_{\bar z}\sigma$. This generates the geodesic flow of ${\rm e}^{2\sigma}|{\rm d}z|^2$ when projecting it onto $S\DD$ (see Remark \ref{rmk_xpres}). Similar to \cite[Proposition~3.1]{BoPa23}, one can show for $\sigma\in \Sigma_\mathrm{s}$ that
\[
	\mathbb X_\sigma\colon\ \bigoplus_{k\ge 0}\Omega_k\rightarrow \bigoplus_{k\ge -1}\Omega_k\quad \text{is onto},
\]	
with a right inverse $R_\sigma$ obeying estimates $||R_\sigma(f)||_{H^{s}} \le C_s(\sigma)\cdot || f||_{H^{s+1}}$ in a Sobolev scale $(H^s(S\DD)\mid s\ge 0)$. Viewing $\beta_0$ as $\C^2$-valued function in $\bigoplus_{k\ge 0}\Omega_k\cap \ker X$, we set
\[
	\beta'_\sigma : = \beta_0 - R_\sigma({\rm e}^{-\sigma}\lambda_\sigma V\beta_0),\qquad \sigma\in \Sigma_\mathrm{s}.
\]
This is automatically fibrewise holomorphic and lies in the kernel of $\mathbb X_\sigma$, hence it extends to a~holomorphic function on $(Z,\D_\sigma)$. Moreover,
\[
	||\beta_\sigma' - \beta_0 ||_{H^s(S\DD)} \lesssim C_s(\sigma) ||{\rm e}^{-\sigma}\lambda_\sigma||_{H^{s+1}(S\DD)} = C_s(\sigma) \times o(1),\qquad \sigma\rightarrow 0 \qquad \text{in}\  C^\infty(\DD,\R).
\]
Showing continuity of the family $(\beta'_\sigma\mid \sigma\in \Sigma_\mathrm{s})$ at $\sigma=0$ thus requires control of the constants~$C_s(\sigma)$ in~$\sigma$. Getting this control is possible by slightly softer methods than used in the proof of Theorem~\ref{thm_betafamily}\,\ref{betafamily3}. However, the so-obtained $\beta$-maps are not canonical in any sense.

\end{Remark}

\subsubsection{Openness of the holomorphic blow-down structure} \label{s_openness}
Define the set
\[
	\mathcal I = \big\{(\beta,\sigma)\in C^\infty\bigl(Z,\C^2\bigr) \times \Sigma_\mathrm{s} \mid \Xi_\sigma \beta = \partial_{\bar \mu} \beta =0 \big\}.
\]
This is closed with respect to the ambient $C^\infty$-topology and contains as subset the collection of tuples with holomorphic blow-down structure
\[
	\mathcal I_+=\{(\beta,\sigma)\in \mathcal I \mid \text{$\beta$ has holomorphic blow-down structure with respect to~} \D_\sigma~\text{and}~\Om_\sigma\}.
\]

The following theorem is proved in Section~\ref{pf_openbds}.

\begin{Theorem}\label{thm_openbds} $\mathcal I_+\subset \mathcal I$ is open.
\end{Theorem}

The two preceding results (see Theorems~\ref{thm_betafamily} and~\ref{thm_openbds}) immediately imply Theorem~\ref{thm_globalbeta} for conformal perturbation. For details on this, we refer the reader to the discussion in Section~\ref{proof_globalbeta}, where also non-conformal perturbations are addressed.

\section[Explicit beta-maps for constant curvature models]{Explicit $\boldsymbol{\beta}$-maps for constant curvature models}\label{s_fra}

In this section, we show that for simple geodesic discs in constant curvature, explicit maps with holomorphic blow-down structure can be constructed (see Theorem~\ref{thm:betaCC} below), and that this map is canonical in the sense of Definition~\ref{def_betaextension} (see Theorem~\ref{thm:preimages} below). We first state the main results in Section~\ref{ssec:main} before covering the proofs of Theorems~\ref{thm:betaCC} and~\ref{thm:preimages} in Sections~\ref{ssec:proofbetaCC} and~\ref{ssec:proofpreimages}, respectively.

\subsection{Main results} \label{ssec:main}

Fix $\kappa\in \Rm$, and let $g_\kappa = {\rm e}^{2\sigma} |{\rm d}z|^2$ on $M = \Dm_R = \big\{z\in \Cm \mid |z|^2\le R^2\big\}$, where ${\rm e}^{2\sigma} = (1+\kappa z\zbar)^{-2}$. When $\big|\kappa R^2\big|<1$, $(\Dm_R, g_\kappa)$ is a simple geodesic disk of constant curvature $4\kappa${, with transport twistor space denoted $Z_{\kappa,R}\approx \Dm_R\times \Dm$}. The first main result of this section is an explicit construction of a blowdown map for any such surface.

\begin{Theorem}\label{thm:betaCC}
 For $\kappa\in \Rm$ and $R>0$ such that $\big|\kappa R^2\big|<1$, the map $\beta \equiv \beta^\kappa$ given by
 \begin{align}
	 \beta (z,\mu) = (w,\xi) = \left( \frac{z-\mu^2\zbar}{1+\kappa \zbar^2 \mu^2}, \mu \frac{1+\kappa z\zbar}{1+\kappa\zbar^2 \mu^2} \right), \qquad (z,\mu)\in Z_{\kappa,R},
	\label{eq:betakappa}
 \end{align}
 has a holomorphic blow-down structure.
\end{Theorem}

\begin{Remark}
	One can arrive at expression \eqref{eq:betakappa} by exploiting the rotation-invariance of $g_\kappa$. In this case, one seeks $\beta|_{SM}$ as an $\Sm^1$-equivariant (in the sense that $\beta\bigl({\rm e}^{{\rm i}t}z,{\rm e}^{{\rm i}t}\mu\bigr) = {\rm e}^{{\rm i}t} \beta(z,\mu)$), fiberwise holomorphic first integral, whose Fourier modes can be solved one-by-one by solving ordinary differential equations in the radial base coordinate.
\end{Remark}

The second main result of this section shows that for the surfaces considered in this section, the $\beta$ map of Theorem~\ref{thm:betaCC} arises precisely as the $\beta$-extension of the map $f(z) = z$ (in the sense of Definition~\ref{def_betaextension}), in this case without the need for the Szeg\H{o} projectors $\mathbb S_{\ev}$, $\mathbb S_{\mathrm{odd}}$.

\begin{Theorem}\label{thm:preimages}
	Fix $\kappa\in \Rm$ and $R>0$ such that $|\kappa| R^2<1$. Then the $\beta$ map defined in Theorem~{\rm\ref{thm:betaCC}} satisfies
	\begin{align*}
		\beta|_{SM} =\bigl(\bigl(I_0 N_0^{-1} z\bigr)^\sharp, \bigl(I_1 N_1^{-1} {\rm d}z\bigr)^\sharp\bigr).
	\end{align*}
\end{Theorem}

Such a result exploits the fact that in the case of rotationally symmetric metrics, more can be said about the Fourier modes of invariant distributions satisfying certain equivariance properties, see Lemma \ref{lem:vanishing} below.

\subsection[Proof of Theorem 4.1]{Proof of Theorem~\ref{thm:betaCC}} \label{ssec:proofbetaCC}

Since the metric is conformally Euclidean as in Section~\ref{s_conformallyeuclidean}, the twistor space $Z_{\kappa,R}$ of $(\Dm_R, g_\kappa)$ has involutive distribution spanned by $\partial_{\mubar}$ and $\Xi = \Xi_\sigma$ as defined equation~\eqref{Xisig}. With ${\rm e}^{-\sigma} = 1+\kappa z \zbar$ here, this gives
\[ 
	\Xi = (1+\kappa z\zbar) \left( \mu^2 \partial_z + \partial_{\zbar} - \kappa \frac{\zbar \mu^2-z}{1+\kappa z\zbar} (\mubar \partial_{\mubar} - \mu\partial_\mu) \right).	
\]

\begin{proof}[Proof of Theorem~\ref{thm:betaCC}] We verify the three conditions in Definition~\ref{def_bds}.

(i) \smash{$\beta|_{ {\partial_+ S\Dm_R}}$ is a totally real $C_\alpha^\infty$-embedding.} We first show that \smash{$\beta|_{{\partial_+ S\Dm^{g_\kappa}_R}}$} is injective. Evaluating the map $\beta = (w,\xi)$ at a point $(z,\mu) = \bigl(R {\rm e}^{{\rm i}\omega}, {\rm e}^{{\rm i}(\omega+\alpha+\pi)}\bigr)$ in $\partial_+ S\Dm_R^{g_\kappa}$ (i.e., with $\omega\in\Sm^1$ and $\alpha\in [-\pi/2,\pi/2]$), a direct calculation gives
	\begin{align}
		(w,\xi) \bigl(R {\rm e}^{{\rm i}\omega}, {\rm e}^{{\rm i}(\omega+\alpha+\pi)}\bigr)= \left( R {\rm e}^{{\rm i}\omega} \frac{1-{\rm e}^{2{\rm i}\alpha}}{1+\kappa R^2 {\rm e}^{2{\rm i}\alpha}}, {\rm e}^{{\rm i}(\omega+\alpha+\pi)} \frac{1+\kappa R^2}{1+\kappa R^2 {\rm e}^{2{\rm i}\alpha}} \right).
		\label{eq:wxi_vertex}
	\end{align}
	Thus $w/\xi = \frac{2{\rm i}R\sin\alpha}{1+\kappa R^2}$ uniquely determines $\alpha$, and then $\omega$ is determined by either relation in \eqref{eq:wxi_vertex}.

Next we verify that \smash{$\beta|_{ {\partial_+ S\Dm_R^{g_\kappa}}}$} is a totally real $C_\alpha^\infty$-embedding, and we first discuss the smooth structure near the glancing. Observe that the scattering relation is the map $({\rm e}^{{\rm i}\omega},{\rm e}^{{\rm i}\alpha})\mapsto \bigl({\rm e}^{{\rm i}(\omega+\pi+2\ss(\alpha))}, {\rm e}^{{\rm i}(\pi-\alpha)}\bigr)$ where \smash{$\ss(\alpha) := \tan^{-1} \bigl( \frac{1-\kappa R^2}{1+\kappa R^2} \tan\alpha\bigr)$}, see, e.g., \cite[Section~4.1]{Monard2020}. Then two coordinates which are respectively odd and even with respect to the scattering relation in a~neighborhood of the glancing are $u_1 = \cos\alpha$ and $u_2 = \omega + \ss(\alpha)$, and hence the $C_\alpha^\infty$ smooth structure is precisely made of functions that are smooth in $u_1^2$ and $u_2$ in a neighbourhood of~$\partial_0 S\Dm_R^{g_\kappa}$. Then the totally real $C_\alpha^\infty$-embedding property is fulfilled by checking that
	\begin{align*}
		\det \left(\frac{\partial\beta}{\partial (u_1)^2}, \frac{\partial\beta}{\partial u_2}, J \frac{\partial\beta}{\partial (u_1)^2}, J\frac{\partial\beta}{\partial u_2}\right) = \frac{1}{4 u_1^2}\det \left(\frac{\partial\beta}{\partial u_1}, \frac{\partial\beta}{\partial u_2}, J \frac{\partial\beta}{\partial u_1}, J \frac{\partial\beta}{\partial u_2}\right)
	\end{align*}
	is nowhere vanishing near $\partial_0 S\Dm_R^{g_\kappa}$ (along with $\beta$ having a non-vanishing Jacobian at interior points). With $\frac{\partial}{\partial u_2} = \frac{\partial}{\partial \omega}$ and $\frac{\partial}{\partial u_1} = \frac{-1}{\sin\alpha} \bigl(\frac{\partial}{\partial\alpha} - \ss'(\alpha) \frac{\partial}{\partial\omega}\bigr)$, this amounts to showing that
	\begin{align*}
		\frac{1}{(\cos\alpha)^2}\det \left(\frac{\partial\beta}{\partial\omega}, \frac{\partial\beta}{\partial\alpha}, J \frac{\partial\beta}{\partial\omega}, J\frac{\partial\beta}{\partial\alpha}\right) (\omega,\alpha) \ne 0, \qquad \text{near }\ \partial_0 S\Dm_R^{g_\kappa}.
	\end{align*}
	To proceed, we write $\frac{\partial\beta}{\partial\omega} = V_1 + \overline{V_1}$ and $\frac{\partial\beta}{\partial\alpha} = V_2 + \overline{V_2}$, where
	\begin{align*}
		&V_1:= \frac{\partial w}{\partial\omega}\partial_w + \frac{\partial\xi}{\partial\omega} \partial_\xi = \frac{1}{1+\kappa R^2 {\rm e}^{2{\rm i}\alpha}} \bigl( R{\rm i} {\rm e}^{{\rm i}\omega} \bigl(1-{\rm e}^{2{\rm i}\alpha}\bigr) \partial_\omega + {\rm i} {\rm e}^{{\rm i}(\omega+\alpha+\pi)} \bigl(1+\kappa R^2\bigr) \partial_\xi\bigr), \\
		&V_2:= \frac{\partial w}{\partial\alpha}\partial_w + \frac{\partial\xi}{\partial\alpha} \partial_\xi = \frac{1+\kappa R^2}{\bigl(1+\kappa R^2 {\rm e}^{2{\rm i}\alpha}\bigr)^2} \bigl(-2{\rm i}R {\rm e}^{{\rm i}(\omega+2\alpha)} \partial_\omega + {\rm i} {\rm e}^{{\rm i}(\omega + \alpha+\pi)} \bigl(1-\kappa R^2 {\rm e}^{2{\rm i}\alpha}\bigr) \partial_\xi \bigr).
	\end{align*}
	Then
	\begin{align*}
		\det \left(\frac{\partial\beta}{\partial\omega}, \frac{\partial\beta}{\partial\alpha}, J \frac{\partial\beta}{\partial\omega}, J\frac{\partial\beta}{\partial\alpha}\right) &= \det \bigl(V_1 + \overline{V_1}, V_2 + \overline{V_2}, {\rm i} (V_1-\overline{V_1}), {\rm i} \bigl(V_2-\overline{V_2}\bigr)\bigr) \\
		&= -4 \det \bigl(V_1, V_2, \overline{V_1}, \overline{V_2}\bigr) = -4 \left| \frac{\partial w}{\partial\omega}\frac{\partial\xi}{\partial\alpha} - \frac{\partial w}{\partial\alpha} \frac{\partial\xi}{\partial\omega} \right|^2 \\
		&= -4 \frac{\bigl(1+\kappa R^2\bigr)^2}{\big|1+\kappa R^2 {\rm e}^{2{\rm i}\alpha}\big|^4} R^2 \big|1+{\rm e}^{2{\rm i}\alpha}\big|^2.
	\end{align*}
	Upon dividing both sides by $(\cos\alpha)^2$, we arrive at the result. The same calculation shows that~$\beta|_{(\partial_+ S\Dm_R^{g_\kappa})^\circ}$ is a local embedding at interior points.

(ii) \smash{$\beta|_{Z_{\kappa,R}^\circ}$} is a biholomorphism onto its image. That $\partial_{\mubar} w = \partial_{\mubar} \xi = 0$ is immediate, while $\Xi w = \Xi \xi = 0$ is a tedious, yet direct, calculation. Hence \smash{$\beta|_{Z_{\kappa,R}^\circ}$} is holomorphic. To describe~$\beta(Z_{\kappa,R}^\circ)$ and invert it there, we will make use of the elementary map
	\[ \beta_{\mathbb{P}} \colon \ \Dm^\circ\times \Dm^\circ\ni (\zeta,T) \mapsto \bigl(\zeta-T\bar{\zeta}, T\bigr) \in \Cm^2, \]
	easily seen to be injective and with image
	\[ \beta_{\mathbb{P}} (\Dm^\circ\times \Dm^\circ) = \left\{ (z_1,z_2) \in \Cm^2 \mid \left| \frac{z_1 + \overline{z_1}z_2}{1-|z_2|^2}\right| <1,\, |z_2|<1 \right\}. \]
	We now observe the following commutative diagram\footnote{The vertical map is the $2:1$ branched cover from transport- to projective twistor space of $(\Dm_R,g_\kappa)$, the map $z\mapsto \zeta$ is the diffeomorphism capturing the projective equivalence between $(\Dm_R,g_\kappa)$ and $(\Dm,g_0)$, inducing a biholomorphism between their projective twistor spaces given by the map $(z,\omega)\mapsto (\zeta,T)$. Finally, $\beta_{\mathbb{P}}$ is the analogue of beta maps to the projective twistor space of the Euclidean unit disk.}
	\[
	 	\begin{tikzcd}
			Z_{\kappa,R}^\circ \arrow["{\beta}"]{r} \arrow["{(z,\mu)\mapsto (z,\mu^2)}"]{d} & \Cm^2 \arrow[dr, bend left, "{\Psi\colon (w,\xi)\mapsto \bigl(\frac{1-\kappa R^2}{R}w, \xi^2 + \kappa w^2\bigr)}"] & {} \\
	 		\Dm_R^\circ \times \Dm^\circ \arrow[r, "\simeq", "{(z,\omega)\mapsto (\zeta,T)}"{yshift=-3ex}] & \Dm^\circ \times \Dm^\circ \arrow[r, "1:1", "{\beta_{\mathbb{P}}}"{yshift=-3ex}] & \Cm^2,
	 	\end{tikzcd}
	\]
	where \smash{$(\zeta,T) = \bigl(\frac{1-\kappa R^2}{1-\kappa |z|^2} \frac{z}{R}, \frac{\kappa z^2 + \omega}{1+\kappa\zbar^2 \omega}\bigr)$}. We then deduce that $\Psi(\beta(Z_{\kappa,R}^\circ)) = \beta_{\mathbb{P}}(\Dm^\circ\times \Dm^\circ)$, and since the range of $\beta$ is stable under $(w,\xi)\mapsto (w,-\xi)$ by the parity properties of $(w,\xi)$, we conclude that
	\begin{align*}
		\begin{split}
			\beta(Z_{\kappa,R}^\circ) &= \big\{(w,\xi)\in \Cm^2 \mid \Psi(w,\xi)\in \beta_{\mathbb{P}} (\Dm^\circ\times \Dm^\circ)\big\}	 \\
			&= \left\{ (w,\xi)\in \Cm^2 \mid \left|\frac{w+\bar{w}\bigl(\kappa w^2 + \xi^2\bigr)}{1-\big|\kappa w^2 + \xi^2\big|^2}\right| < \frac{R}{1-\kappa R^2}, \, \big|\kappa w^2 + \xi^2\big|<1 \right\}.
		\end{split}
	\end{align*}
	Now for $(w,\xi)\in \beta(Z_{\kappa,R}^\circ)$, $\Psi(w,\xi)$ uniquely determines $\bigl(z,\mu^2\bigr)\in \Dm_R^\circ\times \Dm^\circ$, and the relation $\xi = \mu (1+\kappa w\zbar)$ (which can be established directly from \eqref{eq:betakappa}), along with the fact that
	\begin{align*}
		1 + \kappa w \zbar \stackrel{\eqref{eq:betakappa}}{=} \frac{1+\kappa |z|^2}{1+\kappa \zbar^2 \mu^2} \ne 0, \qquad \text{since} \  |\kappa z|^2\le\kappa R^2<1,
	\end{align*}
	 uniquely determines $\mu$.

(iii) $\beta^* \Omega_{\Cm^2} \ge c \underline{\Omega}$ for some constant $\boldsymbol{c>0}$. Recall that $\underline{\Omega} ={\rm i} (\eta_1\wedge \bar{\eta}_1 + \eta_2 \wedge \bar{\eta}_2)$, where
	\begin{align*}
		\eta_1 = \bigl(1-|\mu|^4\bigr) \overline{\Xi}^\vee = \frac{{\rm d}z-\mu^2 {\rm d}\zbar}{1+\kappa z\zbar}, \qquad \eta_2 = \partial_\mu^\vee = {\rm d}\mu + \frac{\kappa\mu}{1+\kappa z\zbar} (z{\rm d}\zbar-\zbar {\rm d}z).
	\end{align*}
	Denoting $(w,\xi) = \beta(z,\mu)$ complex coordinates on the image of $\beta$, for which $\Omega_{\Cm^2} = {\rm i} \bigl({\rm d}w\wedge {\rm d}\bar{w} + {\rm d}\xi \wedge {\rm d}\bar{\xi}\bigr)$, we have
	\begin{align*}
		&\beta^* {\rm d}w
= \frac{1+\kappa z\zbar}{\bigl(1+\kappa\zbar^2 \mu^2\bigr)^2} \bigl( \bigl(1-\kappa \zbar^2\mu^2\bigr) \eta_1 - 2\mu\zbar \eta_2 \bigr), \\
		&\beta^* {\rm d}\xi
= \frac{1+\kappa z\zbar}{\bigl(1+\kappa\zbar^2 \mu^2\bigr)^2} \bigl( 2\kappa\mu\zbar \eta_1 + \bigl(1-\kappa\zbar^2\mu^2\bigr) \eta_2 \bigr).
	\end{align*}
	Writing the above succinctly as
$\beta^* {\rm d}w = a\eta_1 + b\eta_2$, $ \beta^* {\rm d}\xi = c\eta_1 + d\eta2$,
	we have $\beta^* \Omega_{\Cm^2} = {\rm i}\sum_{j,k=1}^2 H_{jk} \eta_j \wedge \bar{\eta}_k$, where
	\begin{align*}
		H_{11} = |a|^2 + |c|^2, \qquad H_{12} = \overline{H_{21}} = a \bar{b} + c \bar{d}, \qquad H_{22} = |b|^2 + |d|^2.
	\end{align*}
	Then minorizing $\beta^* \Omega_{\Cm^2}$ by $\underline{\Omega}$ is done by minorizing the smallest eigenvalue of the matrix $H = \{H_{jk}\}$, given by
	\begin{align*}
		\lambda_{\min} = \frac{\operatorname{tr} H}{2} - \sqrt{ \left(\frac{\operatorname{tr} H}{2}\right)^2 - \det H} = \frac{\det H}{\frac{\operatorname{tr} H}{2} + \sqrt{ \bigl(\frac{\operatorname{tr} H}{2}\bigr)^2 - \det H}} \ge \frac{\det H}{\operatorname{tr} H}.
	\end{align*}
	We now compute
	\begin{gather*}
		\det H = |ad-bc|^2 = \frac{(1+\kappa z\zbar)^4}{\big|1+\kappa\zbar^2\mu^2\big|^8} \big|\bigl(1-\kappa\zbar^2\mu^2\bigr)^2 + 4\kappa \mu^2 \zbar^2\big|^2 = \frac{(1+\kappa z\zbar)^4}{\big|1+\kappa\zbar^2 \mu^2\big|^4}, \\
		\operatorname{tr} H = |a|^2 + |b|^2 + |c|^2 + |d|^2 = 2 \frac{(1+\kappa z\zbar)^2}{\big|1+\kappa\zbar^2 \mu^2\big|^4} \big[ \big|1-\kappa\zbar^2\mu^2\big|^2 + 2|\zbar\mu|^2 \bigl(1+\kappa^2\bigr) \big],
	\end{gather*}
	then
	\begin{align*}
		\lambda_{\min} \ge \frac{\det H}{\operatorname{tr} H} = \frac{(1+\kappa z\zbar)^2}{2\big|1-\kappa\zbar^2\mu^2\big|^2 + 4 |\zbar\mu|^2 \bigl(1+ \kappa^2\bigr)} \ge \frac{\bigl(1-|\kappa| R^2\bigr)^2}{8+4 R^2+4|\kappa|},
	\end{align*}
	where we have used that $|z|\le R$, $|\mu|\le 1$ and $|\kappa|R^2<1$. The last right-hand side is a suitable constant $c$ and item (iii) is proved.	
\end{proof}

\subsection[Proof of Theorem 4.3]{Proof of Theorem~\ref{thm:preimages}} \label{ssec:proofpreimages}

Suppose $(M,g) = \bigl(\Dm, {\rm e}^{2\sigma} |{\rm d}z|^2\bigr)$, where the conformal factor $\sigma\in C^\infty(\Dm)$ is rotationally invariant in the sense that $\sigma({\rm e}^{{\rm i}t} z) = \sigma(z)$ for all $t\in \R$. This $\mathbb S^1$-action lifts to $Z\cong \Dm_z\times \Dm_\mu$ and we consider smooth first integrals on $SM\cong \Dm\times \mathbb S^1$ that satisfy an equivariance property of the form
\begin{align}
	u \bigl({\rm e}^{{\rm i}t}z, {\rm e}^{{\rm i}t}\mu\bigr) = {\rm e}^{{\rm i}pt} u (z,\mu), \qquad (z,\mu)\in \Dm\times \mathbb S^1,\qquad t\in \R,
	\label{eq:pequiv}
\end{align}
for some $p\in \Z$. This property corresponds to an equivariance of the Fourier modes. Indeed, if $u\bigl(z,{\rm e}^{{\rm i}\theta}\bigr) = \sum_{k\in \Zm} \tu_k(z) {\rm e}^{{\rm i}k\theta}$ satisfies \eqref{eq:pequiv}, then
\begin{align*}
	0 = u \bigl({\rm e}^{{\rm i}t}z, {\rm e}^{{\rm i}(t+\theta)}\bigr) - {\rm e}^{{\rm i}pt} u \bigl(z,{\rm e}^{{\rm i}\theta}\bigr) = \sum_{k\in \Zm} \bigl( \tu_k\bigl({\rm e}^{{\rm i}t}z\bigr) {\rm e}^{{\rm i}kt} - {\rm e}^{{\rm i}pt} \tu_k(z) \bigr) {\rm e}^{{\rm i}k\theta},
\end{align*}
which by uniqueness of Fourier series gives
\begin{align}
	\tu_k\bigl({\rm e}^{{\rm i}t}z\bigr) = {\rm e}^{{\rm i}(p-k)t} \tu_k(z), \qquad t\in \R,\quad z\in \Dm,\quad k\in \Zm.
	\label{eq:modesEquiv}
\end{align}
For equivariant first integrals, this implies the following vanishing result.
\begin{Lemma}\label{lem:vanishing} Suppose $(M,g) = \bigl(\Dm, {\rm e}^{2\sigma} |{\rm d}z|^2\bigr)$ is rotationally invariant, and suppose that $u\in C^\infty(SM)$ is a first integral satisfying \eqref{eq:pequiv} for some $p\in \Zm$. If $\tu_k= 0$ for all $k\le p$, then $u\equiv 0$.
\end{Lemma}

\begin{proof}By contradiction, if $u\ne 0$, let $\ell := \min\{k\in \Zm \mid u_k\ne 0\}>p$. Then by \eqref{eq:modesEquiv}, we have that $\tu_\ell \bigl({\rm e}^{{\rm i}t}z\bigr) = {\rm e}^{{\rm i}(p-\ell)t} \tu_\ell(z)$. {Using the splitting $X = \eta_+ + \eta_-$ (where $\eta_\pm$ are the Guillemin--Kazhdan operators, see, e.g., \cite[Definition~6.1.4]{PSU23}), taking the Fourier mode of degree $\ell-1$ of the equation $Xu =0$ gives $\eta_- (\tu_\ell (z) {\rm e}^{{\rm i}\ell \theta}) = 0$. In the current coordinates, \cite[Lemma~6.1.8]{PSU23} gives $\eta_- = {\rm e}^{-\sigma} {\rm e}^{-{\rm i}\theta} (\partial_{\bar{z}} - {\rm i}\partial_{\bar{z}}\sigma \partial_\theta)$, and thus setting $h = {\rm e}^{\ell\sigma} \tu_\ell$, the last equation is equivalent to~${\partial_{\zbar} h = 0}$}, hence $h$ has a convergent power series in $z$ in a neighbourhood of $z=0$. Since ${\rm e}^{\ell\sigma}$ is $\Sm^1$-invariant, $h$ also satisfies $h\bigl({\rm e}^{{\rm i}t}z\bigr) = {\rm e}^{{\rm i}(p-\ell)t} h(z)$. Upon differentiating with respect to $t$ and setting $t=0$, we obtain
	\begin{align*}
		{\rm i}z h'(z) ={\rm i} (p-\ell) h(z) \quad \implies \quad \partial_z \bigl(z^{\ell-p} h(z)\bigr) = 0.
	\end{align*}
	Hence $z^{\ell-p} h(z)$ is constant, and since $\ell>p$, this constant must be zero in order for $h$ to be well-defined at $z=0$. Hence $h\equiv 0$ and $\tu_\ell=0$, a contradiction.
\end{proof}

\begin{Corollary}\label{cor:noSzego} If $(M,g) = \bigl(\Dm,{\rm e}^{2\sigma}|{\rm d}z|^2\bigr)$ is rotationally symmetric and simple, then the $\beta$\nobreakdash-ex\-ten\-sion of the holomorphic embedding $f(z)=z$ satisfies
	\begin{align*}
		\bigl(I_0 N_0^{-1} z\bigr)^\sharp = \Sm_{\mathrm{ev}} \bigl(\bigl(I_0 N_0^{-1} z\bigr)^\sharp \bigr), \qquad \text{and}\qquad \bigl(I_1 N_1^{-1} {\rm d}z\bigr)^\sharp = \Sm_{\mathrm{odd}} \bigl(\bigl(I_1 N_1^{-1} {\rm d}z\bigr)^\sharp \bigr).
	\end{align*}
\end{Corollary}

This tells us that in rotationally symmetric situations, introducing the Szeg\H{o} projectors for the definition of the $\beta$-extension is not needed---it is unclear, whether this is generally true for simple surface.

\begin{proof}[Proof of Corollary \ref{cor:noSzego}]
	{The rotation invariance induces an $\Sm^1$-equivariance of the geodesic flow, and one thus notices that the constructed functions} \smash{$w = \bigl(I_0 N_0^{-1} z\bigr)^\sharp$} and \smash{$\xi = \bigl(I_1 N_1^{-1} {\rm d}z\bigr)^\sharp$} satisfy \eqref{eq:pequiv} with $p=1$. On to the study of $w$, $w$ is additionally fiberwise even, and since $\eta_- w_0 = 0$, then $\eta_+ w_{-2} = 0$ as well, and~$w$ splits into the sum of two smooth first integrals~${f = \Sm_{\mathrm{ev}} w}$ and~${h = w_{-2} + w_{-4} + \cdots}$. To show that $h\equiv 0$, notice that $\bar{h}$ is a smooth first integral satisfying~\eqref{eq:pequiv} with $p=-1$, and such that $\bigl(\bar{h}\bigr)_k = 0$ for all $k\le 1$ (in particular for all~${k\le -1}$). By Lemma \eqref{lem:vanishing}, this forces $h=0$. Similarly for $\xi$, $\xi$ is fiberwise odd, with~${\eta_- \xi_1 = 0}$, and hence~${\xi = \Sm_{\mathrm{odd}} \xi + h}$, where $\bar{h}$ is a smooth first integral satisfying \eqref{eq:pequiv} with~${p=-1}$, and such that $\bigl(\bar{h}\bigr)_k = 0$ for all $k\le 0$ (in particular for all $k\le -1$). By Lemma \eqref{lem:vanishing}, this again forces~${h=0}$.	
\end{proof}

Finally, we prove Theorem~\ref{thm:preimages}.

\begin{proof}[Proof of Theorem~\ref{thm:preimages}]
	To show that $\beta|_{SM} = \bigl(\bigl(I_0 N_{0}^{-1}z\bigr)^\sharp,\bigl(I_1 N_{1}^{-1}{\rm d}z\bigr)^\sharp\bigr)$ (removing the Szeg\H{o} projectors using Corollary \ref{cor:noSzego}), notice that both maps satisfy \eqref{eq:pequiv} with $p=1$ {as explained in the previous proof}, and they agree up to Fourier modes $k=1$. Thus their difference is a~smooth first integral satisfying \eqref{eq:pequiv} with $p=1$, and with vanishing modes up to $k=p=1$. By Lemma~\ref{lem:vanishing}, this difference is identically zero.
\end{proof}

\section[Metric dependence. Part 1: The Upsilon-map and consequences]{Metric dependence. Part 1: The $\boldsymbol{\Upsilon}$-map and consequences}\label{s_upsilon}

Let $(M,g)$ be a compact $d$-dimensional Riemannian manifold $(d\ge 2)$ that is non-trapping and has strictly convex boundary. In order to analyse the geodesic X-ray transform on functions with boundary blow-up as well as conjugacies between geodesic flows, it is convenient to consider the following map, introduced in \cite{MNP20}
\[
	\Upsilon\colon\ \partial_+SM\times [0,1]\rightarrow SM,\qquad \Upsilon(x,v,u)=\varphi_{u\tau(x,v)}(x,v),
\]
The $\Upsilon$-map sends the fibre $\{(x,v)\}\times [0,1]$ either to an orbit of the geodesic flow (if $(x,v)\notin \partial_0SM$) or collapes it to a point (if $(x,v) \in \partial_0SM$).
The purpose of this section is to show that the $\Upsilon$-map depends continuously on the underlying metric. From this, further continuity statements about the geodesic X-ray transform and conjugacies between geodesic flows are derived---these are used both in Sections~\ref{s_verstappen} and~\ref{s_bds}.

\subsection{Set-up for perturbation theory}
Let us define
$\mathbb M_\ntc \subset C^\infty\bigl(M,\otimes_S^2T^*M\bigr)
$ as the space of all Riemannian metrics $g$ such that $(M,g)$ is non-trapping and $\partial M$ is strictly convex. This is an open set of the ambient $C^\infty$-space and we equip it with the subspace topology. It will we useful to embed $M$ into a closed manifold~$N$ of the same dimension and extend all metrics to $N$. By Lemma \ref{lem_rext} (and Remark \ref{lem_rext}), this extension can be implemented by a~continuous map
\begin{equation}\label{extensionE}
 E\colon\ \mathbb M_\ntc \rightarrow C^\infty\bigl(N,\otimes^2_S T^*N\bigr), \qquad Eg\vert_{M}=g.
\end{equation}
To simplify notation, the extended metric is also written as $g$.
Fix a background metric $g_0\in \M_\ntc$ and denote with $SM\equiv SM_{g_0}$ (resp.~$SN\equiv SN_{g_0}$) the corresponding unit tangent bundle.
The rescaling map
\[
	\phi_g\colon\ SN\xrightarrow{\sim} SN_g,\qquad (x,v)\mapsto (x,v/|v|_g),\qquad g\in \M_\ntc,
\]
can be used to transfer all objects of interest to a fixed manifold.
For instance, if $(\varphi^g_t)$ denotes the $g$-geodesic flow on $SN_g$, we define a rescaled flow by
\begin{equation*}
\psi^g\in C^\infty(SN\times \R,SN),\qquad \psi_t^g=\phi_g^{-1}\circ \varphi_t^g\circ \phi_g.
\end{equation*}
The infinitesimal generator of $\psi^g_t$ will be denoted with $\mathbb X_g \in C^\infty(SN,T(SN))$.

\begin{Remark}\label{rmk_xpres}
If $d=2$, then one can show that $\mathbb X_g = a_g(X+\lambda_g V)$ for smooth functions $a_g,\lambda_g\in C^\infty(SM,\R)$, with $\lambda_g$ having non-zero Fourier modes only in the degrees $\pm 1$ and $\pm 3$ (see \cite[Section 3.1]{MePa20}). For us, this is only relevant for $g_0=|{\rm d}z|^2$ and $g={\rm e}^{2\sigma}|{\rm d}z|^2$ on $\DD$, where in the standard coordinate representation \cite[Lemma~3.5.6]{PSU23} one reads off that $a_g = {\rm e}^{-\sigma}$ and~${\lambda_g= \mu \partial_z \sigma - \bar \mu \partial_{\bar z}\sigma}$ on $S\DD = \DD\times \mathbb S^1$.
\end{Remark}

For $(x,v)\in SM_g$, the {\it exit times} of the $g$-geodesic flow are $\tau_\pm^g(x,v)=\pm \sup\{t\ge 0: \varphi^g_{\pm s}(x,v)\in SM_g \text{ for } 0\le s \le t\}$. Their sum defines a smooth function on $SM_g$ and we write its pull-back to $SM$ as
$\tilde \tau_g = \bigl(\tau_+^g+\tau_-^g\bigr)\circ \phi_g \in C^\infty(SM,\R)$.
Additionally define $m_g\in C^\infty(\partial SM,\R)$ by~${m_g(x,v)=g(\nu(x),v)}$, where $\nu$ is the inward pointing unit normal for $g$.

\subsubsection[Upsilon-map revisited]{$\boldsymbol{\Upsilon}$-map revisited}We can now re-define the $\Upsilon$-map as
\[
	\Upsilon_g\colon \ \partial_+SM \times [0,1]\rightarrow SM,\qquad \Upsilon_g(x,v,u)=\psi^g_{u\tilde\tau_g(x,v)}(x,v)
\]
and recall from \cite{MNP20} the following result.

\begin{Proposition}[properties of the $\Upsilon$-map]\label{propups} For all $g\in \M_\ntc$, we have
	\begin{enumerate}[label=\rm (\roman*)]\itemsep=0pt
	\item $\Upsilon_g$ is smooth and satisfies $(\Upsilon_g)_*\bigl(\tilde\tau_g^{-1}\partial_u\bigr) = \mathbb X_g$;
	
	\item $\Upsilon_g$ defines a diffeomorphism from $\partial_+SM\backslash \partial_0SM$ onto $SM\backslash \partial_0SM$ and descends to a homeomorphism
	\[
		\tilde \Upsilon_g\colon\ \partial_+SM\times [0,1]/{\sim} \rightarrow SM,
	\]
	where the equivalence relation $\sim$ identifies two points $(x,v,u)$ and $(x',v',u')$ if and only if~${(x,v)=(x',v')\in \partial_0SM}$;
	
	\item if $\rho\in C^\infty(M)$ is a boundary defining function $($with lift to $SM$ denoted by the same symbol$)$, then there is a smooth, positive function $F_g\colon \partial_+SM\times [0,1]\rightarrow (0,\infty)$ such that
	\[
	\rho \circ\Upsilon_g =F_g
	\cdot u(1-u)\tilde \tau_g^2\qquad \text{on} \  \partial_+SM\times [0,1].
	\]
	\end{enumerate}
\end{Proposition}

\begin{proof}
This immediately follows from the corresponding assertions for $\Upsilon_g'=\phi_g\circ \Upsilon_g\circ \bigl(\phi_g^{-1}\times \operatorname{Id}\bigr)$ and $F_g'=F_g\circ \bigl(\phi_g^{-1}\times \operatorname{Id}\bigr)$, which were proved in \cite{MNP20}.
\end{proof}

\subsubsection{X-ray transform revisited}
The X-ray transform of the $\bigl(\psi^g_t\bigr)$-flow is given by
\[
	I_g\colon\ \rho^{-1/2}C^\infty(SM)\rightarrow C^\infty(\partial_+SM),\qquad I_g f(x,v)=\int_0^{\tilde \tau_g(x,v)} f\circ \psi_t(x,v) {\rm d}t.
\]
This is well-defined in view of the following representation, which is a direct consequence of Proposition~\ref{propups}
\begin{equation}\label{igrep}
		I_g \bigl(\rho^{-1/2}h\bigr) = \int_0^1
		\frac{\Upsilon_g^*h}{F_g^{1/2}}\cdot \frac{{\rm d}u}{\sqrt{u(1-u)}},\qquad h\in C^\infty(SM).
\end{equation}
Further, the extension to $(\psi_t)$-invariant functions is defined by
\[
	\sharp_g\colon\ C^\infty(\partial_+SM)\rightarrow C(SM),\qquad h^{\sharp_g}(x,v) = h\circ \tilde \Upsilon_g^{-1}(x,v),
\]
where $h$ is viewed as $u$-independent function $C^\infty(\partial_+SM\times [0,1])$.

\subsection{Collection of continuous dependences}
In the just described setting, all objects depend continuously on the metric.
\begin{Proposition}\label{thm_continuitygalore}
	All of the following maps are well-defined and continuous:
	\begin{alignat}{3}
 & g\mapsto \psi^g, \qquad&& \M_\ntc\rightarrow C^\infty(SN\times \R,SN);& \tag{i}\\
& g\mapsto \tilde \tau_g, \qquad&& \M_\ntc\rightarrow C^\infty(SM,\R);& \tag{ii}\\
& g\mapsto \tilde \tau_g/m_g, \qquad&& \M_\ntc\rightarrow C^\infty(\partial SM,\R);&\tag{iii}\\
& g\mapsto \Upsilon_g, \qquad&& \M_\ntc\rightarrow C^\infty(\partial_+SM\times [0,1],SM);&\tag{iv}\\
& g \mapsto F_g, \qquad&& \M_\ntc\rightarrow C^\infty(\partial_+SM\times [0,1],\R_{>0});& \tag{v}\\
&g \mapsto I_g, \qquad&& \M_\ntc\rightarrow \mathcal L\bigl(\rho^{-1/2}C^\infty(SM),C^\infty({\partial_+SM)}\bigr);&\tag{vi}\\
& g\mapsto \sharp_g\circ I_g, \qquad&&  \M_\ntc\rightarrow \mathcal L\bigl(\rho^{-1/2}C^\infty(SM),C^\infty(SM)\bigr). & \tag{vii}
\end{alignat}		
Here $\mathcal L(\cdot,\cdot)$ is the space of continuous linear operators, equipped with the topology of pointwise convergence.
\end{Proposition}

Checking continuity in each case is a tedious exercise that offers only little insight, so we confine ourselves to a brief sketch.

\begin{proof}[Sketch of the proof] The local coordinate expression of $\mathbb X_g\in C^\infty(N,TN)$ is a rational function in the components of $g$ and their derivatives and as a consequence the map $g\mapsto \mathbb X_g$ is continuous from $\M_\ntc$ into $C^\infty(N,TN)$. Combining this with Theorem~\ref{thm_isotopies} thus yields item~(i).
The arguments in \cite[Section~3.2]{PSU23} carry over to show that $\tilde \tau_g$ and $\tilde \tau_g/m_g$ lie in the claimed $C^\infty$-spaces, and closer inspection reveals that as such they depend continuously on the flow~$(\psi^g)$. From this also (ii), (iii) and (iv) follow.

Statement (v) follows from the expression of $F_g$ in terms of the $g$-Hessian of~$\rho$, the geodesic flow and $\tilde \tau_g$ that is discussed in the proof of \cite[Proposition~6.12]{MNP20}. By means of \eqref{igrep}, this can then be leveraged to (vi).

Finally, to see continuity in (vii) we need to consider also the scattering relation $\alpha_g\in \Diff_+(\partial SM)$, defined by $\alpha_g(x,v)=\Upsilon_g(x,v,1)$ for $(x,v)\in \partial_+SM$ and otherwise through $\alpha_g\circ \alpha_g = \Id$. For $h\in C^\infty(\partial_+SM)$, we define $A_+^gh\in C^0(\partial SM)$ by $A^g_+h = h$ on $\partial_+SM$ and $A_+^g h = h\circ \alpha_g$ on $\partial_-SM$. Let
\[
	\mathcal X = \{(g,h)\in \M_\ntc \times C^\infty(\partial_+SM)\mid A_+^g h \in C^\infty(\partial SM)\}
\]
and note that this is a closed subspace of $\M_\ntc\times C^\infty(\partial_+SM)$ (smoothness of $A_+^g h$ is checked by compatibility conditions on $\partial_0SM$ and hence is a closed property). Given \smash{$f\in \rho^{-1/2}C^\infty(SM)$}, the map in (vii) can then be factored into
$g\mapsto (g,I_g f)$, $\M_\ntc \rightarrow \mathcal X$ and $ (g,h)\mapsto h^{\sharp_g}$, $\mathcal X\rightarrow C^\infty(SM)$.
By \cite[Proposition 6.13]{MNP20}, the first map indeed lands in $\mathcal X$ and as this space is closed, continuity follows immediately from (vi). The fact that $h^{\sharp_g}$ is smooth is a classical result (see, e.g., \cite[Theorem 5.1.1]{PSU23}) and continuity of the map $(g,h)\mapsto h^{\sharp_g}$ is readily verified by an inspection of its proof.
\end{proof}

\subsection{Conjugacies} The geodesic flows of any two metrics in $\M_\ntc$ are smoothly conjugate up to a time change. We first show this for metrics in the closed subset
\begin{equation}\label{def_mntck}
	\M_\ntc(K) = \{g\in \M_\ntc\mid \supp(g-g_0)\subset K\},\qquad K\subset M^\circ \text{ compact}.
\end{equation}
Here the canonical choice of conjugacy (flow back along the orbits of $g_0$ and then forward along the orbits of $g$) is smooth and can be written in terms of $\Upsilon$-maps.

\begin{Proposition}
	Let $K\subset M^\circ$ be compact and $g\in \M_\ntc(K)$. Then there exists a unique solution $(\Phi_g,a_g)\in \Diff_+(SM)\times C^\infty(SM,\R)$ to
	\begin{equation}\label{conjugacy1}
		\Phi_*\mathbb X_{g_0} = a\mathbb X_{g}, \qquad \mathbb X_g a = 0\quad \text{on} \quad SM\qquad \text{and} \qquad \Phi = \Id\quad \text{on}\quad \partial_+SM.
	\end{equation}
	For $g=g_0$, we have $(\Phi,a)=(\Id,1)$ and moreover, the following map is continuous:
	\[g\mapsto (\Phi_g,a_g),\qquad \M_\ntc(K) \rightarrow \Diff_+(SM)\times C^\infty(SM,\R).\]
\end{Proposition}

\begin{proof}
Suppose $(\Phi,a)$ is a solution to \eqref{conjugacy1} and $(x,v)\in \partial_+SM$. Then $\Phi$ necessarily maps the $g_0$-orbit emerging from $(x,v)$
to the corresponding $g$-orbit, with a constant time change $a(x,v)=\tilde \tau_g(x,v)/\tau_{g_0}(x,v)$. Hence there is the representation
$
\Phi= \tilde \Upsilon_g\circ \tilde \Upsilon_{g_0}^{-1}$, $ a= [\tilde \tau_g/\tilde \tau_{g_0}]^{\sharp_{g}}$,
which at once implies uniqueness and that $\Phi$ is a homeomorphism. To show that this is a~valid solution, we need to verify smoothness. Since $g=g_0$ on $M\backslash K$, there is a neighbourhood $U\subset \partial_+SM$ of~$\partial_0SM$ such that $\Upsilon_{g}=\Upsilon_{g_0}$ on $U\times [0,1]$. Hence
\[
	\Phi = \Id \text{ on } \Upsilon_{g_0}(U\times [0,1])\qquad \text{and} \qquad \Phi = \Upsilon_g\circ \Upsilon_{g_0}^{-1}\qquad \text{on}\  SM\backslash \partial_0SM
\]
and $\Phi$ and $\Phi^{-1}$ must be smooth. Consider the spaces
\[
	C_{\alpha,\pm}^\infty(\partial_+SM)=\big\{h\in C^\infty(\partial_+SM)\mid A_{\pm}^g h \in C^\infty(\partial SM)\big\},
\]
where $u=A_{\pm}^g h\in L^\infty(\partial SM)$ is the unique function with $u\circ \alpha_g =\pm u$ and $\alpha_g(x,v)=\Upsilon_{g}(x,v,1)$ is the scattering relation of~$g$. Since $g=g_0$ on $M\backslash K$, we have $\alpha_{g}=\alpha_{g_0}$ in a neighbourhood of~$\partial_0SM$. This makes the spaces $C_{\alpha,\pm}^\infty(\partial_+SM)$ independent of the choice $g\in \M_\ntc(K)$ used in their definition. Moreover, $m(x,v)={g(\nu(x),v))}$ is a smooth function on $\partial SM$ that is independent of~${g\in \M_\ntc(K)}$. By \cite[Section 3.2]{PSU23},
\[
	\tilde \tau_{g}, \tilde\tau_{g_0} \in C^\infty_{\alpha,-}(\partial_+SM),\qquad \tilde \tau_g/m,\tilde \tau_{g_0}/m\in C^\infty(\partial_+SM)
\]
and as a consequence $\tilde \tau_g/\tilde \tau_{g_0} \in C^\infty_{\alpha,+}(\partial_+SM).$
Finally, since $\sharp_g$ maps $C^\infty_{\alpha,+}(\partial_+SM)$ into $C^\infty(SM)$, also $a$ must be smooth.

To verify that $\Phi=\Phi_g\in \Diff_+(SM)$ depends continuously on the metric, it suffices to verify continuity of the restrictions
$g\mapsto \Phi_g\vert_{V}$, $ \M_\ntc(K)\rightarrow C^\infty(V,SM)$, $ V\subset SM$ open.
For~${V=\Upsilon_{g_0}(U\times [0,1])}$ this restriction is equals the identity and so there is nothing to prove. For $V=SM\backslash \partial_0 SM$ continuity follows from Proposition~\ref{thm_continuitygalore}\,(iv). Next, by Proposition~\ref{thm_continuitygalore}\,(iii), the map
$g\mapsto a_g\vert_{\partial_+SM}$, $ \M_\ntc(K)\rightarrow C_{\alpha,+}^\infty(\partial_+SM)
$
is continuous and hence also $\tilde a_g = (a_g\vert_{\partial_+SM})^{\sharp_{g_0}}$, as a function in $C^\infty(SM)$, depends continuously on $g\in \M_\ntc(K)$. Since $a_{g}=\tilde a_g\circ \Phi_g^{-1}$, the proof is complete.
\end{proof}

For metrics without compact support condition, smoothness might be lost at the glancing region. This can be dealt with by moving the boundary condition to a slightly larger manifold.

\begin{Proposition}\label{prop_conjugacy2}
	There is a neighbourhood $\mathcal U\subset \M_\ntc$ of $g_0$ and a continuous map
	\[
		g\mapsto (\Phi_g,a_g),\qquad \mathcal U\rightarrow \Diff_+(SM)\times C^\infty(SM,\R)
	\]
	such that $(\Phi_g,a_g)$ solves $\Phi_*\mathbb X_{g_0}=a \mathbb X_g$ and $\mathbb X_{g}a =0$. Moreover, $(\Phi_{g_0},a_{g_0})=(\Id,1)$.
\end{Proposition}

\begin{proof}
	Recall that $M$ is embedded into a closed manifold $N$ of the same dimension. Let $M'\subset N$ be a slightly larger manifold with boundary, containing $M$ in its interior and such that $(M',g_0)$ remains non-trapping with strictly convex boundary. By modifying the extension map $E$ from~\eqref{extensionE} if necessary, we may assume that $\supp(Eg - g_0)\subset K$ for a compact set~${M\subset K \subset ( M')^\circ}$ and all $g\in \M_\ntc$. Let $\M_\ntc'(K)\subset C^\infty\bigl(M',\otimes^2_ST^*M'\bigr)$ be defined analogously to~\eqref{def_mntck}, then
	$
		\mathcal U = \{g\in \M_\ntc\mid Eg\vert_{M'} \in \M_\ntc'(K)\}
	$
	is an open neighbourhood of $g_0$. We obtain~$(\Phi_g,a_g)$ by applying the preceding proposition on $M'$ and then restricting to $M$.
\end{proof}

\subsection[Parametrising the C\_alpha\^{}infty-structure]{Parametrising the $\boldsymbol{C_\alpha^\infty}$-structure}\label{s_refG}

The rescaling map restricts to a homeomorphism of the influx boundaries:
\[
		\phi_g|_{\partial_+SM}\colon\ \partial_+{SM} \to \partial_+SM^g,\qquad g\in \M_\ntc.
\]
This is automatically a diffeomorphism with respect to the standard $C^\infty$-structures, but in general it does not respect the $C_\alpha^\infty$-structures from Section~\ref{s_calpha}. Instead, transferring the maximal atlas $\mathcal A_\alpha$ to $\partial_+SM$ yields a smooth structure
$\mathcal A_{\alpha}^g = (\phi_g|_{\partial_+SM})^*\mathcal A_\alpha$,
that is sensitive to the boundary behaviour of $g$. Equivalently, this smooth structure can be obtained by first rescaling the fold map from \eqref{def_foldmap},
\[
	f_g\colon\ \partial_+SM \to \partial S\hat M,\qquad f_g = \phi_g^{-1}\circ F_g\circ \phi_g
\]
and noting that $\mathcal A_\alpha^g$ is uniquely characterised by $f_g\colon (\partial_+SM,\mathcal A_\alpha^g)\to f_g(\partial_+SM)\subset\partial S\hat M$ being a~smooth embedding. Here we focus our attention on metrics $g\in \mathcal U\subset \M_\ntc$ in some open set, for which we can choose a common extension $\hat M\subset N$ of $M$ such that $\bigl(\hat M,g\bigr)$ is non-trapping and with strictly convex boundary. Let us assume that $g_0\in \mathcal U$ and write
$\mathcal G = (\partial_+SM,\mathcal A_\alpha^{g_0})$.
This serves as a model to parametrise all smooth structures $\mathcal A_\alpha^g$ ($g\in \mathcal U$) in a $g$-continuous manner.

\begin{Proposition}\label{prop_refG}
	There exists a family $(\Phi_g\mid g\in \mathcal U)$ of diffeomorphisms $\Phi_g\colon \mathcal G\to (\partial_+SM,\mathcal A_\alpha^g)$ such that the following map is continuous:
	\[
		\mathcal X =\{(g,h)\in \M_\ntc\times C^\infty(\partial SM)\mid \alpha_g^*h=h\} \to C^\infty(\mathcal G),\qquad (g,h)\mapsto \Phi_g^* h
	\]
\end{Proposition}

\begin{proof}
	The map $f_g\colon \partial SM\to \partial S\hat M$ lies in the space $\mathcal W^{\partial_0SM}(\partial SM,\partial S\hat M)$ of global Whitney folds as considered in Appendix~\ref{s_whitney} (cf.~\cite[Section~5.2]{PSU23}). Using Proposition~\ref{thm_continuitygalore}, one checks that the following map is continuous $\mathcal U \to \mathcal W^{\partial_0SM}\bigl(\partial SM, \partial S\hat M\bigr)$, $ g\mapsto f_g$.
Hence, by Theorem~\ref{thm_seswhitney}, there is a continuous map as follows
\[
T\colon\ \mathcal X \to 	C^\infty\bigl(\partial S\hat M\bigr),\qquad T(g,h) = w_g,\qquad f_g^*w_g =h.
\]
\sloppy Note that $\partial \mathcal G = \partial_0SM$ carries its standard smooth structure, because the restriction $f_g|_{\partial_0SM} \colon \partial_0SM\to \partial S\hat M$ is a smooth embedding for all $g\in \mathcal U$. Let now $\tilde f_g$ be an $\mathcal A_\alpha^{g_0}$-smooth embedding that extends this boundary datum, that is,
\smash{$\tilde f_g \in \operatorname{Emb}\bigl(\mathcal G,\partial S\hat M\bigr)$}, $ \tilde f_ g= f_g$ on $ \partial \mathcal G$.
By Theorem~\ref{prop_palais}, this extension can be implemented by a continuous map
\smash{$\mathcal U \to \operatorname{Emb}\bigl(\mathcal G,\partial S\hat M\bigr)$}, \smash{$ g\mapsto \tilde f_g$}.
We may further assume that $\tilde f_g(\mathcal G) = f_g(\partial_+SM)$, such that the map
$
	\Phi_g:= f^{-1}_g\circ \tilde f_g\colon \mathcal G\to \partial SM
$ is well-defined and smooth. If $(g,h)\in \mathcal X$, then
\[
	\Phi_g^*h = \tilde f_g^* \bigl(f_g^{-1}\bigr)^* f_g^* w_g = \tilde f_g^* w_g \in C^\infty(\mathcal G)
\]
and this depends continuously on $(g,h)$ by the above.
\end{proof}

\section{Metric dependence. Part 2: Inverses of the normal operator} \label{s_verstappen}

Let $(M,g)$ be a simple manifold of dimension $d\ge 2$, let $\rho\in C^\infty(M)$ be a boundary defining function and consider the normal operator of the geodesic X-ray transform
\begin{equation*}
N^g_0 \colon\ \rho^{-1/2}C^\infty(M)\rightarrow C^\infty(M),
\end{equation*}
as it has been defined in Section~\ref{s_geoxray}. The purpose of this section is to show that the inverses of $N_0^g$, whose existence was established in \cite{MNP19a}, depend continuously on the metric. We denote with
$
\Ms=\big\{g\in C^\infty\bigl(M,\otimes^2_ST^*M\bigr)\mid (M,g)\text{ is simple}\big\}
$
the set of simple metrics on $M$. This is a~$C^2$-open subset of the ambient space of smooth, symmetric $2$-tensors and will be equipped with the $C^\infty$-topology.

\begin{Theorem}\label{thm_ghtop}
For every $f\in C^\infty(M)$, the following map is continuous:
\begin{equation*}
 g\mapsto \bigl(N^g_0\bigr)^{-1}f,\qquad \M_s\rightarrow \rho^{-1/2} C^\infty(M).
\end{equation*}
\end{Theorem}

\begin{Remark}\label{rem_attenuatedverstappen}
The theorem generalises to the {\it attenuated} setting considered in \cite[Section~12.2]{PSU23}. If $a\in \Omega^1(M)$ is a $1$-form on $M$, then the normal operator \smash{$N_{a,0}^g=\bigl(I_{a,0}^g\bigr)^*I_{a,0}^g$} associated with the $a$-attenuated X-ray transform is an isomorphism $\rho^{-1/2}C^\infty(M)\to C^\infty(M)$. The arguments of this section then also establish continuity of the map
\[	
	(g,a) \mapsto \bigl(N^g_{a,0}\bigr)^{-1}f,\qquad \mathbb M_s\times \Omega^1(M)\to \rho^{-1} C^\infty(M),\qquad f\in C^\infty(M).
\]
\end{Remark}

As in~\cite{MNP19a}, we use that $N^g_0$ is an elliptic classical pseudodifferential of order $-1$ that satisfies a Grubb--H{\"o}rmander transmission condition. As a classical pseudodifferential operator, $N_0^g$~depends continuously on the metric (see Appendix \ref{s_appb}) and as it has an even symbol, it automatically fits into the Grubb--H{\"o}rmander framework. Using these facts, we will derive Theorem~\ref{thm_ghtop} by successively leveraging the continuous dependence to stronger topologies with help of the open mapping theorem. This requires revisiting some arguments of \cite{MNP19a}, but no `hard analysis' is needed.

\subsection{Extension of the normal operator} Let $N$ be a closed $d$-dimensional manifold engulfing $M$. Denote with $\Psi^{m}_\ev(N)$ the Fr{\'e}chet space of classical pseudodifferential operators of order $m$, having an {\it even} symbol---see Appendix \ref{s_appb}, where this is recalled. Further, write
\[
	e_M\colon\ L^2(M)\rightarrow L^2(N),\qquad r_M\colon\ L^2(N)\rightarrow L^2(M),
\]
for the extension by zero and the restriction to $M$.
\begin{Lemma}\label{lem_next} Let $g_0\in \M_\mathrm{s}$. Then there exists a neighbourhood $\mathcal U\subset \M_{\mathrm s}$ of $g_0$ as follows. For every $g\in \mathcal U$, there exist an operator $P^g\in \Psi^{-1}_\ev(N)$ such that
$
r_M P^g e_M = N_0^g$ on $ L^2(M)$.
Further, the following map is continuous:
$
g\mapsto P^g$, $ \mathcal U\rightarrow \Psi^{-1}_\ev(N)$.
\end{Lemma}

\begin{proof}
	First, by Lemma \ref{lem_rext}, there is a continuous extension map
	\[
		E\colon \ \mathcal U\rightarrow C^\infty\bigl(N,\otimes^2_S T^*N\bigr),\qquad Eg|_{M} =g
	\]	
	such that $Eg$ is a Riemannian metric on $N$. Let $M_1,\dots, M_D\subset N$ be $d$-dimensional submanifolds with boundary, such that $M\subset M_1^\circ$, $N=M_1^\circ\cup\dots\cup M_D^\circ$, $M\cap M_2=\dots= M\cap M_D=\varnothing$ and
	\begin{equation}\label{simplecover}
		(M_k,Eg_0|_{M_k}) \text{ is simple for all } 1\le k \le D.
	\end{equation}
	As simpleness is an open condition, and after shrinking $\mathcal U$ if necessary, we may assume that~\eqref{simplecover} also holds true with $g_0$ replaced by $g\in \mathcal U$. Consider now a partition of unity ${\sum_{k=1}^D\varphi_k^2 = 1}$, where~$\varphi_k\in C^\infty(N,[0,1])$ with $\supp \varphi_k \subset M_k^\circ$ for all $k=1,\dots, D$ and $\varphi_1\equiv 1$ in a~neighbourhood of~$M$. Let \smash{$
	P_k^g:=N_0^{Eg\vert_{M_k}}\in \Psi_{\cl}^{-1}(M_k^\circ)
	$}
	be the normal operator of the simple manifold~$(M_k,E g\vert_{M_k})$. Then by Proposition~\ref{claim2}, the map
$g\mapsto\! P^g_k$ is continuous from $\mathcal U$ into~$\Psi_\cl^{-1}(M_k^\circ)$.
Define
\[
P^g = \sum_{k=1}^D \varphi_k P^g_k \varphi_k \in \Psi^{-1}_\ev(N).
\]
Then the continuity of $g\mapsto P^g$ follows from Proposition~\ref{claim2} and the continuity of all natural operations in $\Psi_\cl^{m}$ (such as multiplication by $C^\infty$-functions or extension of operators with compactly supported Schwartz kernels from $M_k^\circ$ to $N$). That $P^g$ has the claimed extension property is proved in \cite[Section~4.1]{MNP19a}. There it is also shown that $P^g$ has an even symbol and as $\Psi^{-1}_\ev(N)\subset \Psi^{-1}_\cl(N)$ is a closed subspace, we may replace the latter by the former in the continuity statement.
\end{proof}

\subsection{The Grubb--H{\"o}rmander framework}\label{s_gruho}

Consider the space
\begin{equation*}
\mathcal E_\mu(M)=\{e_M\rho^\mu u \mid u\in C^\infty(M)\} \subset \mathcal D'(N),\qquad \mu\in \C,\qquad \Re\mu >-1,
\end{equation*}
which is isomorphic to $C^\infty(M)$ and becomes a Fr{\'e}chet space in this way. An operator $A\in \Psi_\cl^m(N)$ is said to satisfy the {\it $\mu$-transmission condition}, if
$
r_M A u\in C^\infty(M)$ for all $ u\in \mathcal E_\mu(M)$.
This property has been studied extensively by Grubb and H{\"o}rmander (see \cite{Gru15}), and is equivalent to a certain symmetry of the local symbol of $A$ at $\partial M$---we refer to the discussion in \cite[p.~27]{MNP19a} for more details. As observed in \cite[Lemma 4.2]{MNP19a}, the operators $P^g$ from the preceding section do satisfy the $\mu$-transmission condition for $\mu=-1/2$. This leads to a number of mapping properties, which ultimately allow to establish the isomorphism property.

We retrace the argument from \cite{MNP19a} with a view on continuous $g$-dependence. To this end it is useful to introduce a scale of Hilbert spaces $H^{\mu(s)}(M)$ ($s\in \R)$, referred to as {\it H{\"o}rmander spaces}. For their definition, which is intricate, we refer to \cite[p.~29]{MNP19a} and the references therein. For us, it suffices to know the following properties:
\begin{enumerate}[label=(\roman*)]\itemsep=0pt
\item $H^{\mu(0)}(M)=L^2(M)$ and $H^{\mu(s+t)}(M)\subset H^{\mu(s)}(M)$ for all $s\in \R$, $t\ge 0$;

\item $\bigcap_{s\in \R} H^{\mu(s)}(M) = \mathcal E_\mu(M)$ and $\mathcal E_\mu(M)$ carries the initial topology with respect to the inclusions $\mathcal E_\mu(M)\hookrightarrow H^{\mu(s)}(M)$;

\item $\mathcal E_\mu(M)\subset H^{\mu(s)}(M)$ is a dense subspace for all $s\in \R$; see \cite[Proposition 4.1]{Gru15};

\item if $A\in \Psi^m_\cl(N)$ satisfies the $\mu$-transmission condition, then it has the continuous mapping property $
r_M A\colon H^{\mu(s)}(M)\rightarrow H^{s-m}(M)$, $ s>\Re \mu -1/2$;
see \cite[Theorem 4.2]{Gru15}.
\end{enumerate}

\begin{Proposition}\label{ctsctsaction} Any $A\in \Psi_{\ev}^m(N)$ $(m\ge -2)$ automatically satisfies the $\mu$-transmission condition for $\mu=m/2$ and hence $r_MA\in \mathcal L(H^{\mu(s)}(M),H^{s-m}(M))$ for every $s>\Re \mu -1/2$. Moreover, the following map is continuous:
\[
A\mapsto r_MA,\qquad \Psi_{\ev}^m(N) \rightarrow \mathcal L\big(H^{\mu(s)}(M),H^{s-m}(M)\big).
\]
\end{Proposition}

The novelty of this proposition is the continuity claim, which is proved by means of the following application of the closed graph theorem.

\begin{Lemma}[continuity upgrade]\label{upgrade}
A linear map $T\colon E\to F$ between Fr{\'e}chet spaces is continuous if and only if it is continuous with respect to some weaker Hausdorff topology on $F$.
\end{Lemma}

\begin{proof}
	Write $\tau_E$, $\tau_F$ for the Fr{\'e}chet space topologies on $E$ and $F$, respectively, and assume that~${\tau\subset \tau_F}$ is Hausdorff. Then continuity of $T$ as map $(E,\tau_E)\to (F,\tau)$ implies that $\mathrm{Graph}(T)$ is $\tau_E\times \tau$-closed. Hence it is $\tau_E\times \tau_F$-closed and the closed graph theorem implies continuity as map $(E,\tau_E)\to (F,\tau_F)$.
\end{proof}

\begin{proof}[Proof of Proposition~\ref{ctsctsaction}]
We wish to apply Lemma \ref{upgrade} with $E=\Psi_\mathrm{ev}^m(N)$ and
\[
F=\mathcal L(H^{\mu(s)}(M),H^{s-m}(M)),
\]
 where on $F$ we consider the initial topology $\tau$ with respect to the following collection of maps (for $L=\max(s,0)\ge 0$)
\[
	\mathcal L(H^{\mu(s)}(M),H^{s-m}(M)) \to H^{s-m-L}(M),\qquad T\mapsto Tf,\qquad \text{where}\quad f\in \mathcal E_\mu(M).\]
This topology is clearly weaker than the operator norm topology on $F$. To see that $\tau$ is Hausdorff, consider two distinct operators $T_1,T_2\in F$. Since $\mathcal E_\mu(M)\subset H^{\mu(s)}(M)$ is dense, we must have $T_1 f \neq T_2 f$ for some $f\in \mathcal E_\mu(M)$. Hence there exists $\epsilon>0$ such that the balls
\[
B_k=\bigl\{g\in H^{s-m-L}(M)\mid ||T_k f - g||_{H^{s-m-L}(M)}<\epsilon\bigr\},\qquad k=1,2,
\]
 are disjoint and the sets $\{T\in F\mid Tf\in B_k\} \in \tau$ separate $T_1$ and $T_2$.

To prove the continuity assertion of the proposition, we may therefore fix $f\in \mathcal E_\mu(M)$ and prove continuity of the map $ A\mapsto r_M Af$, $ \Psi^m_\mathrm{ev}(N) \to H^{s-m-L}(M)$. Take $B\in \Psi_\mathrm{cl}^{s-m-L}(N)$ such that $B\colon H^{s-m-L}(N)\to L^2(N)$ is an isomorphism, then
\[
	||r_MA f||_{H^{s-m-L}(M)}\le ||A f||_{H^{s-m-L}(N)} \lesssim ||BAf||_{L^2(N)}=:p(A)
\]
and we are done after showing that $p(\cdot)$ is a continuous seminorm on $\Psi^m_{\mathrm{cl}}$. However, this follows from the standard continuity properties in the calculus $\Psi_{\mathrm{cl}}^*$; concretely from the continuity of~${(B,A)\mapsto BA}$ as a map $\Psi^{s-m-L}_\mathrm{cl}(N)\times \Psi^{m}_\mathrm{cl}(N) \to \Psi_{\mathrm cl}^{s-L}(N) \subset \Psi^0_{\mathrm{cl}}(N) \hookrightarrow \mathcal L\bigl(L^2(N)\bigr)$.
\end{proof}

\subsection{Proof of Theorem~\ref{thm_ghtop}}
 Recall the following elementary result.

\begin{Lemma}\label{inversion} Let $E$ and $F$ be normed vector spaces. Then the subset $\mathcal L^\times(E,F) \subset \mathcal L(E,F)$ of invertible operators is open and inversion is continuous as map
$
T\mapsto T^{-1}$, $ \mathcal L^\times(E,F)\rightarrow \mathcal L(F,E)$.
\end{Lemma}

\begin{Remark}
In general, the lemma fails for non-normed topological vector spaces. Consider, e.g., $E=F=C^\infty([0,1])$ and $T_n = \operatorname{Id} - \frac 1n \frac {\rm d}{{\rm d}x}\in \mathcal L(E,E)$. This sequence converges to $\operatorname{Id}$ in~$\mathcal L(E,E)$ (equipped with the topology of pointwise convergence), despite $T_n$ not being invertible for any $n\in \mathbb N$. This necessitates the usage of the Hilbert spaces $H^{\mu(s)}(M)$ instead of~$\mathcal E_\mu(M)$.
\end{Remark}

\begin{proof}[Proof of Lemma \ref{inversion}] If $\mathcal L^\times (E,F)$ is non-empty, we may assume without loss of generality that $E=F$. Let $T_0\in \mathcal L^\times (E,E)$, then for all $T\in \mathcal L(E,E)$ for which $H= (T_0-T)T_0^{-1}$ has operator norm $<1$, the inverse $T^{-1}$ exists and satisfies
\begin{gather*}
T^{-1} =T_0^{-1}(\operatorname{Id} - H )^{-1} = T_0^{-1} \sum_{k\ge 0} H^k,\\
\big\Vert T^{-1} -T_0^{-1} \big\Vert \le \bigl\Vert T_0^{-1} \bigr\Vert \cdot \sum_{k\ge 1}  \Vert H \Vert^k \le \big\Vert T_0^{-1} \big\Vert \bigl((1-\Vert H \Vert)^{-1} - 1\bigr).
\end{gather*}
The last expression vanishes in the limit $\Vert H \Vert\rightarrow 0$, which proves that the map $T\mapsto T^{-1}$ is continuous at $T_0$.
\end{proof}

\begin{proof}[Proof of Theorem~\ref{thm_ghtop}] We fix a metric $g_0 \in \M_\mathrm{s}$ and restrict our attention to the open neighbourhood $\mathcal U$ and the continuous family $(P^g\mid g\in \mathcal U)$ of operators from Lemma \ref{lem_next}. By Proposition~\ref{ctsctsaction}, the following map is continuous:
\begin{equation*}
g\mapsto r_MP^g,\qquad \mathcal U\rightarrow \mathcal L\bigl(H^{-1/2(s)}(M),H^{s+1}(M)\bigr),\qquad s>-1.
\end{equation*}
By \cite[Theorem~4.4]{MNP19a}, the operator $r_MP^g \colon H^{-1/2(s)}(M)\rightarrow H^{s+1}(M)$ ($g\in \mathcal U$) is an isomorphism. Hence the nonlinear map from the preceding display takes values in the set of invertible operators, which we denote
$
\mathcal L^\times\bigl(H^{-1/2(s)}(M),H^{s+1}(M)\bigr)$.
By Lemma \ref{inversion}, also the following map is continuous:
\begin{equation*}
g\mapsto \bigl(r_M P^g\bigr)^{-1},\qquad \mathcal U \rightarrow \mathcal L\bigl(H^{s+1}(M),H^{-1/2(s)}\bigr),\qquad s>-1.
\end{equation*}
To complete the proof of Theorem~\ref{thm_ghtop}, let $f\in C^\infty(M)$. Then by the preceding display, $ g\mapsto (r_M P^g)^{-1} f$ is continuous as a map from $\mathcal U$ into $H^{-1/2(s)}(M)$ for every $s>-1$. By property (ii) above, the map is also continuous from $\mathcal U$ into $\mathcal E_{-1/2}(M)$. Further, $r_M\colon \mathcal E_{-1/2}(M)\rightarrow \rho^{-1/2}C^\infty(M)$ is continuous and hence
\begin{equation*}
g\mapsto r_M(r_MP^g)^{-1} f,\qquad \mathcal U\rightarrow \rho^{-1/2}C^\infty(M)
\end{equation*}
is continuous. For $g\in \mathcal U$, let $h=r_M(r_MP^g)^{-1} f$, then due to the extension property in Lemma~\ref{lem_next}, we have
\begin{equation*}
N_0^g h = r_M P^g e_M h = (r_MP^g)(r_MP^g)^{-1} f = f
\end{equation*}
and thus the theorem is proved.
\end{proof}

\section{The holomorphic blow-down structure under perturbations}\label{s_bds}

In this section, we prove Theorem~\ref{thm_openbds}, starting with some considerations regarding metric geometry of the Hermitian manifold $(Z^\circ,\Om)$.

\subsection{Lipschitz maps on Hermitian manifolds}\label{s_metriccx2}
\sloppy Let $(X,\Omega_X)$ and $(Y,\Omega_Y)$ be complex manifolds of dimension $n$ and $m$, respectively, both equipped with Hermitian metrics.
We assume for simplicity that the bundles $\Lambda^{1,0}X$ and~$\Lambda^{1,0}Y$ are smoothly trivial, such that there exist global orthonormal frames
$\big\{\eta_X^1,\dots,\eta_X^n\big\}$ and $\big\{\eta_Y^1,\dots,\eta_Y^m\big\}$.
In these frames, the differential of a smooth map $f\colon X\rightarrow Y$ can be expressed in terms of two complex $n\times m$ matrices
\begin{equation*}
S,T\in C^\infty\bigl(X,\C^{n\times m}\bigr),\qquad S=(s_{jk}),\quad T=(t_{jk}),
\end{equation*}
defined through the following equation
\begin{equation}\label{defst}
f^* \eta_Y^k = \sum_{j=1}^n s_{jk}\eta_X^j + t_{jk} \bar \eta_X^j,\qquad k=1,\dots, m.
\end{equation}

\begin{Remark}
Equivalently, $S(x)$ and $T(x)$ are defined through the following commutative diagram:
\begin{equation}
\label{commdiagram}
\begin{tikzcd}[ampersand replacement=\&]
\bigl(\Lambda^{1,0}Y\oplus \Lambda^{0,1}Y\bigr)_{f(x)}
\arrow{d} \arrow["{\rm d}f_x^T"]{r} \& \bigl(\Lambda^{1,0}X\oplus \Lambda^{0,1}X\bigr)_{x} \arrow{d}\\
\C^m\oplus \C^m \arrow{r}{\qquad~\begin{bmatrix}
S &\bar T\\
T & \bar S
\end{bmatrix}(x)} \&\C^n\oplus \C^n.
\end{tikzcd}
\end{equation}
Here ${\rm d}f_x^T$ is extended $\C$-linearly and the vertical arrows map a vector in $(T_\C X)_x$ to its expression in the basis $\eta_X^1,\dots, \eta_X^n,\bar \eta_X^1,\dots,\bar \eta_X^n$ (similarly for $Y$).
\end{Remark}

We now consider $X$ and $Y$ as a metric spaces, where the metric is given by the geodesic distance functions $d_{\Omega_X}$ and $d_{\Omega_Y}$, respectively. Further, $|\cdot|\colon \C^{n\times n}\rightarrow [0,\infty)$ is the operator norm on Euclidean $\C^n$.

\begin{Proposition}\label{distancebound} Suppose $f\colon (X,\Omega_X)\rightarrow (Y,\Omega_Y)$ is a smooth map whose differential is given with respect to orthonormal frames as $S,T\in C^\infty\bigl(X,\C^{n\times m}\bigr)$.
\begin{enumerate}[label={\rm (\roman*)}]\itemsep=0pt
\item \label{lipschitz1} If $\vert S\vert,\vert T \vert \in L^\infty(X)$, then $f$ is globally Lipschitz from $(X,d_{\Omega_X})$ into $(Y,d_{\Omega_Y})$ and its Lipschitz constant satisfies
\begin{equation*}
\Lip(f) \le \Vert S \Vert_{L^\infty(X)}+ \Vert T \Vert_{L^\infty(X)}.
\end{equation*}
\item\label{lipschitz2} If additionally $f$ is a diffeomorphism and for some $\epsilon>0$, we have
\[
	\Vert S-\Id \Vert_\infty, \Vert T \Vert_\infty <\epsilon/3,
\]
 then $f$ is globally bi-Lipschitz and
 \begin{equation}\label{lipschitz3}
 \Lip (f) \le 1+ \epsilon \qquad \text{and} \qquad \Lip\bigl(f^{-1}\bigr) \le (1-\epsilon)^{-1},
 \end{equation}
 or equivalently,
 \begin{equation}
 \label{lipschitz4}
\sup_{\substack{x,x'\in X\\x\neq x'}}\left \vert 1 - \frac{d_{\Omega_Y}(f(x),f(x'))}{d_{\Omega_X}(x,x')} \right\vert \le \epsilon.
\end{equation}
\end{enumerate}
\end{Proposition}

\begin{proof} For \ref{lipschitz1}, it suffices to uniformly bound the operator norm of ${\rm d}f_x\colon T_xX \rightarrow T_yY$, which satisfies (writing $T=T(x)$ and $S=S(x)$)
\[
\vert {\rm d}f_x \vert^2 = \big\vert {\rm d}f_x^T \big\vert^2 \le \left | \begin{bmatrix}
S &\bar T\\
T & \bar S
\end{bmatrix} \right|^2 =\sup \big\{|Sv + \bar T w|^2 +| Tv + \bar S w |^2\big\},
\]
where the supremum is taken over all $v,w\in \C^m$ with $|v|^2+|w|^2\le 1$. The desired norm estimate then follows from Young's inequality
\begin{align*}
|Sv + \bar T w|^2 +| Tv + \bar S w |^2 & \le |S|^2 \bigl(|v|^2 + |w|^2\bigr) + 4 |S||T||v||w|+ |T|^2 \bigl(|v|^2 + |w|^2\bigr) \\
& \le (|S| + |T|)^2\bigl(|v|^2+ |w|^2\bigr),
\end{align*}
For \ref{lipschitz2}, the first bound in \eqref{lipschitz3} follows immediately from \ref{lipschitz1}. The differential of $f^{-1}$ at $f(x)$ is given as
\begin{equation*}
\begin{bmatrix}
	S & \bar T\\
	T & \bar S
\end{bmatrix}^{-1} = \begin{bmatrix}
	F & F\bar G\\
	\bar F G & \bar F
\end{bmatrix},\qquad \text{where}\quad F= \bigl(S-\bar T \bar S^{-1}T\bigr)^{-1},\quad G= - TS^{-1}.
\end{equation*}
Writing $P=\Id - S$ and $Q = P + \sum_{j\ge 0}\bar T \bar P^j T$, we have
\begin{gather*}
|Q| \le \frac \epsilon 3 + \frac {\epsilon^2} 9 \cdot \frac{1}{1-\epsilon/3} = \frac \epsilon {3-\epsilon},\qquad \vert F\vert\le \sum_{k\ge 0}|Q|^k\le \frac{1}{1-\frac{\epsilon}{3-\epsilon}} = \frac{3-\epsilon}{3-2\epsilon}\\
|G| \le |T| \sum_{j\ge 0} |P|^j \le \frac \epsilon 3 \cdot \frac{1}{1-\epsilon/3} = \frac{\epsilon}{3-\epsilon},\qquad 1 + |G| \le \frac{3}{3-\epsilon}.
\end{gather*}
Combining these estimates, we see that
$
|F| (1+|G|) \le \frac{3}{3-2\epsilon} \le \frac{1}{1-\epsilon}
$
and the estimate on $\Lip\bigl(f^{-1}\bigr)$ follows again from \ref{lipschitz1}. Finally, note that
\begin{align*}
\eqref{lipschitz3}\quad &\Leftrightarrow \quad f^*d_{\Omega_Y} \le (1+\epsilon) d_{\Omega_X} \qquad\text{and}\qquad d_{\Omega_X}\le (1-\epsilon)^{-1} f^*d_{\Omega_Y}\\
&\Leftrightarrow \quad -\epsilon d_{\Omega_X} \le f^*d_{\Omega_Y} - d_{\Omega_X} \le \epsilon d_{\Omega_X}\quad \Leftrightarrow \quad \eqref{lipschitz4},
\end{align*}
which completes the proof.
\end{proof}

\subsection[Jacobian and difference quotient of beta-maps]{Jacobian and difference quotient of $\boldsymbol{\beta}$-maps}

Recalling the set-up from Section~\ref{s_perturbationtheory}, let $Z=\DD\times \DD$ and consider for $\sigma\in \Sigma_\mathrm{s}$ the following vector fields:
\begin{align*}
	\Xi_\sigma&={\rm e}^{-\sigma} \bigl(\mu^2\partial_z+\partial_{\bar z} + \bigl(\mu^2 \partial_z\sigma - \partial_{\bar z}\sigma\bigr)(\bar \mu \partial_{\bar \mu} - \mu \partial_\mu) \bigr) \in C^\infty(Z,T_\C Z).
\end{align*}
Further, recall that
$
\mathcal I=\big\{(\beta,\sigma)\in C^\infty\bigl(Z,\C^2\bigr)\times \Sigma_\mathrm{s}\mid \Xi_\sigma\beta = \partial_{\bar \mu} \beta=0\big\}
$. If $(\beta,\sigma)\in \mathcal I$, then the Jacobian of $\beta$ with respect to the frame $\big\{\bigl(1-|\mu|^4\bigr)\bar \Xi_\sigma^\vee,\partial_\mu^\vee\big\}\subset \Lambda^{1,0}Z^\circ$ is given by the following $2\times 2$-matrix
\[
		S_\beta = \begin{bmatrix}
			\bigl(1-|\mu|^4\bigr)^{-1}\bar \Xi_\sigma \beta_1 & \bigl(1-|\mu|^4\bigr)^{-1}\bar \Xi_\sigma \beta_2\\
			\partial_\mu \beta_1 & \partial_\mu \beta_2
		\end{bmatrix}.
\]
As this frame is orthonormal for the Hermitian metric $\Om_\sigma$ introduced in Section~\ref{s_conformallyeuclidean}, this is precisely one of the Jacobian matrices considered above (the anti-holomorphic part $T_\beta$ is zero, because $\beta$ is holomorphic).

\begin{Lemma}\label{sbetacts} Let $(\beta,\sigma)\in \mathcal I$, then $|S_\beta|\in L^\infty(Z^\circ)$ and consequently $\beta^*\Omega_{\C^2} \le C \Om_\sigma$ for some~${C>0}$. Further, the following map is continuous
$(\beta,\sigma)\mapsto S_\beta$, $ \mathcal I\rightarrow L^\infty\bigl(Z^\circ,\C^{2\times 2}\bigr)$.
\end{Lemma}

\begin{proof} For any $f\in C^\infty(Z)$, we have
\begin{equation*}
\bigl(\bar \Xi_\sigma - \bar \mu^2 \Xi_\sigma\bigr)f = {\rm e}^{-\sigma} \bigl(1-|\mu|^4\bigr) \left[ \partial_z f + \partial_z \sigma\cdot (\bar \mu \partial_{\bar \mu} - \mu \partial_\mu)f\right].
\end{equation*}
If $f$ is holomorphic ($\Xi_\sigma f = \partial_{\bar \mu} f =0$), this implies that
\begin{equation*}
\bigl(1-|\mu|^4\bigr)^{-1} \bar \Xi_\sigma f = {\rm e}^{-\sigma} [ \partial_z f - (\mu \partial_z \sigma) \partial_\mu f ] \in C^\infty(Z).
\end{equation*}
Applying this to the components of $\beta$, we see that $S_\beta \in C^\infty\bigl(Z,\C^{2\times 2}\bigr)$ and in particular, $|S_\beta|$ is bounded on $Z^\circ$, with the continuous dependence being obvious. Moreover, the matrix representation of $\beta^*\Omega_{\C^2}$ in the frame $\big\{\bigl(1-|\mu|^4\bigr)\bar \Xi_\sigma^\vee,\partial_\mu^\vee\big\}$ is $S_\beta S_\beta^*$. Hence
\begin{equation}\label{omegainequality1}
\beta^*\Omega_{\C^2} \le C \Om_\sigma\quad \Leftrightarrow \quad S_\beta S_\beta^* \le C \operatorname{Id}.
\end{equation}
This is satisfied upon setting $C=\sup_{Z^\circ}|S_\beta|^2$.
\end{proof}

In view of Proposition~\ref{distancebound} and the preceding lemma, the map $\beta$ is automatically Lipschitz with respect to the geodesic distance $d_\sigma$ on $(Z^\circ,\Om_\sigma)$. That is, the following difference quotient yields a well-defined element of $L^\infty\bigl(Z^\circ\times Z^\circ,\C^2\bigr)$:
\[
		B(z,\mu,z',\mu') = \frac{\beta(z,\mu)- \beta(z',\mu')}{d_\sigma((z,\mu),(z',\mu'))},\qquad (z,\mu,z',\mu')\in Z^\circ \times Z^\circ.
\]
We can now reformulate the holomorphic blow-down structure in terms of the $L^\infty$-functions $S_\beta$ and $B$. (Note that both quantities depend on the tuple $(\beta,\sigma)\in \mathcal I$, though this is suppressed in the notation.)

\begin{Proposition}[characterisation of $\mathcal I_+$]\label{chari}
	Let $(\beta,\sigma)\in \mathcal I$. Then $(\beta,\sigma)\in \mathcal I_+$ {\rm(}that is, $\beta$ has a holomorphic blow-down structure with respect to $\D_\sigma$ and $\Om_\sigma)$, if and only if the following conditions are satisfied:
	\begin{enumerate}[label=\rm(\roman*)]\itemsep=0pt
		\item\label{chari1} $\Phi_{g_\sigma}^*\beta \colon\mathcal G \rightarrow \C^2$ is a totally real embedding $($with $\mathcal G$ and $\Phi_{g_\sigma}$ as in Section~{\rm\ref{s_refG}}$)$;
		
		\item \label{chari2} $|{\det S_\beta}|\ge c_1$ on $Z^\circ$ for some constant $c_1>0$;
		
		\item \label{chari3} $|B|\ge c_2$ almost everywhere on $Z^\circ$ for some constant $c_2>0$.
	\end{enumerate}
\end{Proposition}

\begin{proof} In view of Section~\ref{s_refG}, property \ref{chari1} and part \ref{bds3} of Definition~\ref{def_bds} are identical, so it remains to show that the other items are equivalent.
Let $0\le \lambda_\beta\le \Lambda_\beta$ be the eigenvalues of~$S_\beta S_\beta^*$. Then $\Lambda_\beta = |S_\beta|^2$, $|{\det S_\beta}|^2 = \lambda_\beta \Lambda_\beta$ and similarly to \eqref{omegainequality1} one shows that
 \[
\beta^*\Omega_{\C^2} \ge c \Om_\sigma\quad \Leftrightarrow \quad S_\beta S_\beta^* \ge c \operatorname{Id} \quad \Leftrightarrow \quad \lambda_\beta \ge c.
\]
If additionally $\beta\colon Z^\circ\rightarrow \beta(Z^\circ)$ is known to be a biholomorphism, then
 \begin{gather*}
\beta^*\Omega_{\C^2} \ge c \Om_\sigma\quad \Leftrightarrow \quad \Omega_{\C^2} \ge c\bigl(\beta^{-1}\bigr)^*\Om_\sigma \quad\Rightarrow \quad \beta^{-1} \text{ is Lipschitz} \quad \Leftrightarrow \quad |B|\ge c,
\end{gather*}
where $\beta^{-1}$ being Lipschitz is understood with respect to the Euclidean metric on $\beta(Z^\circ)$ and the geodesic distance $d_\sigma$ on $Z^\circ$. The implication in the middle follows then as in the preceding lemma, using Proposition~\ref{distancebound}.

Let now $(\beta,\sigma)\in \mathcal I_+$, then the preceding consideration immediately imply \ref{chari2} and \ref{chari3}. Vice versa, if $(\beta,\sigma)\in \mathcal I$ satisfies \ref{chari2}, then \[
\lambda_\beta = |{\det S_\beta|^2/}|S_\beta|^2
 \ge c_1^2/\sup_{Z^\circ}|S_\beta|^2=:c \quad \Rightarrow \quad \beta^*\Omega_{\C^2} \ge c \Om_\sigma
\]
and as a further consequence $\beta\colon Z^\circ\rightarrow \C^2$ is a local diffeomorphism.
If also \ref{chari3} is satisfied, then~${\beta\colon Z^\circ \rightarrow \C^2}$ must be injective and we conclude that it is a biholomorphism onto its image.\looseness=-1
\end{proof}

\begin{Remark}\label{rlocint}
The preceding proof also justifies \eqref{locint}. Indeed, ${\rm d}\beta_1 \wedge {\rm d}\beta_2 = (\det S_\beta)\cdot \eta_1\wedge \eta_2$, where $\eta_1 = \bigl(1-|\mu|^4\bigr) \bar\Xi_\sigma^\vee = e^\sigma \bigl({\rm d}\bar z - \bar \mu^2 {\rm d}z\bigr)$ and $\eta_2 = ({\rm d}\bar \mu + \bar \mu(\partial_{\bar z}\sigma {\rm d}\bar z - \partial_z \sigma {\rm d}z))$, such that $\eta_1\wedge \eta_2$ is everywhere non-zero.
\end{Remark}

\subsubsection{Openness of the first two conditions} The first two conditions in Proposition~\ref{chari} are easily seen to be open in the $C^\infty$-topology. Precisely, the map
\[
(\beta,\sigma) \mapsto (\Phi_{g_\sigma}^*\beta,  \det S_\beta ),\qquad \mathcal I\mapsto C^\infty\bigl(\mathcal G,\C^2\bigr)\times L^\infty(Z^\circ)	
\]
is continuous and the two sets encoding conditions \ref{chari1} and \ref{chari2}, respectively,
\begin{gather*}
\operatorname{Emb}_\mathrm{TR}\bigl(\partial_+S\DD,\C^2\bigr)\subset C^\infty\bigl(\partial_+S\DD,\C^2\bigr)\!\qquad \text{and} \!\qquad\{f\in L^\infty(Z^\circ)\mid \mathrm{ess inf } |f| > 0 \} \subset L^\infty(Z^\circ)
\end{gather*}
are both open. Here $\operatorname{Emb}_\mathrm{TR}\bigl(\partial_+ S\DD,\C^2\bigr)$ is the set of totally real embeddings and this is open in the $C^\infty$-topology, because $\partial_+S\DD$ is compact \cite[Chapter~2, Theorem~1.4]{Hir76}.\footnote{See also the discussion in Section~\ref{s_hermitian}. Totally realness is clearly a $C^1$-open condition.}
Condition~\ref{chari3} is trickier, because the map
\[
	(\beta,\sigma)\mapsto B,\qquad \mathcal I\rightarrow L^\infty\bigl(Z^\circ\times Z^\circ,\C^2\bigr)
\]
is {\it not} continuous. Indeed, if $\sigma_1\neq \sigma_2$, then the corresponding difference quotients cannot directly be compared because $1-d_{\sigma_1}/d_{\sigma_2}\notin L^\infty(Z^\circ\times Z^\circ)$. This is due to the fact that $\Id\colon (Z^\circ,d_{\sigma_1})\rightarrow (Z^\circ,d_{\sigma_2})$ is not bi-Lipschitz. To overcome this problem, we seek out adequate bi-Lipschitz maps with small distortion.

\subsection[Proof of Theorem 3.10]{Proof of Theorem~\ref{thm_openbds}} \label{pf_openbds} We fix a tuple $(\beta_0,\sigma_0) \in \mathcal I_+$ and claim that this is an interior point. Recall from Remark \ref{rmk_xpres} that the geodesic flow of ${\rm e}^{2\sigma}|{\rm d} z|^2$, projected to $\DD\times \mathbb S^1\subset \partial Z$, is generated by the following vector field
$
\mathbb X_\sigma = \mu \Xi_\sigma|_{S\DD}\in C^\infty(S\DD,T(S\DD))$,  $\sigma \in \Sigma_\mathrm{s}$.
Let us restate Proposition~\ref{prop_conjugacy2} in the present context.

\begin{Proposition}\label{flowback} There is a neighbourhood $\mathcal U$ of $\sigma_0$ in $\Sigma_s$ and a continuous map
\[
	\sigma \mapsto (\psi_\sigma,a_\sigma),\qquad \mathcal U\rightarrow \Diff_+(S\DD)\times C^\infty(S\DD,\R)
\]
such that $(\psi_\sigma)_*(\mathbb X_{\sigma_0})=a_\sigma \mathbb X_\sigma$ and $(\psi_{\sigma_0},a_{\sigma_0})=(\Id,1)$.
\end{Proposition}
Any diffeomorphism $\Psi$ of $Z$ induces a smooth map $(Z^\circ,\D_{\sigma_0})\rightarrow (Z^\circ,\D_\sigma)$ between complex surfaces. As explained in Section~\ref{s_metriccx2}, we may express the differential of this map in terms of the orthonormal frames $\big\{\bigl(1-|\mu|^4\bigr) \Xi_{\sigma_0}^\vee,\partial_\mu^\vee\big\}$ and $\big\{\bigl(1-|\mu|^4\bigr) \Xi_\sigma^\vee,\partial_\mu^\vee\big\}$. This yields two matrix-valued maps
\begin{equation}\label{defstpsi}
S_\Psi,T_\Psi\in C^\infty\bigl(Z^\circ,\C^{2\times 2}\bigr),
\end{equation}
defined as in \eqref{defst}. If $\Psi=\Psi_\sigma$ extends the diffeomorphism $\psi_\sigma$ from the preceding proposition (whose existence follows from Proposition~\ref{prop_palais}), we can estimate these matrices as follows.

\begin{Proposition}\label{STestimate} Suppose $(\Psi_\sigma\mid \sigma\in \mathcal U)\subset \mathrm{Diff}_+ (Z)$ is a continuous family of diffeomorphisms with $\Psi_0\equiv \Id$ and $\Psi_\sigma=\psi_\sigma$ on $S\DD$. Then for every $\epsilon >0$, there exists a neighbourhood $\Sigma_0$ of~$\sigma_0\in \Sigma_\mathrm{s}$ such that the matrices \eqref{defstpsi} satisfy
\begin{equation*}
|| S_{\Psi_\sigma} - \Id ||_{L^\infty(Z^\circ)}, || T_{\Psi_\sigma} ||_{L^\infty(Z^\circ)} <\epsilon/3,\qquad \text{for}\quad \sigma \in \Sigma_0.
\end{equation*}
\end{Proposition}

\begin{proof}
We extend the functions $a_\sigma$ from Proposition~\ref{flowback} to all of $Z$, such that $\sigma\mapsto a_\sigma$ is continuous from $\Sigma_\mathrm{s}$ to $C^\infty(Z)$ and such that $a_0\equiv 1$. By this proposition, the vector fields
 \[
	\Theta'_\sigma:={\Psi_\sigma}_*\Xi_0 - a_\sigma \Xi_\sigma \in C^\infty(Z,T_\C Z),\qquad \sigma\in \Sigma_\mathrm{s}
\]	
satisfy $\Theta'_\sigma = 0$ on the boundary hypersurface $S\DD$. As $\bigl(1-|\mu|^4\bigr)$ is a boundary defining function for this hypersurface, there is a split exact sequence of Fr{\'e}chet spaces:
\[
0\rightarrow C^\infty(Z,T_\C Z)\xrightarrow{(1-|\mu|^4)\times }C^\infty(Z,T_\C Z)\xrightarrow{\Theta' \mapsto \Theta' |_{S\DD}} C^\infty(S\DD, T_\C Z|_{S\DD}) \rightarrow 0.
\]
As a consequence, $\Theta_\sigma =\Theta'_\sigma/ \bigl(1-|\mu|^4\bigr)$ is defined smoothly up to the boundary of $Z$ and moreover, the following map is continuous:
\[
	\sigma\mapsto \Theta_\sigma,\qquad \Sigma_\mathrm{s}\rightarrow C^\infty(Z,T_\C Z).
\]
Further, we define $\rm H_\sigma = \Psi_{\sigma,*}\partial_\mu - \partial_\mu$. Then also $\sigma\mapsto {\rm H}_\sigma$ is continuous as map $\Sigma_\mathrm{s}(K)\rightarrow C^\infty(Z,T_\C Z)$. Clearly, $\Theta_{\sigma_0} \equiv {\rm H}_{\sigma_0} \equiv 0$.

The matrix $S=S_{\Psi_\sigma}$ has entries $s_{jk}\in C^\infty(Z)$ $(1\le j,k\le 2)$, given by
\begin{gather*}
s_{11} = \big\langle \Psi_{\sigma*} \Xi_{\sigma_0}, \Xi_\sigma^\vee\big\rangle = \big\langle \Theta_\sigma,  \bigl(1-|\mu |^4\bigr)\Xi_\sigma^\vee\big\rangle + a_\sigma,\\
s_{12} = \big\langle \Psi_{\sigma*} \big[\bigl(1-|\mu|^4\bigr)^{-1} \Xi_{\sigma_0}\big], \partial_\mu^\vee\big\rangle = \big\langle \Theta_\sigma, \partial_\mu^\vee\big\rangle, \\
s_{21} = \big\langle \Psi_{\sigma*} \partial_\mu, \bigl(1-|\mu|^4\bigr)\Xi_\sigma^\vee\big\rangle = \big\langle {\rm H}_\sigma, \bigl(1-|\mu|^4\bigr) \Xi_\sigma^\vee\big\rangle,\\
s_{22} = \big\langle \Psi_{\sigma*} \partial_\mu, \partial_
\mu^\vee\big\rangle = \big\langle {\rm H}_\sigma, \partial_\mu^\vee\big\rangle + 1.
\end{gather*}
Both $\bigl(1-|\mu|^4\bigr)\Xi_\sigma^\vee$ and $\partial_\mu^\vee$ extend to smooth $1$-forms on $Z$ and, as elements of $C^\infty(Z,T^*_\C Z)$, depend continuously on $\sigma$---this is apparent in the explicit expression in~\eqref{def_omegasigma}.
 Hence their $L^\infty$-norm can be bounded uniformly in $\sigma$, in a neighbourhood of $\sigma =\sigma_0$. Further, all of $|| \Theta_\sigma ||_{\infty}, ||{\rm H}_\sigma ||_\infty $ and $||a_\sigma -1 ||_\infty$ are arbitrarily small in a neighbourhood of
$\sigma=\sigma_0$. This yields the desired bound on $||S-\Id||_\infty$. For $||T||_\infty$, one proceeds analogously.
\end{proof}

\subsubsection*{Proof of Theorem~\ref{thm_openbds}}
We have to show that the given tuple $(\beta_0,\sigma_0)\in \mathcal I_+$ is an interior point. Denote with $B_0\in L^\infty\bigl(Z^\circ\times Z^\circ,\C^2\bigr)$ the associated difference quotient.
 Let us introduce the notation \[
 \Delta \beta(z,\mu,z',\mu')=\beta(z,\mu)-\beta(z',\mu')\qquad \text{and} \qquad \hat \Psi_\sigma(z,\mu,z',\mu')=(\Psi_\sigma(z,\mu),\Psi_\sigma(z',\mu')),\]
 where $(\Psi_\sigma\mid \sigma\in \mathcal U)$ is a family of diffeomorphisms as in the preceding proposition. The difference quotient $B=\Delta \beta/d_{\sigma}$ associated to another tuple $(\beta,\sigma)$ can then be compared to $B_0$ as follows:
\begin{align*}
\big|\big|\hat \Psi_\sigma^*B-B_0\big|\big|_\infty &= \left\Vert
\frac{\hat \Psi_\sigma^*\Delta \beta}{\hat \Psi_\sigma^*d_\sigma} \times \left[1- \frac{\hat \Psi_\sigma^*d_\sigma}{d_{\sigma_0}}\right] + \frac{\hat \Psi_\sigma^*\Delta \beta - \Delta \beta_0}{d_{\sigma_0}}
\right\Vert_\infty \\
&\le
\underbrace{\left
\Vert
B \right\Vert_\infty }_{\rm (a)}\times \underbrace{\left\Vert 1- \frac{\hat \Psi_\sigma^*d_\sigma}{d_{\sigma_0}} \right\Vert_\infty}_{\rm (b)} + \underbrace{ \left\Vert\frac{\hat \Psi_\sigma^*\Delta \beta - \Delta \beta_0}{d_{\sigma_0}} \right\Vert_\infty }_{\rm (c)}.
\end{align*}
For $(\beta,\sigma)$ in a small neighbourhood of $(\beta_0,\sigma_0)$ inside $\mathcal I$, we will estimate the terms~(a), (b) and~(c) to the effect that $\big|\big|\hat \Psi_\sigma^*B-B_0\big|\big|_\infty \le ||B_0||_\infty/2$. As a consequence,
\begin{equation*}
| B | \ge |B_0 | - \big|\big|\hat\Psi_\sigma^*B-B_0\big|\big|_\infty \ge |B_0|/2>0\qquad \text{a.e.},
\end{equation*}
such that $(\beta,\sigma)\in \mathcal I_+$, as desired. Let us now consider the three terms separately:
\begin{enumerate}[label=\rm(\alph*)]\itemsep=0pt
	\item By Proposition~\ref{distancebound}, we have $||B||_\infty \le || S_\beta ||_\infty$ and by Lemma \ref{sbetacts} this upper bound depends continuously on $(\beta,\sigma)$. In particular, it stays uniformly bounded for small perturbations inside $\mathcal I$.
	
	\item Let $\epsilon>0$ and $\Sigma_0$ as in Proposition~\ref{STestimate}. Then for $\sigma\in \mathcal U$, this factor is bounded by $\epsilon$ due to Proposition~\ref{distancebound}.

	\item This factor equals the Lipschitz constant of $f=\Psi_\sigma^*\beta- \beta_0$, considered as map $f\colon (Z^\circ, d_{\sigma_0})\rightarrow \C^2$ and will be estimated as in Proposition~\ref{distancebound}\,\ref{lipschitz1}. To this end, we write $S_\bullet$ and $T_\bullet$ (with $\bullet \in \{f,\Psi_\sigma,\beta,\beta_0, \})$ for the Jacobian matrices with respect to the orthonormal frames
$\big\{\bigl(1-|\mu|^4\bigr)\bar{\Xi}^\vee_\sigma,\partial_\mu^\vee\big\}$, $ \big\{\bigl(1-|\mu|^4\bigr)\bar{\Xi}^\vee_{\sigma_0},\partial_\mu^\vee\big\}$ and $\{{\rm d}w,{\rm d}\xi\}$,
 respectively. Then, by~\eqref{commdiagram},
\[
	\begin{bmatrix}
		S_f & \bar T_f\\
		T_f & \bar S_f
	\end{bmatrix}
	=
	\begin{bmatrix}
		S_{\Psi_\sigma} & \bar T_{\Psi_\sigma}\\
		T_{\Psi_\sigma} & \bar S_{\Psi_\sigma}
	\end{bmatrix}
	\begin{bmatrix}
		S_\beta & 0\\
		0 & \bar S_\beta
	\end{bmatrix}
	-
	\begin{bmatrix}
		S_{\beta_0} & 0\\
		0 & \bar S_{\beta_0}
	\end{bmatrix},
\]
and hence
\begin{gather*}
	|| S_f ||_\infty = || S_{\Psi_\sigma}S_\beta - S_{\beta_0}||_\infty \le || (S_{\Psi_\sigma}-\Id) ||_\infty\cdot ||S_\beta||_\infty + ||(S_\beta - S_{\beta_0})||_\infty,\\
	||T_f||_\infty =|| T_{\Psi_\sigma}S_\beta||_\infty \le ||T_{\Psi_\sigma} ||_\infty \cdot || S_\beta||_\infty.
\end{gather*}
For $(\beta,\sigma)$ in a small neighbourhood of $(\beta_0,\sigma_0)\in \mathcal I$, both norms can be made arbitrarily small. This follows from the estimates in Proposition~\ref{STestimate} and Lemma \ref{sbetacts}. Using Proposition~\ref{distancebound}\,\ref{lipschitz1}, we thus arrange that $\Lip (f) \le \epsilon $.
\end{enumerate}

\section[Proofs of Theorems 3.8, 1.3 and 1.4]{Proofs of Theorems \ref{thm_betafamily}, \ref{thm_globalbeta} and \ref{TNNT}}\label{s_tnnt}

In this section, we conclude the proofs of the main theorems and their corollaries.

\subsection{Proof of Theorem~\ref{thm_betafamily}} \label{s_proof_betafamily}
As discussed below the theorem, it only remains to prove the continuous dependence on $\sigma$. Recall that
\[
	\beta_\sigma|_{S\DD} = \bigl(\underbrace{\mathbb S_\mathrm{ev}\bigl( I_0^{g_\sigma}(N_0^{g_\sigma})^{-1}z\bigr)^{\sharp_{g_\sigma}}}_{\beta^{(0)}_\sigma|_{S\DD}}
	, \underbrace{\mathbb S_\mathrm{odd} \bigl( I_1^{g_\sigma}(N_1^{g_\sigma})^{-1}{\rm d}z\bigr)^{\sharp_{g_\sigma}}}_{\beta^{(1)}_\sigma|_{S\DD}} \bigr) \in C^\infty(S\DD),\qquad \sigma\in \Sigma_\mathrm{s},
\]
where the X-ray transform and the $\sharp$-operator are defined with respect to the $g_\sigma$-geodesic flow projected onto the fixed manifold $S\DD$, as described in Section~\ref{s_upsilon}. Combining Proposition~\ref{thm_continuitygalore} and Theorem~\ref{thm_ghtop} then shows that \smash{$\sigma \mapsto \beta^{(0)}_\sigma$} is continuous from $\Sigma_\mathrm{s}$ into $C^\infty(S\DD)$. For \smash{$\beta^{(1)}_\sigma$}, one argues analogously, using instead of Theorem~\ref{thm_ghtop} that also
$\sigma\mapsto N_1^{g_{\sigma}}({\rm d}z)$, $ \Sigma_\mathrm{s} \to \rho^{-1/2}\Omega_1$.
is a continuous map. To see this, one reduces to the attenuated setting as in the proof of Proposition~\ref{prop_isonk} and then uses the continuity statement from Remark \ref{rem_attenuatedverstappen}.
The extension of~$\beta_\sigma|_{S \DD}$ to $\beta_\sigma \in C^\infty(Z)$ is clearly continuous and thus the proof is complete.

\subsection[Proof of Theorem 1.3]{Proof of Theorem~\ref{thm_globalbeta}} \label{proof_globalbeta}
On the disk $M=\DD$, we consider perturbations of a fixed conformally Euclidean metric
\[
	g_\kappa = {\rm e}^{2\sigma_\kappa}|{\rm d}z|^2,\qquad \sigma_\kappa=-\log\bigl(1+\kappa|z|^2\bigr),\qquad |\kappa|<1.
\]
Let $(\beta_\sigma\mid \sigma\in \Sigma_\mathrm{s})$ be the family of canonical $\beta$-maps from Theorem~\ref{thm_betafamily}. For $\sigma=\sigma_\kappa$, we know from Theorem~\ref{thm_constantcurvature} that $\beta_\kappa$ has a holomorphic blow-down structure, that is, the tuple $(\beta_\kappa,\sigma_\kappa)$ lies in the set $\mathcal I_+$ from Section~\ref{s_openness}. As explained below Theorem~\ref{thm_openbds}, the openness of $\mathcal I_+$ implies the existence of a neighbourhood $\mathcal U\subset \Sigma_\mathrm{s}$ of $\sigma_\kappa$, such that the family of $\beta$-maps satisfies~$
	(\beta_\sigma,\sigma)\in \mathcal I_+$ for all $\sigma\in \mathcal U.
$

Let us now view $g_\kappa={\rm e}^{2\sigma_\kappa}|{\rm d}z|^2$ as an element of $\Riem(\DD)$, the space of all Riemannian metrics on $\DD$, and apply the Riemann mapping theorem with continuous dependence (Theorem~\ref{RMT}). Let $P\colon \mathcal \Riem(\DD)\rightarrow C^\infty(\DD,\R)$ be the continuous map from Corollary \ref{cor_rmt}. Then $P(g_\kappa)=\sigma_\kappa$ and~${\mathcal V=P^{-1}(\mathcal U)\subset \Riem(\DD)}$ is an open neighbourhood of $g_\kappa$. Then, for $\sigma = P(g)$ there are maps as follows
\smash{$Z(\DD,g) \xrightarrow{\sim} Z\bigl(\DD,{\rm e}^{2\sigma}|{\rm d}z|^2\bigr) \xrightarrow{\beta_\sigma} \C^2$}.
Here the first map is the biholomorphism that lifts the isometry $(\DD,g)\cong \bigl(\DD,{\rm e}^{2\sigma}|{\rm d}z|^2\bigr)$ provided by Corollary \ref{cor_rmt}.
 The concatenation is then a holomorphic map from twistor space $Z=Z(\DD,g)$ into $\C^2$ with holomorphic blow-down structure.

\subsection[Proof of Theorem 1.4]{Proof of Theorem~\ref{TNNT}} As preparation for the theorem, we first prove the following.

\begin{Proposition}\label{thm_smallballs}
	Let $(M,g)$ be a Riemannian surface and $x_0\in M$. Write
\[
\bar B_\epsilon(x_0) = \{x\in M\mid d_g(x,x_0)\le \epsilon\}
\]
 for the closed geodesic ball around $x_0$ with radius $\epsilon >0$. Then there exists a family of smooth functions $\sigma_\epsilon\colon \DD\rightarrow \R$ $(\epsilon >0)$ such that
	\begin{enumerate}[label={\rm (\roman*)}]\itemsep=0pt
		\item there is an orientation preserving isometry
$\bigl({\bar B_\epsilon(x)},g \bigr) \cong \bigl(\DD,{\rm e}^{2(\sigma_\epsilon+\log \epsilon)}|{\rm d}z|^2\bigr)$;
		\item as $\epsilon \rightarrow 0$, we have $\sigma_\epsilon \rightarrow 0$ in $C^\infty(\DD,\R)$.
	\end{enumerate}
\end{Proposition}

Since this is a local result, we may assume without loss of generality that $M$ is complete. The desired isometry comes from first using geodesic normal coordinates and then the Riemann mapping theorem. Let $\big\{v,v^\perp\big\}\subset T_{x_0}M$ be an oriented $g$-orthonormal basis and define:
\begin{equation*}
\psi_\epsilon\colon\ \DD\rightarrow \bar B_\epsilon(x_0),\qquad \psi(z)= \exp_x\bigl( \epsilon \Re z \cdot v + \epsilon\Im z \cdot v^\perp\bigr)
\end{equation*}

\begin{Lemma}\label{lem_convtoeuclid}
As $\epsilon \rightarrow 0$, we have $\epsilon^{-2} \psi_\epsilon^*g \rightarrow |{\rm d}z|^2$ in $C^\infty\bigl(\DD,\otimes_S^2T^*\DD\bigr)$.
\end{Lemma}

\begin{proof}
We first consider the case $\epsilon =1$, in which case the pull-back $\psi_1^*g$ can be expressed in polar coordinates $(r,\theta)$ as
$\psi_1^*g = {\rm d}r^2 + f(r,\theta)^2 \cdot {\rm d}\theta^2$,
where $f\in C^\infty\bigl([0,1]\times \mathbb S^1\bigr)$ is determined through the Jacobi equation in $r$
\[
		\partial_r^2f + K_1\cdot f = 0 \qquad\text{on} \  [0,1]\times \mathbb S^1, \qquad f(0,\cdot) =0,\qquad \partial_{r}f(0,\cdot) \equiv 1
	\]
Here $K_1(r,\theta)=\psi^*K\bigl(r{\rm e}^{{\rm i}\theta}\bigr)$ is the Gauss curvature of $(M,g)$ in polar coordinates, which can be viewed as a function $K_1\in C^\infty\bigl([0,1]\times \mathbb S^1\bigr)
$. For $\epsilon>0$, we now introduce rescaled polar coordinates~$(s,\theta)$, where $r=\epsilon s$. If we define
\[
	f_\epsilon\in C^\infty\bigl([0,1]\times\mathbb S^1\bigr)\qquad \text{by}\quad f_\epsilon(s,\theta)=\epsilon^{-1}f(\epsilon s,\theta),
\]
then
\[
	{\rm d}r^2 + f(r,\theta)^2 {\rm d}\theta^2 = \epsilon^2 {\rm d}s^2 + f(\epsilon s,\theta)^2 {\rm d}\theta^2 = \epsilon^2\bigl({\rm d}s^2 + f_\epsilon(s,\theta)^2{\rm d}\theta^2\bigr),
\]
which is to say that $\epsilon^{-2} \psi_\epsilon^* g = {\rm d}s^2 + f_\epsilon^2 {\rm d}\theta^2$. Let $f_0(s,\theta)=s$, then the lemma is equivalent to
\[
	f_\epsilon \rightarrow f_0 \qquad \text{in}\quad C^\infty\bigl([0,1]\times \mathbb S^1\bigr).
\]
To prove the latter convergence result, we derive an ODE for $f_\epsilon$ and conclude by its continuous dependence on the parameters. Note that
\begin{gather*}
	\partial_s f_\epsilon(s,\theta)=\partial_r f(\epsilon s,\theta),\qquad
\partial_s^2f_\epsilon(s,\theta) =\epsilon \partial_r^2f(\epsilon s,\theta)= - \epsilon K_1(s\epsilon,\theta) f(\epsilon s,\theta).
\end{gather*}
Hence, if we define $K_\epsilon \in C^\infty\bigl([0,1]\times \mathbb S^1\bigr)$ by $K_\epsilon(s,\theta)=\epsilon^2 K_1(\epsilon s,\theta)$, then
\[
		\partial_s^2f_\epsilon + K_\epsilon f_\epsilon = 0 \qquad\text{on} \  [0,1]\times \mathbb S^1, \qquad f_\epsilon(0,\cdot) =0,\qquad \partial_{s}f_\epsilon(0,\cdot) \equiv 1.
	\]
As $K_\epsilon\rightarrow 0$ in $C^\infty\bigl([0,1]\times \mathbb S^1\bigr)$ for $\epsilon\rightarrow 0$, we can apply Corollary \ref{cor_pardep} to conclude.
\end{proof}

\begin{proof}[Proof of Proposition~\ref{thm_smallballs}]
Let $P\colon \Riem(\DD)\rightarrow C^\infty(\DD,\R)$ the map from Corollary \ref{cor_rmt}, where~${\Riem(\DD)\subset C^\infty\bigl(\DD,\otimes^2_ST^*\DD\bigr)}$ is the open cone of Riemannian metrics on $\DD$. With $\psi_\epsilon\colon \DD\rightarrow M$ as above, we define
$\sigma_\epsilon = P\bigl(\epsilon^{-2}\psi_\epsilon^*g\bigr)$, $ \epsilon >0$.
The preceding lemma, together with the continuity of $P$, imply that $\sigma_\epsilon\rightarrow 0$ in $C^\infty(\DD,\R)$, as $\epsilon\rightarrow 0$. Further, there are isometries
\[
	\bigl(\bar B_\epsilon(0),g\bigr)\cong (\DD,\psi_\epsilon^* g)\cong \bigl(\DD, {\rm e}^{2\sigma_\epsilon + 2\log \epsilon}|{\rm d}z|^2\bigr),
\]
provided by $\psi_\epsilon$ and the isometry from Corollary \ref{cor_rmt}, respectively.
\end{proof}

\begin{proof}[Proof of Theorem~\ref{TNNT}] Let $B_\epsilon(x_0)$ be the open geodesic ball of radius $\epsilon$ about a point $x_0\in M$.
For $\epsilon>0$ sufficiently small, we will obtain a local $\beta$-map on the open subset $U=Z(B_\epsilon(0),g)$ of twistor space $Z$. As any point in $Z$ is contained in such an open set $U$, this will prove the theorem.

Let $(\sigma_\epsilon\mid \epsilon >0)$ be the family of conformal factors from Proposition~\ref{thm_smallballs} and let $(\beta_{\sigma_\epsilon}\mid \epsilon>0)$ be the associated $\beta$-maps from Theorem~\ref{thm_betafamily}. We define a local $\beta$-map through the following concatenation:
\[
	\beta\colon U \cong Z\bigl(\DD^\circ,{\rm e}^{2\sigma_\epsilon+2\log \epsilon}|{\rm d}z|^2\bigr) \cong Z\bigl(\DD^\circ,{\rm e}^{2\sigma_\epsilon}|{\rm d}z|^2\bigr) \xrightarrow{\beta_{\sigma_\epsilon}}\C^2.
\]
Here the first biholomorphism is the lift of the isometry in Proposition~\ref{thm_smallballs} and the second one equals, in the model from Section~\ref{s_conformallyeuclidean}, simply the identity---this is also a biholomorphism, because the complex structure of twistor space is invariant under constant rescalings of the metric (that is, $\D_{\sigma} = \D_{\sigma+C}$ for any constant $C\in \R$).
For sufficiently small $\epsilon$, the map $\beta_{\sigma_\epsilon}$ has a holomorphic blow-down structure by Theorem~\ref{thm_openbds}
and in view of the discussion below Definition~\ref{def_bds} this implies properties \ref{thmtnnt1} and \ref{thmtnnt2}. The statement about point separation is immediate: If $q_1,q_2\in U\backslash SM$ are two distinct points, then $\beta(q_1)\neq \beta(q_2)$ and thus there is a holomorphic map $f\colon \C^2\rightarrow \C$ with $f(\beta(q_1))\neq f(\beta(q_2))$. Since $\beta$ is holomorphic, we have $\beta^*f\in \mathcal O(U)$, as desired.\end{proof}

\subsection[Proofs of Corollaries 1.5 and 1.6]{Proofs of Corollaries \ref{cor_curves} and \ref{cor_maps}}

\begin{proof}[Proof of Corollary \ref{cor_curves}]
	We consider the set $\Sigma'=\{\mu \in \Sigma\mid f_1 =f_2 \text{ in a neighbourhood of } \mu\}$,
which is obviously open in $\Sigma$. We are done once we have shown that it is also non-empty and closed.

{\it Non-empty:} Let $\mu_* \in \partial \Sigma$ and $p=f_1(\mu_*)=f_2(\mu_*)\in Z$. Let $U$ be as in Theorem~\ref{TNNT} and~${V=f_1^{-1}(U)\cap f_2^{-1}(U)}$. This is an open neighbourhood of $\mu_*$ in $\DD$. If $g\colon U\to \C$ is holomorphic, then also the functions $g\circ f_1$ and $g\circ f_2$ are holomorphic and agree on $V\cap \partial \Sigma$ (which is non-empty, as it contains $\mu_*$), hence $g\circ f_1 = g\circ f_2$ on all of $V$ by the identity principle. By Theorem~\ref{TNNT}, the holomorphic functions on $U$ separate points and thus we must have $f_1=f_2$ on~$V$, which shows that $\mu_*\in \Sigma$.

{\it Closed:} Let $(\mu_n)\subset \Sigma'$ be a sequence with limit $\mu_*\in \Sigma$. Since $f_1(\mu_n)=f_2(\mu_n)$, we also have $f_1(\mu_*)=f_2(\mu_*)\in Z$ and call this point $p$. Again, let $U$ be a neighbourhood of $p$ as in the Theorem~\ref{TNNT} and let $V=f_1^{-1}(U)\cap f_2^{-1}(U)$. If $g\colon U\to \C$ is holomorphic, then $h=g\circ f_1-g\circ f_2$ is holomorphic on $V$ and satisfies $h(\mu_n)=0$ for infinitely many $n$. By the identity principle~${h\equiv 0}$ on $V$ and, again using point separation, $f_1=f_2$ on $V$. This shows that $\mu_*\in \Sigma$.
\end{proof}

\begin{proof}[Proof of Corollary \ref{cor_maps}]
	Let $x\in M_1$ and let $f\colon \DD\rightarrow Z_{1,x}\subset Z$ be a holomorphic curve whose image is the fibre $Z_{1,x}$ over $x$. Then both $\Phi\circ f$ and $\Psi\circ f$ are holomorphic curves and by assumption they have the same boundary values. Hence $\Phi\circ f = \Psi\circ f$, which means that $\Phi$ and~$\Psi$ agree on the fibre $Z_{1,x}$. This holds for all $x\in M_1$ and thus $\Phi=\Psi$.
\end{proof}

\appendix

\section{Continuity theorems in geometry}\label{s_appa}

We recapitulate some aspects from geometric analysis with a focus on continuous dependence on the underlying geometry.

\subsection{Riemann mapping theorem}\label{s_RMT}
Define
\begin{gather*}
\Riem(\DD)= \big\{g\in C^\infty\bigl(\DD,\otimes_S^2T^*\DD\bigr)\mid g\text{ is a Riemannian metric}\big\},\\
\Bel(\DD) = \{m\in C^\infty(\DD,\C)\mid |m| < 1\}.
\end{gather*}
These are the spaces of {\it Riemannian metrics} and {\it Beltrami coefficients} on $\DD$, respectively, and they both lie open in the ambient $C^\infty$-spaces and as such they are Fr{\'e}chet manifolds. Each~${m\in \Bel(\DD)}$ encodes an orientation compatible complex structure on $\DD$ with anti-holomorphic tangent bundle
$
\D_m = \mathrm{span}_\C(\partial_{\bar z} + m \partial_z)\subset T_\C \DD$,
and vice versa, every complex structure on $\DD$ that is compatible with the standard orientation is given in terms of such a~Beltrami coefficient. In particular, the complex structure determined by $g\in \Riem(\DD)$ is given~by
\begin{equation*}
m_g(z)=\frac{\big\langle {\rm d}z,v+{\rm i}v^\perp\big\rangle}{\big\langle {\rm d}\bar z,v+{\rm i}v^\perp\big\rangle},\qquad \text{ where } \big\{v,v^\perp\big\}\subset T_z \DD \text{ is an oriented $g$-ONB}.
\end{equation*}
 Let $\Diff_+(\DD)$ be the orientation preserving diffeomorphisms of $\DD$, equipped with the $C^\infty$-topology. Moreover, let $\Diff_{+}(\DD,\{0,1\})\subset \Diff_+(\DD)$ be the closed subset of diffeomorphisms with $\psi(0)=0$ and $\psi(1)=1$.
 Then the following version of the Riemann mapping theorem was proved by
 Earle--Schatz in \cite[Theorem on p.~171]{EaSc70}.

\begin{Theorem}[RMT with continuous dependence]\label{RMT}
	For any $m\in \Bel(\DD)$, there exists a unique $\psi\in \Diff_+(\DD,\{0,1\})$ such that $\psi\colon (\DD,\D_m)\rightarrow (\DD,\D_0)$ is holomorphic $(\psi_*\D_m=\D_0)$. The map~${m\mapsto \psi}$ is a homeomorphism $
	\Bel(\DD)\cong \Diff_+(\DD,\{0,1\})$.
\end{Theorem}

\begin{Corollary}\label{cor_rmt}
There is a continuous map $P\colon \Riem(\DD)\rightarrow C^\infty(\DD,\R)$ such that
\begin{enumerate}[label=\rm(\roman*)]\itemsep=0pt
	\item $P\bigl({\rm e}^{2\sigma}|{\rm d}z|^2\bigr)=\sigma$ for all $\sigma\in C^\infty(\DD,\R)$;
	
	\item $(\DD,g) \cong \bigl(\DD,{\rm e}^{2\sigma}|{\rm d}z|^2\bigr)$ isometric for $\sigma=P(g)$ and $g\in \Riem(\DD)$.
\end{enumerate}
\end{Corollary}

\begin{proof}
Let $\psi \in \Diff_+(\DD,\{0,1\})$ the image of $m_g$ under the homeomorphism in Theorem~\ref{RMT}. Write $\phi=\psi^{-1}$, let $\mathbf 1\in C^\infty(\DD,T\DD)$ be a vector field with $|\mathbf 1|_{|{\rm d}z|^2}\equiv 1$ and define
$
	P(g) := \log |\mathbf 1|_{\phi^*g}.
$ Using that $\Diff_+(\DD,\{0,1\})$ is a topological group that acts continuously on $\Riem(\DD)$ by pull-backs, one readily verifies that $P$ is continuous.
Let us verify the two properties: If~${g={\rm e}^{2\sigma}|{\rm d}z|^2}$, then $m_g=0$ and $\psi=\phi=\Id$, such that \smash{$P(g)=\log (e^\sigma|\mathbf 1|_{|{\rm d}z|^2}) =\sigma$}. For general $g\in \Riem(\DD)$, the map $\phi=\psi^{-1}\colon \bigl(\DD,|{\rm d}z|^2\bigr) \rightarrow (\DD,g)$ is conformal, which is to say that~${\phi^*g={\rm e}^{2\sigma}|{\rm d}z|^2}$ for {\it some} $\sigma\in C^\infty(\DD,\R)$. Clearly $P(g)=\sigma$, and the proof is complete.
\end{proof}

\subsection{Integration of vector fields} Let $\mathcal M$ be a closed $d$-dimensional manifold and $T\ge 1$. Let $\mathcal M_T=\mathcal M\times [0,T]$. Then an {\it isotopy} on $\mathcal M$ is a smooth map $\Phi\colon \mathcal M_T\rightarrow \mathcal M$
such that $\Phi_t(x)=\Phi(x,t)$ defines a diffeomorphism of $\mathcal M$ for all $0\le t \le T$, with $\Phi_0=\operatorname{Id}_\mathcal M$. We denote the set of all isotopies on $\mathcal M$ by
\[
	\mathrm{Iso}_T(\mathcal M) \subset \{\Phi\in C^\infty(\mathcal M_T,\mathcal M)\mid \Phi(\cdot,0)=\Id_\mathcal M\}.
\]
The set on the right lies closed inside the ambient $C^\infty$-space and contains $\mathrm{Iso}_T(\mathcal M)$ as an open subset.
To every isotopy $\Phi$, we can associate a time-dependent vector field $X$, i.e., an element of $C^\infty(\mathcal M_T,T\mathcal M)$, via the relation
$\partial_t \Phi(x,t) = X(\Phi_t(x),t)$, $ (x,t)\in \mathcal M_T$.
Vice versa, every such vector field $X$ can be integrated to an isotopy $\Phi^X$ by solving this ordinary differential equation. We record here the---certainly well-known---observation that the map~${X\mapsto \Phi^X}$ is continuous in the $C^\infty$-topology.

\begin{Theorem}\label{thm_isotopies} The map $X\mapsto \Phi^X$ is a homeomorphism $C^\infty(\mathcal M_T,T\mathcal M)\cong \mathrm{Iso}_T(\mathcal M)$.
\end{Theorem}

\begin{proof}[Sketch of the proof]
The result can be reduced to the continuity of the exponential map
\[
	X\mapsto \Phi^X_1,\qquad C^\infty(\mathcal M,T\mathcal M)\rightarrow \mathrm{Diff}(\mathcal M)
\]
for time-independent vector fields, which is proved by Ebin--Marsden in \cite[Theorem 3.1]{EbMa70}. Precisely, the authors prove continuity into $\Diff^s(\mathcal M)$, the space of diffeomorphisms of Sobolev $H^s$-regularity, for all $s>d/2+2$. Since $\Diff(\mathcal M)$ carries the initial topology with respect to the inclusions into $\Diff^s(\mathcal M)$, this gives the desired continuity.

To obtain continuity of the full isotopy, one can embed $\mathcal M_T$ into a closed manifold $\mathcal N$ of dimension $d+1$ and consider a smooth vector field $Y$ on $\mathcal N$ with $Y(x,t)=tX(x)$ for $(x,t)\in \mathcal M_T\subset \mathcal N$. The isotopy $\Psi^Y\in \mathrm{Iso}_1(\mathcal N)$ leaves $\mathcal M_T$ invariant and satisfies
\[\Psi_s^Y(x,t)=\bigl(\Phi_s^{tX}(x),t\bigr)=\bigl(\Phi^X_{st}(x),t\bigr),\qquad (x,t)\in \mathcal M_T\subset \mathcal N, 0\le s \le 1.\]
The full isotopy $\Phi^X\in \mathrm{Iso}_T(\mathcal M)$ can be recovered from $\Psi_1^Y$, and after constructing a continuous extension map $X\mapsto Y$, this allows to derive continuity of the map
\[
 X\mapsto \Phi^X,\qquad C^\infty(\mathcal M,T\mathcal M)\rightarrow \mathrm{Iso}_T(\mathcal M).
\]	
The passage to time-dependent vector fields is done similarly, by extending $X$ to a time-independent vector field on $\mathcal N$ and then applying the preceding result. This proves that also the following map is continuous:
\[
 X\mapsto \Phi^X,\qquad C^\infty(\mathcal M_T,T\mathcal M)\rightarrow \mathrm{Iso}_T(\mathcal M).	
\]
It is clear that this is a bijection, with inverse provided by differentiating a given isotopy. This is clearly continuous and so we indeed obtain a homeomorphism.
\end{proof}

For convenience, we record the following result as a corollary (it could also be derived using variation of constants).

\begin{Corollary}\label{cor_pardep}
Let $M$ be a closed manifold and $M_T=M\times[0,T]$. For $A\in C^\infty(M_T)$ and $a_0,b_0\in C^\infty(M)$, denote with $f_A\in C^\infty(M_T)$ the unique solution to the initial value problem
\[
	\partial_{t}^2f + Af = 0 \qquad\text{on}\  M_T,\qquad (f,\partial_tf)\vert_{t=0} = (a_0,b_0).
\]	
Then the map $A\mapsto f_A$ is continuous from $C^\infty(M_T)$ into $C^\infty(M_T)$.
\end{Corollary}

\begin{proof}
	On $\mathcal M = M_x \times \R_a\times \R_b$ consider the time-dependent vector field $X(x,a,b,t)={b \partial_a - A(x,t)a \partial_b}$, whose isotopy satisfies $\Phi^X(x,a_0(x),b_0(x),t)=(x,f_A(x,t),\partial_t f_A(x,t))$ for all ${(x,t)\in M_T}$. Projecting onto the second component, we retrieve $f_A$ and the corollary follows from the preceding theorem.
\end{proof}

\subsection{Riemannian extensions}

\begin{Lemma}\label{lem_rext}
Let $M$ be a compact manifold with boundary, embedded into a closed manifold~$N$ of the same dimension. Then for any Riemannian metric $g_0$ there is a neighbourhood $\mathcal U$ inside~$C^\infty\bigl(M,\otimes^2_ST^*M\bigr)$ and a continuous map
\[
	E\colon\ \mathcal U \rightarrow C^\infty\bigl(N,\otimes^2_ST^*N\bigr), \qquad Eg|_{M}=g
\]
such that $Eg$ is a Riemannian metric on $N$ for all $g\in \mathcal U$.
\end{Lemma}

 \begin{proof} First, there is a continuous {\it linear} extension operator
\begin{equation*}
E_0\colon\ C^\infty\bigl(M,\otimes_S^2T^*M\bigr)\rightarrow C^\infty\bigl(N,\otimes_S^2T^*N\bigr),\qquad E_0f\vert_M = f,
\end{equation*}
this is proved, e.g., by Seeley in \cite{See64}. Let $h$ be an arbitrary metric on $N$ and let $\epsilon>0$ be so small that the symmetric $2$-tensor $g_0 - 2\epsilon h$ is positive definite on $M$. For $f\in C^\infty\bigl(M,\otimes_S^2T^*M\bigr)$, define
\begin{equation*}
U(f,\delta) =\{x\in N\mid (E_0f-\delta h)_x \text{ is positive definite}\}, \qquad \delta>0.
\end{equation*}
Then $U_0=U(g_0,2\epsilon)$ is an open neighbourhood of $M$ in $N$. Let $\chi\in C^\infty(S,[0,1])$ be $\equiv 1$ in a~neighbourhood of $M$ and $\equiv 0$ on $N\backslash U_0$. We claim that
\[
\mathcal U :=\bigl\{g\in C^\infty\bigl(M,\otimes^2_S T^*M\bigr)\mid U(g,\epsilon)\supset \bar U_0\bigr\},\qquad Eg := \chi E_0g + (1-\chi)h
\]
have the desired properties. Clearly, $\mathcal U$ is open and contains $g_0$. It remains to show that $(Eg)_x$ is positive definite for all $x\in N$. If $\chi(x)=0$, this is obvious. If $\chi(x)>0$, then $x\in \bar U_0\subset U(g,\epsilon)$, and hence each term in
\begin{equation*}
(Eg)_x = \chi(x)(E_0g-\epsilon h)_x + (1+\epsilon \chi(x) - \chi(x))h_x
\end{equation*}
is positive definite.
\end{proof}

\begin{Remark}\label{lem_rext2}
It is sometimes convenient to define an extension operator on an {\it arbitrary} open subset $\mathcal V\subset C^\infty\bigl(M,\otimes^2_ST^*M\bigr)$ of a Riemannian metric. This is always possible by means of a continuous partition of unity on $\mathcal V$. (The $C^\infty$-topology being metrisable implies that $\mathcal V$ is paracompact.)

\end{Remark}

\subsection{Extension of embeddings} Let $Y$ be a compact manifold with boundary and let $X$ be a closed manifold. The boundary restriction of $C^\infty$-embeddings,
$\operatorname{Emb}(Y,X)\to \operatorname{Emb}(\partial Y,X)$, $ f \mapsto f\vert_{\partial Y}$,
is a continuous map with respect to the $C^\infty$-topologies. Palais showed that this map is a locally trivial fibration (see \cite{Pal60}, the case with boundary is discussed in Section 6). We record a simple consequence of this.

\begin{Proposition}\label{prop_palais}
The image $\mathcal U\subset \operatorname{Emb}(\partial Y,X)$ of the restriction map is open and there exists a continuous extension map
$E\colon \mathcal U\to \operatorname{Emb}(Y,X)$, $ Ef|_{\partial Y} = f$.
\end{Proposition}

\subsection{Whitney folds}\label{s_whitney}

Let $X$ and $Y$ be closed orientable $n$-manifolds and $S\subset Y$ a closed embedded hypersurface. We wish to consider smooth maps and involutions
$(f,\alpha)\in C^\infty(Y,X)\times \Diff(Y)$ with $ f\circ \alpha = f $ and~${\alpha\circ \alpha=\Id}$,
such that $f$ has a~{\it Whitney fold} at every point in $S$ (see also Section~\ref{s_calpha}), having~$\alpha$ as its non-trivial involution. The purpose of this section is to describe a setting where the local theory from \cite[Section C.4]{Hor07} can be globalised and to supplement the theory with adequate continuity theorems.

To ensure that $Y$ supports global involutions with fixed point set $S$, we will from now on impose the following condition: {\it $Y$ is the union of two connected submanifolds $Y_1,Y_2\subset Y$ having disjoint interiors and a common boundary $\partial Y_1=\partial Y_2 =S$}. To further ease the discussion, we fix two volume forms $\omega_X$ and $\omega_Y$ and, given a smooth map $f\colon Y\to X$, write $\lambda_f\in C^\infty(Y,\R)$ for the function defined by $f^*\omega_X=\lambda_f \omega_Y$.

\begin{Definition}\label{def_whitney} Define the following sets:
\begin{align*}
	\mathcal W_\loc^S(Y,X)&= \left\{f\in C^\infty(Y,X)\mid \begin{array}{ll}{\rm (1)}&\lambda_f=0 \text{ on } S\\
	 {\rm (2)}&|d\lambda_f| + |{\rm d}f| > 0 \text{ pointwise on } TY|_S \\
	 {\rm (3)}&f|_S\in \operatorname{Emb}(S,X)
	 \end{array}\right\},\\
	 \mathcal W^S(Y,X) &=\left\{f\in \mathcal W^S_\loc(Y,X)\mid \begin{array}
{ll}
 {\rm (4) } & |\lambda_f|>0\text{ pointwise on } Y\backslash S\\
{\rm (5)}& f|_{Y_k}\colon Y_k \to X \text{ is a topological}\\
	 &\text{embedding for } k=1,2\\
	 {\rm (6)}&f(Y_1)=f(Y_2)
	 \end{array}\right\}.
\end{align*}
\end{Definition}

\begin{Proposition}[local and global folds and their involutions] \
	\begin{enumerate}[label=\rm(\roman*)]\itemsep=0pt
		\item $\mathcal W_\loc^S(Y,X)$ is an open subset of $\{f\in C^\infty(Y,X)\mid \lambda_f = 0\text{ on } S\}$;
		
		\item for all $f\in \mathcal W_\loc^S(Y,X)$ there exists an $\epsilon>0$ and embeddings \smash{$\varphi_X^f\colon  S\times [-\epsilon,\epsilon]\to X$} and \smash{$\varphi_Y^f\colon  S\times [-\epsilon,\epsilon]\to Y$}
		such that the following diagram commutes:
		\[
		 	\begin{tikzcd}
		 		S \arrow["{y\mapsto (y,0)}"{yshift=1ex}]{r} \arrow[hookrightarrow]{rd} & S\times [-\epsilon,\epsilon] \arrow["{(y,t)\mapsto \bigl(y,t^2\bigr)}"{yshift=1ex}]{r} \arrow["{\varphi^f_Y}"]{d} & S\times [-\epsilon,\epsilon] \arrow["{\varphi^f_X}"]{d}\\
		 		& Y \arrow["{f}"]{r} & X.
		 	\end{tikzcd}
		\]
		\item every $f_0\in \mathcal W_\loc^S(Y,X)$ has an open neighbourhood $\mathcal U$ such that for $f\in \mathcal U$ one can choose $\epsilon$ to be uniform and \smash{$f\mapsto \bigl(\varphi^f_X,\varphi^f_Y\bigr)$} to be continuous as a map
		\[
		\mathcal U\to \operatorname{Emb}(S\times [-\epsilon,\epsilon],X) \times \operatorname{Emb}(S\times [-\epsilon,\epsilon],Y).
		\]
		\item $\mathcal W^S(Y,X)\subset \mathcal W^S_\loc(Y,X)$ is open;

		\item for every $f\in \mathcal W^S(Y,X)$, there exists a unique $\alpha \in \Diff(Y)\backslash \{\Id\}$ with $f\circ \alpha = f$ and $\alpha^2 = \Id$. The following map is continuous $f\mapsto \alpha$, $ \mathcal W^S(Y,X) \to \Diff(Y)$.

			\end{enumerate}
\end{Proposition}

\begin{proof}
Part (i) is obvious. For (ii), we may choose $\varphi_X\colon S\times [-1,1]\to X$ to be any embedding that extends $f$ in the sense that $\varphi_X(y,0)=f(y)$ for all $y\in S$.
Then in a neighbourhood~$U$ of~${S\subset Y}$ the map
\smash{$\pi:=\mathrm{pr}_1\circ
(\varphi_X)^{-1}	\circ f\colon  U \to S$},
is a submersion with $\ker d\pi|_S = \ker {\rm d}f|_S$ being transversal to $S$. Hence, after shrinking $U$ if necessary, $\pi\colon U\to S$ is a fibre bundle with fibre~$\R$. As $S\subset U$ is two-sided, there exists a trivialisation
$\tau\colon S\times \R \xrightarrow{\sim} U$
with $\tau(y,0)=y$ for all $y\in S$. Define $F\colon S\times \R\to \R$ by~${F(y,s) = \mathrm{pr}_2\circ \varphi_X^{-1}\circ f\circ\tau(y,s)}$. Then for all $y\in S$ we have~$F(y,0) = \partial_s F(y,0)=0 \neq \partial_s^2 F(y,0)$, the last `$\neq$' being due to property (2). For sufficiently small $c,\delta>0$, we have~$
	\partial_s^2 F(y,s) \ge c$ ($(y,s)\in S\times [-\delta,\delta]$),
possibly after flipping a sign in the definition of~$\varphi_X$.
Taylor expanding~$F(y,\cdot)$ about $s=0$ yields
\[
	F(y,t) = s^2 R(y,t),\qquad R(y,s) = \int_0^1(1-\sigma) \partial_{s}^2F(y,s\sigma) {\rm d}\sigma \ge c/2.
\]
Consequently, $\sqrt{R(y,s)}$ defines a smooth function on $S\times [-\delta,\delta]$ and $s\mapsto s \sqrt{R(y,s)}$ has a~positive derivative for $|s| < c/||\partial_s R||_{\infty}$. Decreasing $\delta$ if necessary, the map
\[
\kappa\colon\ S\times [-\delta,\delta]\to S\times \R,\qquad \kappa(y,s) = \bigl(y, s \sqrt{R(y,s)} \bigr)
\] thus becomes a diffeomorphism onto its image. If $\epsilon>0$ is so small that $S\times [-\epsilon,\epsilon]$ lies in the image, then $\varphi_Y = \tau\circ \kappa^{-1}\colon S\times [-\epsilon,\epsilon]\to Y$ is the desired embedding.

For (iii), let us revisit the choices made in the construction of the embeddings $\varphi_X$ and $\varphi_Y$. The definition of $\varphi_X$ relies on extending the smooth embedding $f\colon S\to X$ and this can be made continuous in $f$ due to Proposition~\ref{prop_palais}. Next, for $f$ near $f_0$ the choice of trivialisation~$\tau$
can be fixed as follows: Upon choosing a Riemannian metric on $Y$, there is a unique unit vector field $\xi_0$ spanning $\ker d\pi$ near $S$ and pointing inside $Y_1$ along $S$. We may then define~$\tau_0(\cdot,s)$ as the time-$s$-map of its flow, at least if $|s|$ is sufficiently small. The trivialisation $\tau$ is then obtained from $\tau_0$ by reparametrising the $s$-parameter, which can be done uniformly for $f$ in a~neighbourhood of~$f_0$. With all choices fixed, the remaining construction of $\varphi_Y$ relies only on elementary calculus and as such is continuous in $f$.

For (iv), we fix $f_0\in \mathcal W^S(Y,X)$ and let $\mathcal U$ and $\epsilon>0$ be as in (iii). Then the sets
\[
	K_k = Y_k \backslash \varphi_Y^{f_0}\left(S\times (-\epsilon/2,\epsilon/2)\right), \qquad k\in \{1,2\},
\]
are compact and connected. Consider the following conditions on $f\in \mathcal U$:
\begin{alignat*}{3}
&	{\rm(a)} \ Y_k\backslash K_k^\circ \subset \varphi_Y^f\left(S\times (-\epsilon,\epsilon)\right),\qquad && {\rm (b)} \  \inf_{K_k} |\lambda_f| > 0,&\\
&	{\rm (c)} \ f(K_k)\cap f(S)=\varnothing,\qquad &&{\rm (d)} \ f(K_1^\circ) \cap f(K_2^\circ) \neq \varnothing.&
\end{alignat*}
These conditions are all satisfied for $f_0$ and further they define an open subset of $\mathcal U$. We claim that they imply that $f\in \mathcal W^S(Y,X)$, proving that $f_0$ is an interior point.

To prove the claim, suppose that $f\in \mathcal U$ satisfies conditions (a)--(d). Then property (4) can be checked pointwise on the sets $K_1,K_2$ and a tubular neighbourhood of $S$, where it follows from (b) or the normal form of $f$, respectively.
Due to (c) and (d), there is a unique connected component $W$ of $X\backslash f(S)$ that contains both $f(K_1)$ and $f(K_2)$. In order to verify conditions~(5) and (6) it suffices to show that
$f\colon Y_k \to \bar W$ is a homeomorphism $(k=1,2)$.
To see this, we first observe that $f$ is a {\it local} homeomorphism in this setting. At points in $K_k$ this follows from (4) and (c) and at points in $Y_k\backslash K_k$ this can be verified in the local normal form. As both~$Y_k$ and~$\bar W$ are compact and connected, it follows that $f\colon Y_k\to \bar W$ is a covering map and due to (3) the degree has to be $1$.

Finally, for (v) the involution $\alpha$ can be defined by $\alpha|_{Y_{k}} = f|_{Y_{j}}^{-1}\circ f|_{Y_k}$ ($\{j,k\} =\{1,2\}$) and is clearly unique. Moreover, $\alpha$ leaves $Y\backslash S$ invariant and as an element $\alpha|_{Y\backslash S} \in \Diff(Y\backslash S)$ it depends continuously on $f\in \mathcal W^S(Y,X)$, because it is explicitly given in terms of $f$ and its inverse. For $(x,t)\in S\times (-\epsilon,\epsilon)$, we have $\alpha\circ\varphi_X^f(y,t) = \varphi_X^f(y,-t)$ and thus using $\varphi_X^f$ as a local chart we see that also near $S$ the involution $\alpha$ is smooth and depends continuously on~$f$.
\end{proof}

The main theorem about pull-backs by Whitney folds---stated here in its global form---asserts that the following non-linear sequence is exact
\begin{equation}\label{seswhitney}
	\mathcal W^S(Y,X)\times C^\infty(X)\to \mathcal W^S(Y,X)\times C^\infty(Y)\to C^\infty(Y).
\end{equation}
Here we first map $(f,w) \mapsto (f,f^*w)$ and then $(f,h)\mapsto h\circ \alpha- h$, where $\alpha$ is the involution associated to $f$. Exactness means that the image of the first map equals the $0-$preimage of the second map. For the latter, we also write
\[
	\mathcal X = \big\{(f,h)\in \mathcal W^S(Y,X)\times C^\infty(Y)\mid h = h\circ \alpha\big\}.
\]

\begin{Theorem}\label{thm_seswhitney} The non-linear sequence \eqref{seswhitney} admits a continuous left-split. Precisely, there exists a continuous map
$(f,h)\mapsto w$, $ \mathcal X \to C^\infty(X)$
such that $f^*w = h$.
\end{Theorem}

\begin{proof}[Sketch of the proof]
	The assertion is non-trivial only near the fold, where one may pass to the normal form provided in the preceding proposition. Writing
	\[
		\mathcal X_0 = \{h\in C^\infty\left(S\times (-\epsilon,\epsilon)\right)\mid h(y,t) = h(y,-t) \},
	\]
	it then suffices to find a continuous {\it linear} map $h\mapsto w $, $ \mathcal X_0 \to C^\infty(S\times (-\epsilon,\epsilon))$
	such that $w\bigl(y,t^2\bigr)=h(y,t)$ for all $(y,t)\in S\times (-\epsilon,\epsilon)$. This however can be achieved by the same methods as in \cite[Theorem C.4.4]{Hor07}, using the open mapping theorem to establish continuity.
\end{proof}

\section{Continuous dependence of the normal operator} \label{s_appb}

We recapitulate the proof that the {\it normal operator} of a simple manifold (defined in Section~\ref{s_geoxray}) is a classical pseudodifferential operator and show that as such it depends continuously on the metric.

\subsection[Classical psido and their topology]{Classical $\boldsymbol{\psi}$do and their topology}
A smooth function $a\colon \R^d_x\times \R^d_\xi\rightarrow \C$ is called a {\it classical symbol} of order $m\in \R$, if
\begin{equation}\label{classical2}
	a(x,\xi)=|\xi|^m \hat a(x,\xi/|\xi|,1/|\xi|),\qquad |\xi|>0,
\end{equation}
for some $\hat a \in C^\infty\bigl(\R^d\times \mathbb S^{d-1}\times [0,\infty)\bigr)$. The space of all such symbols is denoted with $S_\cl^m\bigl(T^*\R^d\bigr)$ and it is turned into a Fr{\'e}chet space by the identification with
\[
	\big\{(a,\hat a)\in C^\infty\bigl(\R^d\times \R^d\bigr)\times C^\infty\bigl(\R^d\times \mathbb S^{d-1}\times [0,\infty)\bigr)\mid \text{\eqref{classical2} holds}\big\},
\]
which is closed in the $C^\infty$-topology. A symbol $a\in S^m_\cl\bigl(T^*\R^d\bigr)$ is called {\it even}, if the function $\hat a$ from \eqref{classical2} has a smooth extension to $\R^d\times \mathbb S^{d-1}\times \R$ with the property
\[
	\hat a(x,-\omega,h)= \hat a(x,\omega,-h),\qquad (x,\omega,h)\in \R^d\times \mathbb S^{d-1}\times \R.
\]
The set of even symbols defines a closed subspace $S^m_\ev\bigl(T^*\R^d\bigr)$ of $S^m_\cl\bigl(T^*\R^d\bigr)$.

\begin{Remark} Taylor expanding $\hat a(x,\omega,h)$ in $h=0$ yields an asymptotic expansion
\[
	a(x,\xi) \sim \sum_{k\ge 0} a_k(x,\xi),\qquad a_k(x,\xi)=|\xi|^{m-k} \partial_h^k\hat a(x,\xi/|\xi|,0)/k!,\qquad |\xi|>0,
\]
which is frequently taken as equivalent definition of classicality (cf.\ \cite{Shu01,Wun13}). From this point of view, evenness of $a$ is the assertion that $a_k(x,-\xi)=(-1)^{k} a_k(x,\xi)$ for all $k\ge 0$ (cf.~\cite[Definition 4.6]{MNP19a}).
\end{Remark}

For $a\in S^m_\cl\bigl(T^*\R^d\bigr)$, the {\it quantisation map}
\[
	\Op\colon\ S^m_\cl\bigl(\R^d\bigr)\rightarrow \mathcal D'\bigl(\R^d\times \R^d\bigr),\qquad \Op(a)(x,y)=
	\frac 1{(2\pi)^d} \int_{\R^d}{\rm e}^{{\rm i}(x-y)\xi}a(x,\xi) {{\rm d}\xi},
\]	
interpreted as oscillatory integral, is well-defined, continuous and injective. Identifying operators with their Schwartz kernels, the space of classical pseudodifferential operators on $\R^d$ can then be defined as
$\Psi^m_\cl\bigl(\R^d\bigr) = \Op\bigl(S^m_\cl\bigl(T^*\R^d\bigr)\bigr) + C^\infty\bigl(\R^d\times \R^d\bigr)$.
To topologise this we give an alternative definition, valid on any open set $U\subset \R^d$. We define $\Psi^m_\cl(U)\subset \mathcal D'(U\times U)$ to consist of those Schwartz kernels $A(x,y)$ for which
\begin{enumerate}[label=(\alph*)]\itemsep=0pt
	\item $A$ is smooth away from the diagonal $\Delta_U=\{(x,y)\in U\times U\mid x=y\}$;
	
	\item if $\varphi \in C_c^\infty(U)$, then $A_\varphi(x,y)=\varphi(x)A(x,y)\varphi(y)$ lies in $\Op\bigl(S^m_\cl\bigl(T^*\R^d\bigr)\bigr)$.
\end{enumerate}
For $U=\R^d$, this is equivalent to the preceding definition.
Writing $A_\varphi=\Op(a_\varphi)$, we equip~$\Psi^m_\cl(U)$ with the initial topology for the collection of maps
\begin{equation}\label{toppsido}
	\Psi^m_\cl(U)\xrightarrow{A\mapsto A\vert_{U\times U\backslash \Delta_U}} C^\infty(U\times U\backslash \Delta_U),\qquad \Psi^m_\cl(U)\xrightarrow{A\mapsto a_\varphi} S^m_\cl\bigl(T^*\R^d\bigr),
\end{equation}
where $\varphi$ runs through all elements of $C_c^\infty(U)$. In other words, a sequence of operators in $\Psi^m_\cl(U)$ converges if and only if their kernels converge in the $C^\infty$-topology away from the diagonal, and all local symbols converge in $S^m_\cl\bigl(T^*\R^d\bigr)$.
With this topology, $\Psi^m_\cl(U)$ is a Fr{\'e}chet space---as the completeness is crucial in this article, we include a brief sketch of this well-known fact.

\begin{Lemma}
	$\Psi_\cl^m(U)$ is a Fr{\'e}chet space.
\end{Lemma}

\begin{proof}[Sketch of the proof]
	Let $U=\bigcup_{k=1}^\infty V_k$ be a locally finite open cover with $\bar V_k\subset U$ ($k=1,2,\dots$). Define $W_0=U\times U\backslash \Delta_U$, $W_k=V_k\times V_k$ for $k\in \mathbb N$, as well as
	\[
		\tilde \Psi^m_\cl(U)\subset \bigg\{(A_0,a_1,a_2,\dots)\in C^\infty(W_0)\times \prod_{k\ge 1}S^m_\cl\bigl(T^*\R^d\bigr)\mid \eqref{dagger} \text{ holds}\bigg\},
	\]
	where the tuples are subjected to the constraint
	\begin{equation}\label{dagger}
		A_j=A_k \qquad\text{on}\  W_j\cap W_k \quad(j,k\ge 0),\qquad \text{where}\quad A_k = \Op(a_k)\quad(k\ge 1).
	\end{equation}
	If \eqref{dagger} is satisfied, then there is a Schwartz kernel $A\in \mathcal D'(U\times U)$ with $A=A_k$ on $W_k$ ($k\ge 0$)
	 and sending a given tuple to this kernel yields a map
$\tilde \Psi_\cl^m(U) \rightarrow \mathcal D'(U\times U)$.
	One checks that this map induces a homeomorphism $\tilde \Psi^m_\cl(U)\cong \Psi_\cl^m(U)$ and as the former space is a closed subspace of a countably infinite product of Fr{\'e}chet spaces, also $\Psi^m_\cl(U)$ is a Fr{\'e}chet space---we leave the details to the reader.
\end{proof}

On a manifold $X$ of dimension $d$, the space $\Psi^m_\cl(X)\subset \mathcal D'(X\times X)$ is defined by replacing condition (b) above with
\begin{itemize}\itemsep=0pt
	\item[(b$'$)] if $\psi\colon V\xrightarrow{\sim} U\subset \R^d$ is a local chart and $\varphi\in C_c^\infty(U)$, then $A_{\varphi,\psi}(x,y) = \varphi(x) A\bigl(\psi^{-1}(x),\allowbreak\psi^{-1}(y)\bigr) \varphi(y)$ lies in $\Op\bigl(S^m_\cl\bigl(T^*\R^d\bigr)\bigr)$.
\end{itemize}
Analogous considerations then give a natural Fr{\'e}chet space structure on $\Psi^m_\cl(X)$ that coincides with the one above when $X=U\subset \R^d$. By replacing classical symbols at each step with even symbols, we also obtain the closed subspace $\Psi^m_\ev(X)\subset\Psi^m_\cl(X)$ of {\it even pseudodifferential operators}.

\subsection{Fourier transform of homogeneous functions}
Suppose $Q\colon \R^d_x\times \bigl(\R^d_z\backslash 0\bigr) \times \R_h\rightarrow \C$ is smooth and satisfies, for some $\lambda\in \R$, the homogeneity property
$Q(x,tz,h)=t^\lambda Q(x,z,th)$, $ t>0$.
For later use, we wish to consider the inverse Fourier transform of $Q$ in the $z$-variable, restricted to frequencies $\omega \in \mathbb S^{d-1}$. Formally, this yields a~function
\[
	T Q(x,\omega,h)=\int_{\R^d} {\rm e}^{{\rm i}z\omega} Q(x,z,h) {\rm d}z,\qquad (x,\omega,h)\in \R^d\times \mathbb S^{d-1}\times \R.
\]
Due to the homogeneity, $Q$ may be recovered from its restriction to $|z|=1$ as $Q=E_\lambda(Q|_{|z|=1})$, where the extension operator $E_\lambda$ is defined by
\[
	E_\lambda q(x,z,h)=|z|^\lambda q(x,z/|z|,|z|h),\qquad q\in C^\infty\bigl(\R^d\times \mathbb S^{d-1}\times \R\bigr).
\]

\begin{Lemma}\label{tlambda}
	If $\lambda>-d$, then $T_\lambda=T E_\lambda$ defines a continuous linear map
	\[
	T_\lambda\colon\ C_c^\infty\bigl(\R^d\times \mathbb S^{d-1}\times \R\bigr) \rightarrow C^\infty\bigl(\R^d\times \mathbb S^{d-1}\times \R\bigr).
	\]	
\end{Lemma}

\begin{proof}[Sketch of the proof]
It is enough to show continuity $T_\lambda\colon C_K^\infty\to C^\infty$ for all compacts $K\subset \R^d\times \mathbb S^{d-1} \times \R$, where $C^\infty_K$ stands for smooth functions supported in $K$. Showing continuity into~$\mathcal S'\bigl(\R^{2d+1}\bigr)$ is straightforward and, by virtue of Lemma \ref{upgrade}, this can be upgraded to continuity into a $C^\infty$-space, if one shows that $T_\lambda Q$ is smooth on $\R^d\times \bigl(\R^d\backslash 0\bigr) \times \R$. The latter assertion is standard if $Q$ is independent of the parameters $x$ and $h$ (see, e.g., \cite[Section~7.4]{Ste93}) and the arguments carry over to the present setting.
\end{proof}

\subsection{A continuous family of pseudodifferential operators}\label{ctspsido}
 Let $\Psi^m_\cl(M^\circ)$ denote the space of classical pseudodifferential operators of order $m\in \R$ on the interior $M^\circ$ of $M$. It is well known (see, e.g., \cite[Theorem~8.1.1]{PSU23}) that
$
N_0^g\in \Psi^{-1}_\cl(M^\circ)$,
provided~$g$ is simple. Here we show that {\it as classical pseudodifferential operator}, it depends continuously on the metric $g$.

\begin{Proposition}\label{claim2}
The following map is continuous:
$
g\mapsto N_0^g$, $ \Ms \rightarrow \Psi^{-1}_\cl(M^\circ)$.
Here $\Psi^{-1}_\cl(M^\circ)$ carries its natural Fr{\'e}chet space topology.
\end{Proposition}

In the setting of closed Anosov manifolds, a similar continuity statement was proved in \cite[Proposition~4.1]{GKL22}. Near the diagonal of $M^\circ \times M^\circ$ their analysis essentially carries over to the present setting and yields continuity into $\Psi^{-1}(M^\circ)$, that is, into the weaker topology of non-classical pseudodifferential operators. Keeping this in mind, we will keep the following discussion brief and focus on classicality.

\begin{proof}{Proof of Proposition~\ref{claim2}} Simple manifolds are diffeomorphic to a compact ball in $\R^d$, so without loss of generality we assume that $M=\overline U$, where $U=\big\{x\in \R^d \mid |x|<1\big\}$, with metric $g=(g_{ij})$ given in terms of Euclidean coordinates. By \cite[Lemma~8.1.10]{PSU23}, the Schwartz kernel of $N^g_0$ equals
\[
	N^g_0(x,y) = \frac{1}{d_g(x,y)^{d-1}} \times 2 J_g(x,y) \sqrt{g(y)},\qquad (x,y)\in U\times U,
\]
where $d_g$ is the geodesic distance and $J_g$ is a Jacobian factor. They depend continuously on the metric in the sense that the following map is continuous:
\begin{equation}\label{ctsexc}
	g\mapsto \bigl(d_g^2,J_g\bigr), \qquad \Ms\rightarrow C^\infty(U\times U)^2.
\end{equation}
Proving this fact requires revisiting some fairly standard constructions in Riemannian ge\-ome\-try---this tedious but elementary exercise is omitted. Further, we have the following.
\begin{Lemma}
	For every $g\in \Ms$, there exists a smooth function $G^g\colon U\times U\rightarrow \R^{d\times d}$, taking values in symmetric matrices, such that
	\begin{enumerate}[label=\rm(\roman*)]\itemsep=0pt
		\item $G^g(x,x)=g(x)$ for $x\in U$;
		\item $d_g(x,y)^2=(x-y)^jG^g_{jk}(x,y) (x-y)^k$ for $(x,y)\in U\times U$;
		\item $g\mapsto G^g,$ $\Ms\rightarrow C^\infty\bigl(U\times U, \R^{d\times d}\bigr)$ is continuous.
	\end{enumerate}
\end{Lemma}
\begin{proof}
	Let $H_{ij}^g(x,y)=\partial_{y_i}\partial_{y_j} d^2_g(x,y)$, then Taylor expanding $d_g^2(x,\cdot)$ about $y=x$ yields
	\[
		d_g^2(x,y)= \int_{0}^1 \frac{(1-t)}{2} H^g_{ij}(x,x+t(y-x)){\rm d}t\cdot (x-y)^i(x-y)^j
	\]
	and defining $G^g_{ij}$ as this integral expression, the continuity claim (iii) follows directly from \eqref{ctsexc}. The other two properties are well known (see, e.g., \cite[Lemma 8.1.12]{PSU23}).
\end{proof}

We now proceed with the proof that $N^g_0 \in \Psi^{-1}_\cl(U)$ depends continuously on $g\in \Ms$, using the characterisation of the topology in \eqref{toppsido}.
Clearly, the kernel $N^g_0$ is smooth away from the diagonal and due to \eqref{ctsexc} the map
\[
	g\mapsto N^g_0\vert_{U\times U\backslash \Delta_U},\qquad \Ms\rightarrow C^\infty(U\times U\backslash \Delta_U)
\]
is continuous. Next, if $\varphi\in C_c^\infty(U)$, then following \cite[Chapter 8.1]{PSU23}
\[
	N^g_{0,\varphi}(x,y):=\varphi(x)N^g_0(x,y)\varphi(y) = +
	\Op(a^g_\varphi),\qquad a_\varphi^g(x,\xi)=\int_{\R^d} {\rm e}^{{\rm i}z\xi} N^g_{0,\varphi}(x,x-z) {\rm d}z.
\]
As the integrand in the definition of $a^g_\varphi$ is supported in the compact set $|z|\le 2$ and depends continuously on the metric, one can derive continuity of the following map:
\[
	g\mapsto a^g_\varphi,\qquad \Ms\rightarrow C^\infty\bigl(\R^d\times \R^d\bigr).
\]
Classicality is the assertion that the symbol at infinity, defined by
\[
\hat a_\varphi^g(x,\omega,h)=\int_{\R^d} \frac{{\rm e}^{{\rm i}z\omega} q_\varphi^g(x,x-hz) {\rm d}z}{\big[z^iG_{ij}(x,x-hz)z^j\big]^{\frac{d-1}2}}, \qquad q_\varphi^g(x,y) = 2 J_g(x,y)\varphi(x) \sqrt{g(y)} \varphi(y),
\]
is smooth up to $h=0$. This, however, together with the continuity of the map
\[
	g\mapsto \hat a_\varphi^g,\qquad \mathbb M_\mathrm{s} \to C^\infty\bigl(\R^d\times \mathbb S^{d-1} \times [0,\infty)\bigr),
\]
follows from Lemma \ref{tlambda}. This completes the proof of Proposition~\ref{claim2}.
\end{proof}

\subsection*{Acknowledgements}

JB would like to thank {Gunther Uhlmann} for inviting him to the University of Washington, where part of this research was carried out---this also enabled visiting UC Santa Cruz, where he was welcomed with much appreciated hospitality. Further thanks go to {Hadrian Quan} for several useful conversations about Hermitian structures on twistor space and to the referees for their comments with suggestions and improvements. FM was partially supported by NSF-CAREER grant DMS-1943580 and GPP was partially supported by NSF grant DMS-2347868.

\pdfbookmark[1]{References}{ref}
\LastPageEnding

\end{document}